\documentclass[reqno]{amsart}
\usepackage{amssymb,mathrsfs, amsmath,amssymb,amsfonts, amsthm, dsfont, tikz}
\usepackage{hyperref}
\usepackage[inner=1.0in,outer=1.0in,bottom=1.0in, top=1.0in]{geometry}

\usepackage{bbm}%\\mathbbm{1}

\usepackage{graphicx}
\usepackage{subcaption}

\usepackage{transparent}%\Inkscape transparancy use

\usepackage{tikz-cd,tikz}
\usepackage{todonotes}

\usepackage{comment}
\usepackage{mdwlist}%\suspend\resume{enumerate}

\newtheorem{theorem}{Theorem}[section]

\newtheorem{proposition}[theorem]{Proposition}

\newtheorem{definition}[theorem]{Definition}
\newtheorem{lemma}[theorem]{Lemma}

\theoremstyle{definition}
\newtheorem{remark}[theorem]{Remark}

\numberwithin{equation}{section}

\DeclareMathOperator*{\arsinh}{arsinh}
\DeclareMathOperator*{\arcsec}{arcsec}
\DeclareMathOperator*{\Arg}{Arg}
\newcommand{\abs}[1]{\left\lvert #1 \right\rvert}

\begin{document}

\title{Sharp effective equidistribution of closed geodesics in compact hyperbolic surfaces}

\author{Junehyuk Jung}
\address{Department of Mathematics, Brown University, Providence, RI 02912}

\email{junehyuk\_jung@brown.edu}

\author{Insung Park}
\address{Institute for Mathematical Sciences, Stony Brook University, Stony Brook, NY 11794-3660}

\email{insung.park@stonybrook.edu}

\author{Peter Zenz}
\address{Department of Mathematics, Brown University, Providence, RI 02912}

\email{peter\_zenz@brown.edu}

\thanks{}

\begin{abstract}
We prove a sharp effective equidistribution theorem for closed geodesics on a compact hyperbolic surface. This is achieved by carefully analyzing the Selberg trace formula twisted by a Maass eigenfunction, which we refer to as Zelditch's trace formula, following Zelditch's observation that the geometric side can be expressed as a sum of period integrals of eigenfunctions on closed geodesics. The main technical input is a sharp upper bound for the matrix coefficients $\langle \phi, \varphi^2 \rangle$ uniformly in $t_\phi\to \infty$ and $t_\varphi \to \infty$, where $t_\phi$ and $t_\varphi$ are the eigenparameters corresponding to Maass forms $\phi$ and $\varphi$ on $X$.  
\end{abstract}

\maketitle
%\tableofcontents
\section{Introduction}
Let $X$ be a compact hyperbolic surface. For any closed geodesic $\gamma$ on $X$, let $\gamma_0$ denote its corresponding primitive closed geodesic. Bowen's theorem \cite{Bowen_Equidist} implies the equidistribution of the closed geodesics on $X$, i.e., for $f\in C^\infty (X)$, we have
\[
    \frac{\sum_{l(\gamma)<T} \int_{\gamma_0} f ds}{\sum_{l(\gamma)<T} \int_{\gamma_0} 1 ds} = \frac{1}{\mathrm{vol}(X)}\int_X f d\mu + o(1)
\]
as $T \to \infty$, where $ds$ is the arc length differential and $\mu$ is the hyperbolic area measure.  In this article, we prove a sharp effective equidistribution theorem. Let $0=\lambda_0<\lambda_1\le \dots$ denote the eigenvalues of the Laplace--Beltrami operator $-\Delta$ on $X$. Define $t_i\in [0,\infty)\cup i[0,\frac12]$ by $\lambda_i=\frac14+t_i^2$.

\begin{theorem}\label{theorem:main}
Let $X$ be a closed hyperbolic surface and let $f$ be a smooth function on $X$. Let $\lambda_1$ denote the smallest positive eigenvalue of the Laplace--Beltrami operator on $X$.
\begin{enumerate}
\item For every $\epsilon>0$, we have
\[
    \frac{\sum_{l(\gamma)<T} \int_{\gamma_0} f ds}{\sum_{l(\gamma)<T} \int_{\gamma_0} 1 ds} =  \frac{\int_X f d\mu}{{\rm Vol}(X)}+
    \|f\|_{C^{D(\epsilon)}}
    O_{X,\epsilon}\left(\max \left\{e^{\left(-\frac{1}{2}+\Im(t_1)\right)T},e^{(-\frac{1}{4}+\epsilon)T}\right\}\right),
\]
where $D(\epsilon)\sim 20/\epsilon$ as $\epsilon\to 0$ and $\|\cdot \|_{C^r}$ denotes the $C^r$-norm on $C^{\infty}(X)$.

\item Suppose that $\lambda_1 < 3/16$ and $\lambda_1$ is a simple eigenvalue.\footnote{The first positive eigenvalue $\lambda_1$ is simple for a generic smooth Riemannian metric \cite{Uhlenbeck_GenericProp}. Its genericity within the subspace of metrics of constant curvature $-1$ is unknown.} Then there exists $f\in C^\infty (X)$ with $\int_X fd\mu=0$ such that
\[
		\left|\frac{\sum_{l(\gamma)<T} \int_{\gamma_0} f ds}{\sum_{l(\gamma)<T} \int_{\gamma_0} 1 ds}\right| \gg e^{\left(-\frac{1}{2}+\Im(t_1)\right)T}
\]
as $T \to \infty$.
\end{enumerate}
\end{theorem}

We begin by discussing related work in the literature.

%Closed geodesics can be considered as periodic orbits of the geodesic flow. 

%The equidistribution of periodic orbits has been extensively studied for even more general dynamical systems, e.g., \cite{Margulis_AnosovSystemBOOK, Bowen_Equidist, ParrPoll_ZetaFtn}.

% Let us review previous work on geodesic counting and equidistribution, followed by an outline of the proof of Theorem \ref{theorem:main}.

\subsection{Background}
Let $X$ be a compact hyperbolic surface. For any $T>0$, let $\pi_X(T)$ be the number of primitive closed geodesics of length less than $T$.

Huber obtained the first such result in 1959 \cite{Huber_PrimeGeodSurf1} using Selberg's trace formula which relates the lengths of closed geodesics to the spectrum of the Laplace--Beltrami operator on $X$. Let us denote by $0=\lambda_0 < \lambda_1 \leq \lambda_2 \leq \ldots$ the complete set of eigenvalues of the Laplace--Beltrami operator on $X$. Because Selberg's trace formula dualizes a summation over the closed geodesics to a summation over $\{\lambda_j\}$, one may analyze $\pi_X(T)$ in terms of the spectral information of the given manifold. Huber carried out the details in a series of papers \cite{Huber_PrimeGeodSurf1,Huber_PrimeGeodSurf2,Huber_PrimeGeodSurf2(Improved)} and proved
\begin{equation}\label{eq:PGT}
    \pi_X(T) = \mathrm{li} \left(e^T\right) + \sum_{j:0<\lambda_j<3/16} \mathrm{li} \left( e^{\left(\frac{1}{2} + \Im(t_j)\right)T}\right) + O_X\left(T^{-1/2}e^{\frac{3}{4}T}\right),
\end{equation}
where $\mathrm{li}$ is the logarithmic integral function
\[
    \mathrm{li}(x) = \frac{x}{\log x} +\frac{x}{(\log x)^2}+\frac{2x}{(\log x)^3}+\cdots+\frac{(N-1)!x}{(\log x)^N}+O_N\left(\frac{x}{(\log x)^{N+1}}\right) .
\]

The problem of interest in this paper is to understand how the set of closed geodesics of length less than $T$ is distributed in $X$ as $T \to \infty$. To be precise, given a closed geodesic $\gamma$, let $\gamma_0$ be the corresponding primitive closed geodesic, and define the measure $\mu_\gamma: C_0^\infty (X) \to \mathbb{R}$ to be
\[
    \mu_\gamma: \phi\mapsto \int_{\gamma_0} \phi(s) ds,
\]
where $ds$ is the arc length element of $\gamma_0$. In particular, $\mu_\gamma (1) = l(\gamma_0)$, where $l(\cdot)$ is the length of the given piecewise smooth curve. The first result concerning the distribution of closed geodesics appears in Margulis' thesis \cite{Margulis_AnosovSystemBOOK}, which states that for any fixed $\epsilon>0$ and for any fixed $f \in C_0^\infty (X)$, we have
\[
    \frac{\sum_{T-\epsilon <l(\gamma) < T+\epsilon} \mu_\gamma(f)}{\sum_{T-\epsilon <l(\gamma) < T+\epsilon} \mu_\gamma(1)} = \frac{1}{\mathrm{vol}(X)}\left(\int_X f d\mu \right) +o(1),
\]
as $T \to \infty$, where $d\mu$ is the hyperbolic volume form on $X$. Note that the summation runs over all closed geodesics which are not necessarily primitive. Bowen's Theorem often appears in a weaker form in the literature:
\begin{equation}\label{eq:Bowen}
\frac{\sum_{l(\gamma) < T} \mu_\gamma(f)}{\sum_{l(\gamma) < T} \mu_\gamma(1)} = \frac{1}{\mathrm{vol}(X)} \left(\int_X f d\mu\right) +o(1),
\end{equation}
as $T \to \infty$, where the summation is taken over \textit{all} closed geodesics.

\begin{remark}
In \cite{Bowen_Equidist}, Bowen studied the closed orbits of Anosov flows on compact Riemannian manifolds. See also \cite{ParrPoll_ZetaFtn}, where these results were generalized using thermodynamic formalism.
\end{remark}

Bowen's theorem for hyperbolic surfaces of finite area was later revisited by Zelditch \cite{Zelditch_EquidistCompact} and \cite{Zelditch_EquidistNonCompact} using spectral geometric techniques. The key idea of Zelditch is to investigate the trace of a composition between a pseudo-differential operator $\mathrm{Op}(a)$ and an invariant Hilbert--Schmidt operator in the trace class. Note that Selberg's trace formula computes the trace of an invariant  Hilbert--Schmidt operator in the trace class, which essentially leads to a summation involving only the length of the closed geodesics.\footnote{Selberg's trace formula for non-spherical representations captures holonomy of the closed geodesics as well. See \cite{DeverMilic_AmbientPrimeGeodesic}.} Because of this limitation, it is not plausible to study the expression \eqref{eq:Bowen} using Selberg's trace formula as is. What Zelditch observed is that the trace of the invariant operator when composed with a pseudo-differential operator $\mathrm{Op}(a)$ becomes a summation involving
\[
    \int_{\gamma_0} a(s) ds
\]
and the lengths of the closed geodesics, \textit{if} $a$ is chosen to be a Maass form. Such an expression is suitable for studying \eqref{eq:Bowen}, and so by carefully choosing the test function Zelditch proved that for any $A>0$ one has
\[
    \frac{\sum_{l(\gamma) < T} \mu_\gamma(f)}{\sum_{l(\gamma) < T} \mu_\gamma(1)} = \frac{1}{\mathrm{vol}(X)}\int_X f d\mu + O_{X,A,f}\left(T^{-A}\right)
\]
as $T \to \infty$. 

A related generalized Selberg trace formula was developed by Bir\'o \cite{Biro_Gen_STF}, who evaluates the diagonal of an 
automorphic kernel against a fixed Laplace eigenfunction.
Bir\'o subsequently applied this formula to local averages
in the hyperbolic circle problem \cite{Biro_LocalAverage}.

Margulis--Mohammadi--Oh \cite{MMO_Holonomy} obtained an equidistribution result with an exponential error $e^{-\epsilon T}$ for geometrically finite hyperbolic manifolds with a suitable lower bound on the critical exponent, where the exponent $\epsilon$ is implicit.

The equidistribution of holonomies of closed geodesics with an exponentially decaying error term was established by Sarnak--Wakayama \cite{SarnakWakayama_EquidistHol} for the finite volume case using Selberg's Trace Formula and by Margulis--Mohammadi--Oh \cite{MMO_Holonomy} for the geometrically finite case using the exponential mixing of the frame flow.

\subsection{Revisiting Zelditch's approach and structure of the proof}
Recall from \cite{Zelditch_EquidistCompact} that for any non-constant eigenfunction $\phi_k$, 
\begin{equation}\label{eqn:Equidist_k}
    \sum_{l(\gamma) < T} \mu_\gamma(\phi_k) = \sum_{j:\lambda_j<3/16} \gamma_{k,j}\langle\phi_k,\phi_j^2\rangle e^{(\frac12+\Im(t_j))T}+O_{X,k}(Te^{\frac34T}),
\end{equation}
where $\lambda_i=\tfrac14+t_i^2$ and
\[
    \gamma_{k,j}= \frac{2^{1-2\Im(t_j)}\sqrt{\pi}\,\Gamma\left(2\Im(t_j)\right)}{\Gamma\left(\frac14+\Im(t_j)+\frac12i t_k\right)\Gamma\left(\frac14+\Im(t_j)-\frac12i t_k \right)\left(\frac12+\Im(t_j)\right)}
\]
for $j$'s with $\lambda_j<3/16$.
Note that this already implies our main theorem (Theorem \ref{theorem:main}) when the test function $f$ is a finite linear combination of eigenfunctions. Proving the theorem for all smooth test functions requires resolving two major challenges:

\begin{enumerate}
    \item[(i)] Establishing a sharp uniform upper bound for the matrix coefficients $\langle\phi_k,\phi_j^2\rangle$ in terms of $t_k$ and $t_j$.
    \item[(ii)] Analyzing the explicit dependence of the remainder term $O_{X,k}(Te^{\frac34T})$ on $t_k$.
\end{enumerate}

The main contribution of the article is resolving these two challenges and we present the proof as follows.

Firstly, in Section \ref{sec:Zelditch_Trace_Formula}, for the sake of completeness and a self-contained discussion, we formulate Zelditch's trace formula (Theorem \ref{thm:ZTF}) \cite{Zelditch_EquidistCompact,Zelditch_EquidistNonCompact}. 

Regarding the first difficulty (i), we prove in Section \ref{sec:Triple_Product} a uniform estimate for the matrix coefficient $\langle \phi_k,\phi_j^2\rangle$ in the following form.
\begin{theorem}\label{thm:TripleProduct}
Let $Y$ be a compact hyperbolic surface. Let $\varphi$ and $\phi$ be $L^2$-normalized eigenfunctions of the Laplace--Beltrami operator on $Y$ with eigenvalues $\frac{1}{4}+\tau^2$ and $\frac{1}{4}+t^2$, respectively. Then
\[
\int_Y \varphi(z)|\phi(z)|^2\,d\mu(z)
\ll_Y
(1+|\tau|)\bigl(1+\log(1+|\tau|)\bigr)
\exp\left(
-\frac{\pi}{4}
\left( |2t+\tau|+|2t-\tau|-4|t| \right)
\right),
\]
uniformly in $\phi$ and $\varphi$.
\end{theorem}
We achieve this by following and generalizing the strategy of \cite{Sarnak_Int_of_Prod_of_Eftns}, where Theorem \ref{thm:TripleProduct} for fixed $\phi$ is obtained. It is necessary to obtain a uniform estimate for certain associated Legendre functions, known as conical functions. To this end, we make use of a series of works by Dunster \cite{Dunster_Concial}, \cite{Dunster_Prep}, \cite{Dunster_K_Bessel_by_Exp}, partly in collaboration with Boyd \cite{BoydDunster_EstimateLegendra}, Gil, and Segura \cite{DunsterGilSegura_BesselbyAiry}, on the Liouville--Green approximation of conical functions by Bessel functions. In Appendix \ref{AppendixConical}, we carry out careful estimates of the error terms and verify that this method yields the desired estimate for the conical functions.

\begin{remark}
When a non-constant $\varphi$ is fixed, improving the trivial bound 
    \[
\int_Y \varphi(z)|\phi(z)|^2\,d\mu(z) \ll_{\varphi} 1
    \]
is related to the Quantum Unique Ergodicity conjecture by Rudnick and Sarnak \cite{Rudnick_Sarnak}. With additional arithmetic assumptions on $Y$ and $\phi$, Lindenstrauss \cite{Lindenstrauss} and Soundararajan \cite{Soundararajan_QUE} proved the estimate $o_\varphi(1)$, which is referred to as the Arithmetic Quantum Unique Ergodicity theorem.

When $\phi$ is fixed, the problem was investigated in various contexts \cite{Bernstein_Reznikov},  \cite{KroSta_HoloExtRep}, \cite{Sarnak_Int_of_Prod_of_Eftns} concerning the correct exponential decay in $\tau$. 
\end{remark}

Regarding the second issue (ii), we explicitly identify that certain residues on the spectral side of the trace formula grow exponentially fast in $t_k$. Since there are no such exponential terms on the geometric side, they must, a priori, cancel each other. Since these cancellations are not apparent on the spectral side, we transfer the problematic terms to the geometric side and verify that the required cancellations indeed occur. This will be addressed in Section \ref{sec:Equidistrib}.

\subsection*{Notations}
\begin{itemize}
    \item For positive functions $f,g$,
        \begin{itemize}
            \item $f(x)\ll g(x)$ if there exists $C>0$ such that $f(x)< Cg(x)$ uniformly for all $x$.
            \item $f(x)\asymp g(x)$ if $f(x)\ll g(x)$ and $g(x)\ll f(x)$.
            \item $f(x)\sim g(x)$ as $x\to 0$ (resp.\@ $x\to \infty$) if $f(x)/g(x)\to 1$ as $x\to 0$ (resp.\@ $x\to \infty$).
        \end{itemize}
    \item Since we fix a compact hyperbolic surface $X$, we sometimes omit the dependence on $X$ from the notation, writing simply $\ll$ instead of $\ll_X$, unless this dependence needs to be emphasized.
    \item For $s\in \mathbb{C}$ and $n\in \mathbb{Z}_{\ge0}$,
    \[
        [s]_n:=s(s+1)\cdots (s+(n-1)).
    \]
    \item For any $z\in \mathbb{C}$, we write $z=\Re(z)+i\Im(z)$.
    \item Although our estimates are uniform in both spectral parameters $t_j$ and $t_k$, these parameters play completely different roles. To make the notations easier to read, we also write
    \[
        \phi_k=\varphi
        \qquad\text{and}\qquad
        t_k=\tau.
    \]
\end{itemize}

\subsection*{Acknowledgments} J.J. thanks Peter Sarnak, Federico Pasqualotto, Peter Humphries, Hee Oh, and Sug Woo Shin for helpful discussions. I.P. thanks Misha Lyubich for helpful comments and the Institute for Mathematical Sciences at Stony Brook University for travel support during this project. P.Z. is grateful to Mark Dunster and Kannan Soundararajan for insightful and helpful discussions, and to the Stanford Mathematics Department for its hospitality during part of this work.

\subsection*{AI disclosure} The authors used generative AI tools for limited assistance with language editing, rewriting, and routine computations. All mathematical ideas, arguments, and results are the authors' own. Any AI-assisted content was independently verified by the authors, who take full responsibility for the content of the paper.

\section{Preliminary}
\subsection{Hyperbolic plane}
Let $\mathbb{H} = \{z=x+iy~:~ y>0\}$ be the upper half space endowed with the hyperbolic Riemannian metric $g=g_{ij} = y^{-2}\delta_{ij}$.

The group of orientation-preserving isometries of $\mathbb{H}$ is $ \mathrm{Isom}^+ (\mathbb{H})\cong {\rm PSL}_2(\mathbb{R})$ where we realize the Iwasawa decomposition $\mathrm{PSL}_2(\mathbb{R}) \cong NAK$ by
\begin{align*}
        A & = \{ a_t \colon z \mapsto e^t z~:~ t\in \mathbb{R}\}\\
        N & = \{n_{u}\colon z \mapsto z+u ~:~ u\in \mathbb{R}\}\\
        K & = \mathrm{Stab}(i).
\end{align*}

\subsection{AN-coordinate system}
Observe that $AN=NA$. We introduce an $AN$-coordinate system for $\mathbb{H}$ with $N\cong \mathbb{R}$. For $u,t\in \mathbb{R}$, consider the map to $\mathbb{H}$ given by
\[
(u,t) \mapsto e^t (u+i).
\]
Then the hyperbolic metric $g_{ij}=y^{-2}\delta_{ij}$ in these coordinates is given by
\[
g_{ij}' = g_{kl} \frac{\partial x_k}{\partial u_i}\frac{\partial x_l}{\partial u_j}= y^{-2} \frac{\partial x_k}{\partial u_i}\frac{\partial x_k}{\partial u_j}=(I+U)(I+U)^{T}
\]
where
\[
U=\begin{pmatrix}
 0 & 0\\
 u  & 0\\
\end{pmatrix}.
\]
Here we define $x_1=x$, $x_2=y$, $u_1=u$ and $u_2=t$ for notational convenience. In particular, the volume form in this coordinate system is given by $dudt$.

\subsection{Orthonormal eigenbasis}
We identify eigenfunctions of the Laplace--Beltrami operator on a compact hyperbolic surface $Y= \Gamma \backslash \mathbb{H}$ with the eigenfunctions of $\Delta_{\mathbb{H}}$ on $\mathbb{H}$ that are $\Gamma$-invariant, where $\Gamma$ is a torsion-free cocompact lattice of ${\rm PSL}(2,\mathbb{R})$. For instance, an integral over $\Gamma \backslash \mathbb{H}$ may be regarded as an integral over a fundamental domain of $\Gamma$. We let $\{\phi_j\}_{j=0}^\infty$ be a real-valued orthonormal Laplacian eigenbasis of $L^2(\Gamma \backslash \mathbb{H})$, i.e.,
\begin{itemize}
\item $\{\phi_j\}_{j=0}^\infty$ spans $L^2(\Gamma \backslash \mathbb{H})$,
\item $\langle\phi_j,\phi_k\rangle = \delta_{j,k}$,
\item $-\Delta_\mathbb{H} \phi_j = \lambda_j \phi_j$,
\end{itemize}
where we assume $0=\lambda_0 < \lambda_1 \leq \lambda_2 \leq \ldots$. Let $\lambda_j = \frac{1}{4}+t_j^2$ with $t_j\in [0,\infty)\cup i(0,1/2]$ and $t_0=i/2$. Any $f \in L^2(Y)$ satisfies
\[
f(z) = \sum_{j=0}^\infty \langle f,\phi_j\rangle \phi_j(z)
\]
in the sense of $L^2$. Note that if $f$ is assumed to be smooth, the convergence is pointwise and uniform.

\subsection{Mellin transform}
%There are a number of integral transforms that one may encounter when analyzing Zelditch's trace formula. We review some of the basic integral transforms here.

%We denote the Fourier transform of $f:\mathbb{R} \to \mathbb{C}$ by $\mathcal{F}f$ so that
%\[
%    \mathcal{F}f (t)=\int_{-\infty}^\infty f(r) e^{-irt} dr.
%\]
%The inverse Fourier transform is given by
%\[
%    \mathcal{F}^{-1}:f \mapsto  \frac{1}{2\pi} \int_{-\infty}^\infty f(r) e^{irt} dr.
%\]
Denote by
\[
    M: \phi \mapsto \int_0^\infty \phi(y) y^s \frac{dy}{y}
\]
the Mellin transform of $\phi:[0,\infty) \to \mathbb{C}$, whose inverse transform is given by
\[
M_\sigma^{-1}: F \mapsto \frac{1}{2\pi i }\int_{(\sigma)} F(s) y^{-s} ds
\]
for some $\sigma \in \mathbb{R}$, where $(\sigma)$ is the vertical contour $\{z\in \mathbb{C}~:~ \mathrm{Re}(z) =\sigma\}$ pointing upward.

\begin{proposition}\label{prop:melplan}
    Assume that $M\phi(s)$ is absolutely convergent for $\sigma_1<\mathrm{Re}(s)<\sigma_2$. For $\sigma\in \mathbb{R}$ such that $\sigma_1<\sigma<\sigma_2$, if
\[
    \phi(y)y^{\sigma}\in L^2((0,\infty),d^\times y),
\]
then we have
\[
    \int_0^\infty |\phi(y)|^2 y^{2\sigma} \frac{dy}{y}
=\frac{1}{2\pi}\int_{-\infty}^\infty \left|M\phi(\sigma+it)\right|^2 dt.
\]
\end{proposition}
\subsection{Gamma function asymptotics and identities}

For $\epsilon>0$, suppose $z\in \mathbb{C}$ satisfies
\[
    \arg(z)\in(-\pi+\epsilon,\pi-\epsilon).
\]
Then, it follows from Stirling's formula that for all $|z|\gg1$, we have
    \begin{equation}\label{eqn:Stirling}
        |\Gamma(z)|\asymp_{\epsilon}\sqrt{2\pi}|z|^{\Re(z)-\frac{1}{2}}\exp(-\Im(z)\arg(z)-\Re(z)).
    \end{equation}
As a consequence, we have, for $x,y\in \mathbb{R}$,
\begin{equation}\label{eqn:Stirling_vertical}
    |\Gamma(x+iy)|\asymp_x e^{-\pi|y|/2}|y|^{x-1/2}
    \quad
    \text{as }
    |y|\to \infty,
\end{equation}
and
\begin{equation}\label{eqn:Gamma_ratio}
    \left|\frac{\Gamma(z+a)}{\Gamma(z+b)}\right|\asymp_{a,b}\left|z^{a-b}\right|
    \quad
    \text{as }
    |z|\to \infty.
\end{equation}
    
When the condition $\arg(z)\in (-\pi+\epsilon,\pi-\epsilon)$ is not satisfied, we use Euler's reflection formula
\begin{equation}\label{eqn:reflection}
    \Gamma(z)=\frac\pi{\Gamma(1-z)\,\sin \pi z}.
\end{equation}

We also record Legendre's duplication formula
\begin{equation}\label{eqn:duplication}
    \Gamma(z)\Gamma(z+\tfrac12)=2^{1-2z}\sqrt{\pi}\,\Gamma(2z).
\end{equation}

% \begin{lemma}\label{lem:Gamma_ratio_real_para}
%     For $a,b>0$, we have
%     \[
%         \left|\frac{\Gamma(a+ib)}{\Gamma(a)}\right|
%         \le
%         \begin{cases}
%             \exp(-b^2/4a)& a>b\\
%             \exp(-b/4)&a\le b.
%         \end{cases}
%     \]
% \end{lemma}
% \begin{proof}
%     By \cite[8.326-(2)]{GR_Table}, we have
%     \begin{align*}
%         \left|\frac{\Gamma(a+ib)}{\Gamma(a)}\right|^2
%         =\prod_{m=0}^\infty \left|1+\frac{b^2}{(a+m)^2}\right|^{-1}.
%     \end{align*}
%     It follows that
%     \begin{align*}
%         -2\log \left|\frac{\Gamma(a+ib)}{\Gamma(a)}\right|
%         &=\sum_{m\ge0} \log\left(1+\frac{b^2}{(a+m)^2}\right)\\
%         &\ge \int_0^\infty \log\left(1+\frac{b^2}{(a+x)^2}\right)dx\\
%         &=2b\arctan\frac{b}{a}-a\log\left(1+\frac{b^2}{a^2}\right)\\
%         &=2a\,g(b/a),
%     \end{align*}
%     where $g(t):=t\arctan t -\frac12\log(1+t^2)$. It is easy to show
%     \[
%         g(t)\ge
%         \begin{cases}
%             \frac{1}{4}t^2 & t\in(0,1)\\
%             \frac{1}{4}t & t\in [1,\infty).
%         \end{cases}
%     \]
%     Then the desired estimate follows.
% \end{proof}

\subsection{Weyl's law}

\begin{lemma}[Weyl's law]\label{lem:Weyl Law}
    For any closed Riemannian manifold $X$ of dimension $n$,
    \[
        N(T)=\frac{\omega_n}{(2\pi)^n} {\rm Vol}(X)\, T^{n/2}+O_X\left(T^{\frac{n-1}{2}}\right),
    \]
    where $N(T):=\#\{j:\lambda_j\le T\}$ counts eigenvalues
    with multiplicity, and $\omega_n$ is the volume of the $n$-dimensional Euclidean unit ball.
\end{lemma}
Weyl's law was first proven for bounded domains in $\mathbb{R}^2$ \cite{Weyl_1911}. See \cite{Ivrii_100YrsWeylLaw} for a comprehensive exposition of Weyl's law.

\begin{lemma}\label{lem:sum(1+t)^delta} For any closed Riemannian surface $Y$ and $\delta>0$, we have
\[
    \sum_{|t_j|\le T}(1+|t_j|)^{-\delta}
    =
    \begin{cases}
        O_{Y,\delta}(1)&\delta> 2\\
        O_Y(\log T) &\delta=2\\
        O_{Y,\delta}(T^{2-\delta})&\delta< 2
    \end{cases}
\]
as $T\to \infty$, where $\lambda_j=\frac{1}{4}+t_j^2$. Moreover, when $\delta> 2$, we have
\[
    \sum_{t_j> T}(1+|t_j|)^{-\delta}
    =
            O_{Y,\delta}(T^{2-\delta}).
\]
\end{lemma}
\begin{proof}
Suppose $T>1$ is sufficiently large. Since $\sqrt{x}\le 1+\sqrt{x-1/4}$ for all $x\ge 1$, we have
\[
    \sum_{\lambda_j<T}(1+|t_j|)^{-\delta}
    \le O_Y(1)+\sum_{1<\lambda_j\le T}\lambda_j^{-\delta/2}.
\]
Hence, for $\delta\neq 2$, we have
\begin{align*}
    \sum_{1<\lambda_j\le T}\lambda_j^{-\delta/2}
    & = \int_1^T x^{-\delta/2} dN(x)\\
    & =\frac{N(T)}{T^{\delta/2}}-N(1)+\frac{\delta}{2} \int_1^T N(x) \frac{1}{x^{\delta/2+1}}dx \\
    & \ll_{Y,\delta} 1+T^{1-\delta/2},
\end{align*}
where $\int dN$ is the Riemann–Stieltjes integral with respect to $N$. In the last line, we use Lemma~\ref{lem:Weyl Law}.

When $\delta=2$, by a similar argument, we obtain
\[
    \sum_{1<\lambda_j\le T}\lambda_j^{-\delta/2}
    \ll_{Y,\delta} \log T.
\]

Similarly, when $\delta>2$, we obtain
\[
    \sum_{\lambda_j> T}(1+|t_j|)^{-\delta}\le\sum_{\lambda_j> T}\lambda_j^{-\delta/2}
     = \int^\infty_T x^{-\delta/2} dN(x)=O_{Y,\delta}(T^{1-\delta/2}).
\]

Since $\lambda_j=1/4+t_j^2$, by replacing $T$ with $T^2$, we obtain the lemma.
\end{proof}

\subsection{Weighted geodesic length sum}
Let $X$ be a compact hyperbolic surface. For any closed geodesic $\gamma$, denote by $\gamma_0$ its primitive closed geodesic. There is an asymptotic formula of the geodesic length sum: For all $T\gg1$, 
\begin{equation}\label{eqn:PrimLengthSum}
    \sum_{l(\gamma)<T} l(\gamma_0) =e^T + \sum_{\lambda_j\in(0,3/16)}\frac{e^{(1/2+|\Im(t_j)|)T}}{1/2+|\Im(t_j)|}+ O_{X}\left(\sqrt{T}\,e^{3T/4}\right).
\end{equation}
The closed geodesics $\gamma$ in $X=\Gamma\backslash\mathbb{H}$ are counted with orientation, equivalently as nontrivial conjugacy classes $\{\gamma \}$ in $\Gamma$.

\begin{lemma}\label{lem:wPrimLengthSum}
We have
\[
    \sum_{l(\gamma)<T} l(\gamma_0)e^{-\delta l(\gamma)}
    =
    \left\{
    \begin{aligned}
        &O_{X,\delta}(1) &&\delta>1\\
        &T+O_X(1)   &&\delta=1.
    \end{aligned}
    \right.
\]
Moreover, for $\delta>1$, we have
\[
    \sum_{l(\gamma)\ge T} l(\gamma_0)e^{-\delta l(\gamma)}
    =O_{X,\delta}(e^{-(\delta-1)T}).
\]
\end{lemma}
\begin{proof}
    By \eqref{eqn:PrimLengthSum}, we have
    \[
        S(T):=\sum_{l(\gamma)<T} l(\gamma_0)=e^T+O_X(\sqrt{T}e^{(1-\eta_X)T}),
    \]
    where $\eta_X=\min\{1/4,1/2-|\Im(t_1)|\}$. Then 
    \begin{align*}
        \sum_{l(\gamma)< T} l(\gamma_0)e^{-\delta l(\gamma)}
        &=\int_0^Te^{-\delta x}\,dS(x)\\
        &= e^{-\delta T}S(T)+\delta\int_0^Te^{-\delta x} S(x)\,dx\\
        &=\begin{cases}
            O_{X,\delta}(1) & \delta>1\\
            T+O_X(1) & \delta=1.
        \end{cases}
    \end{align*}
    On the other hand, for $\delta>1$,
    \begin{align*}
        \sum_{l(\gamma)\ge T} l(\gamma_0)e^{-\delta l(\gamma)}
        &=\int_T^\infty e^{-\delta x}\,dS(x)\\
        &= -e^{-\delta T}S(T)+\delta\int_T^\infty e^{-\delta x} S(x)\,dx\\
        &=O_{X,\delta} (e^{-(\delta-1)T}).
    \end{align*}
\end{proof}

\section{Zelditch's Trace Formula}\label{sec:Zelditch_Trace_Formula}
In this section, we derive Zelditch's trace formula for compact hyperbolic surfaces. Although the fundamental arguments are based on \cite{Zelditch_EquidistCompact}, we reproduce all necessary details in this section to make it self-contained. See also Bir\'o's generalized Selberg trace formula
\cite{Biro_Gen_STF}, which studies the same type of
eigenfunction-weighted diagonal kernel.
\begin{theorem}\label{thm:ZTF}
Let $\varphi$ be a non-constant eigenfunction of the Laplace--Beltrami operator on a compact hyperbolic surface $Y$ with eigenvalue $\lambda=\frac{1}{4}+\tau^2>0$.  Let $\kappa_\pm = \frac{1}{2}\left(\frac{1}{2}\pm i\tau\right)$. Suppose $\psi \in C_c^\infty((0,\infty))$ and $M\psi(0)=0$. Then, we have
\[
\sum_{j=1}^\infty \langle \varphi\phi_j,\phi_j\rangle h(t_j)  =   \sum_{\{\gamma\}} \psi\left(\frac{(e^{l(\gamma)}-1)^2}{e^{l(\gamma)}}\right)\int_{\gamma_0}\varphi(s)ds,
\]
where both sides are absolutely convergent. Moreover,
\[
    h(t) =  \frac{1}{2\pi i}\int_{(\sigma)}\Theta_\tau^t (s) M \psi(s) ds
\]
where
\[
    \Theta^t_\tau(s)=\frac{2^{2-2s} \sqrt{\pi}}{\sin(\pi s)}\frac{\Gamma\left(-\frac{1}{2}+it+s\right)\Gamma\left(-\frac{1}{2}-it+s\right)}{\Gamma\left(-\frac{1}{2}\left(\frac{1}{2}+i\tau\right)+s\right)\Gamma\left(-\frac{1}{2}\left(\frac{1}{2}-i\tau\right)+s\right)}\cosh(\pi t)
\]
and $\frac{1}{2}+\Im(t)<\sigma <1$.
\end{theorem}

We refer to $h$ and $\psi$ as the \emph{spectral} and \emph{geometric test functions}, respectively. 

A spectral test function $h(t)$ is said to be \emph{admissible} if there exist $\epsilon,\delta>0$ such that
\begin{enumerate}
    \item $h$ is even,
    \item $h$ is holomorphic in
    $\{t\in\mathbb C:|\Im t|<1/2+\epsilon\}$, and
    \item $|h(t)|\ll (1+|t|)^{-(2+\delta)}$ uniformly in this strip.
\end{enumerate}
See \cite[(1.63)]{Iwaniec_SpectralMethods_Book} or \cite[Condition (B)]{Biro_Gen_STF}.
If $h$ is admissible, then both sides of the Selberg trace formula converge absolutely; see \cite[Theorem 10.2 and Remarks on p.\@ 152]{Iwaniec_SpectralMethods_Book}.

If $\psi\in C_c^\infty((0,\infty))$ and $M\psi(0)=0$,
then the corresponding spectral test function $h$ is admissible;
see \cite[Theorem 2]{Biro_Gen_STF}. See also Remark~\ref{rem:role of Mpsi(0)=0} for the role of $M\psi(0)=0$.

We obtain Theorem~\ref{thm:ZTF} by multiplying the Selberg pre-trace formula by $\varphi$ and then integrating over $Y$. Since $|\varphi|_\infty \ll (1+|\tau|)^{1/2}$ \cite{SoggeZelditch}, the absolute convergence of both sides of Theorem~\ref{thm:ZTF} follows from the corresponding absolute convergence in the Selberg pre-trace formula.

\subsection{The spectral side}
For the spectral and geometric trace identities below, let $h$ be
an admissible spectral test function, and let
$k:[0,\infty)\to\mathbb C$ be its inverse Selberg/Harish-Chandra
transform, with the normalization in
\eqref{eqn:Selberg/Harish-Chandra}.
Define the associated point-pair invariant by
\begin{equation}
    k(z_1,z_2)
    :=k\left(\frac{|z_1-z_2|^2}{y_1y_2}\right).
\end{equation}
A point-pair invariant gives rise to a symmetric integral kernel on $\Gamma \backslash \mathbb{H}$ via averaging over $\Gamma$
\[
K(z_1,z_2) = \sum_{\gamma \in \Gamma} k(\gamma z_1, z_2)
\]
and the associated linear operator $T_K: L^2 (\Gamma \backslash \mathbb{H}) \to L^2 (\Gamma \backslash \mathbb{H})$
\begin{equation}\label{eq:TK}
T_K: f \mapsto \int_{\Gamma \backslash \mathbb{H}} K(z_1,z_2) f(z_2) d\mu_{z_2}= \int_{\Gamma \backslash \mathbb{H}} K(z_2,z_1) f(z_2) d\mu_{z_2}.
\end{equation}
Zelditch's Trace Formula \cite{Zelditch_EquidistCompact} expresses the trace of the composition of a pseudo-differential operator $\mathrm{Op}(\sigma)$ with the symbol $\sigma \in C^\infty (SY)$ and $T_K$ in two different ways. Here $SY$ is the unit tangent bundle of $Y$. Theorem \ref{thm:ZTF} is a specialization of Zelditch's Trace Formula, where the symbol is assumed to be an eigenfunction on $Y$. In this case, $\mathrm{Op}(\varphi)$ is simply a multiplication operator given by
\[
\mathrm{Op}(\varphi): f\mapsto \varphi f,
\]
and  $\mathrm{Op}(\varphi) \circ T_K$ is a Hilbert--Schmidt operator with the integral kernel $\varphi(z_1) K(z_1,z_2)$.

To compute its trace, we recall Selberg's pre-trace formula \cite{Selberg_TraceFormula1956}
\begin{equation}\label{selpretrace}
K(z_1,z_2) = \sum_{j=0}^\infty h(t_j)  \phi_j(z_1) \overline{\phi_j(z_2)},
\end{equation}
where the Selberg/Harish-Chandra transform $k \leftrightarrow h$ is given by
\begin{align}
\begin{split}\label{eqn:Selberg/Harish-Chandra}
    Q(v) &= \int_{v}^\infty k(u) (u-v)^{-1/2} \,du\\
    g(r) &= Q\left(\left(2\sinh \frac{r}{2}\right)^2\right)\\
    h(t) &= \int_{\mathbb{R}} g(r) \cos rt \,dr.    
\end{split}
\end{align}
By multiplying \eqref{selpretrace} by $\varphi$ and then integrating over the diagonal, we have
\begin{equation}\label{spec}
\mathrm{Tr} \left(\mathrm{Op}(\varphi) \circ T_K\right) = \int_{\Gamma \backslash \mathbb{H}} \varphi(z) K(z,z)  d\mu_{z} = \sum_{j=0}^\infty h(t_j)  \langle \varphi\phi_j,\phi_j \rangle.
\end{equation}

\subsection{The geometric side}
The most important observation in Zelditch's work \cite{Zelditch_EquidistCompact} is that when the base manifold is a hyperbolic surface one can express the trace of $\mathrm{Op}(\varphi) \circ T_K$ as a weighted summation of integrals of $\varphi$ over the closed orbits of the geodesic flow on the unit tangent bundle.

Let $\mathcal{F}$ be a fundamental domain of $\Gamma$. Then we first have
\begin{align*}
\mathrm{Tr}\left(\mathrm{Op}(\varphi) \circ T_K\right) &= \int_{\Gamma \backslash \mathbb{H}} \varphi(z) K(z,z) d\mu_z\\
&= \int_{\mathcal{F}} \varphi(z) \sum_{\gamma \in \Gamma}k(\gamma z,z) d\mu_z\\
&= \sum_{\gamma \in \Gamma} \int_{\mathcal{F}} \varphi(z) k(\gamma z,z) d\mu_z.
\end{align*}
We partition $\Gamma$ into conjugacy classes
\[
\{\gamma\} := \{\tau^{-1}\gamma \tau~:~ \tau \in \Gamma_\gamma \backslash \Gamma  \},
\]
where $\Gamma_\gamma$ is the centralizer of $\gamma$ in $\Gamma$. Then we consider the summation taken over each conjugacy class $\{\gamma\}$:
\begin{align*}
\mathrm{Tr}_{\{\gamma\}}&=\sum_{\tau \in \Gamma_\gamma \backslash \Gamma } \int_{\mathcal{F}} \varphi(z) k(\tau^{-1}\gamma \tau z,z) d\mu_z\\
&= \sum_{\tau \in \Gamma_\gamma \backslash \Gamma } \int_{\mathcal{F}} \varphi(\tau z) k(\gamma \tau z,\tau z) d\mu_z\\
&= \sum_{\tau \in \Gamma_\gamma \backslash \Gamma } \int_{\tau\mathcal{F}} \varphi(z) k(\gamma z,z) d\mu_z\\
&=  \int_{\Gamma_\gamma \backslash \mathbb{H}} \varphi(z) k(\gamma z,z) d\mu_z.
\end{align*}
Because we assume $Y=\Gamma \backslash \mathbb{H}$ is a compact hyperbolic surface, any element in $\Gamma$ is either the identity or a hyperbolic element. Thus we only consider these two cases, under the assumption that $\varphi$ is an eigenfunction of the Laplace--Beltrami operator with the eigenvalue $\lambda>0$.

The first case is when $\gamma$ is the identity. Because $\varphi$ is assumed to be a non-constant eigenfunction, we see that the contribution from the identity element is trivial:
\[
\mathrm{Tr}_{\{\mathrm{Id}\}} = \int_{\mathcal{F}} \varphi(z)k(z,z) d\mu_z = k(0) \int_{\Gamma \backslash \mathbb{H}} \varphi(z) d\mu_z = 0.
\]
Hence for the rest of the section, let $\gamma$ be a hyperbolic element. Let $\gamma_0$ be the generator of $\Gamma_\gamma$. Pick $\alpha \in {\rm PSL}_2(\mathbb{R})$ such that  $\alpha \gamma_0\alpha^{-1}= a_{0}  $ for some $a_0 \in A$. Let $\alpha \gamma\alpha^{-1} = a_\gamma \in A$. Then we have
    \begin{align*}
        \mathrm{Tr}_{\{\gamma\}} & = \int_{ \alpha \Gamma_\gamma \alpha^{-1} \backslash \mathbb{H}} \varphi(\alpha^{-1} z) k(\gamma \alpha^{-1} z, \alpha^{-1} z) \,d \mu\\
        &= \int_{\alpha \Gamma_\gamma \alpha^{-1} \backslash \mathbb{H}} \varphi(\alpha^{-1} z)  k( \alpha \gamma \alpha^{-1} z, z) \,d \mu\\
        &= \int_{\langle a_0 \rangle \backslash \mathbb{H}} \varphi(\alpha^{-1} z)  k( a_\gamma  z, z) \,d \mu.
    \end{align*}
We choose the fundamental domain of $\langle a_0 \rangle \backslash \mathbb{H}$ to be
\[
\{z\in \mathbb{H}~:~ 1<y<e^{l(\gamma_0)}\},
\]
which is the same as $\{(u,t)~:~ 0<t<l(\gamma_0)\}$ in the $AN$-coordinate system. We have $(u,t) = a_t (u,0)$ in the $AN$-coordinate system. Because $A$ is abelian, we have
\[
k(a_\gamma  az,az )= k(a_\gamma  z,z )
\]
for any $a\in A$. Therefore the integral evaluated in the $AN$-coordinate system is equal to
\begin{align*}
\mathrm{Tr}_{\{\gamma\}}=&\int_{-\infty}^\infty \int_0^{l(\gamma_0)} \varphi(\alpha^{-1} (u,t))  k( a_\gamma  (u,t), (u,t)) dtdu\\
= &\int_{-\infty}^\infty \int_0^{l(\gamma_0)} \varphi(\alpha^{-1} (u,t))  k( a_\gamma  a_t(u,0), a_t (u,0)) dtdu\\
= &\int_{-\infty}^\infty  k( a_\gamma  (u,0), (u,0)) \left(\int_0^{l(\gamma_0)} \varphi(\alpha^{-1} (u,t))  dt\right)du\\
=&\int_{-\infty}^\infty  k\left( \frac{\left|e^{l(\gamma)}(u+i) - (u+i)\right|^2}{e^{l(\gamma)}}\right) \left(\int_0^{l(\gamma_0)} \varphi(\alpha^{-1} (u,t))  dt\right)du\\
=&\int_{-\infty}^\infty  k\left( \frac{(e^{l(\gamma)}-1)^2}{e^{l(\gamma)}} (u^2+1)\right) \left(\int_0^{l(\gamma_0)} \varphi(\alpha^{-1} (u,t))  dt\right)du.
\end{align*}
In \cite{Zelditch_EquidistCompact}, Zelditch observed that because $\varphi$ is assumed to be an eigenfunction of the Laplace--Beltrami operator, the integral
\[
    I(u) = \int_0^{l(\gamma_0)} \varphi(\alpha^{-1} (u,t))  dt
\]
must be a solution of the following second-order ordinary differential equation
\[
    \left((1+u^2)\frac{d^2}{du^2}  + 2u \frac{d}{du} + \lambda\right) I = 0.
\]
Recall that $\lambda=\frac{1}{4}+\tau^2$ and $\kappa_\pm = \frac{1}{2}\left(\frac{1}{2}\pm i\tau\right)$.
The differential equation has two linearly independent solutions 
\[
    I_1 (u) =  {_2F_1}\left(\kappa_+,\kappa_-,\frac{1}{2},-u^2\right)
\]
and
\[
    I_2(u) = u\cdot{_2F_1}\left(\frac{1}{2}+\kappa_+,\frac{1}{2}+\kappa_-,\frac{3}{2},-u^2\right).
\]
If we write
\[
    I(u) = c_1I_1(u) + c_2I_2(u),
\]
then
\[
    c_1 = I(0) = \int_{\gamma_0} \varphi \, ds.
\]
Also, since $k\left( \frac{(e^{l(\gamma)}-1)^2}{e^{l(\gamma)}} (u^2+1)\right) $ is an even function in $u$ while $I_2(u)$ is odd, we have
\[
    \int_{-\infty}^\infty  k\left( \frac{(e^{l(\gamma)}-1)^2}{e^{l(\gamma)}} (u^2+1)\right) I_2(u) du =0.
\]
Therefore we obtain
\begin{equation}\label{geom}
\mathrm{Tr}_{\{\gamma\}} = \left( \int_{\gamma_0} \varphi  ds\right)\left(\int_{-\infty}^\infty  k\left( \frac{(e^{l(\gamma)}-1)^2}{e^{l(\gamma)}} (u^2+1)\right) {_2F_1}\left(\kappa_+,\kappa_-,\frac{1}{2},-u^2\right) du\right).
\end{equation}

\subsection{\texorpdfstring{The integral transform $H_\tau: k\mapsto \psi$}{The integral transform H}}
Define

\[
    (H_\tau k)(u) = 2\int_0^\infty k(u(r^2+1)) {_2F_1}\left(\kappa_+,\kappa_-,\frac{1}{2},-r^2\right) dr.
\]
Then by combining \eqref{spec} and \eqref{geom}, we have
\[
\sum_{j=0}^\infty h(t_j)  \langle \varphi\phi_j,\phi_j \rangle  = \mathrm{Tr} \left(\mathrm{Op}(\varphi) \circ T_K\right) =  \sum_{\{\gamma\}} \left\{\left(H_\tau k\left(\frac{(e^{l(\gamma)}-1)^2}{e^{l(\gamma)}} \right)\right) \left(\int_{\gamma_0} \varphi  ds\right)\right\}.
\]
To complete the proof of Theorem \ref{thm:ZTF}, for $\psi \in C_0^\infty(0,\infty)$ with $k=H_\tau^{-1}\psi$, we express $h(\cdot)$ in terms of $\psi$. To this end, we first need to invert $H_\tau$.  Note that $0<\mathrm{Re}(\kappa_\pm)<\frac{1}{2}$.
\begin{proposition}\label{prop:philambda}
$H_\tau=M_\sigma^{-1} \Phi_\tau M$ where $\sigma > 1/2$, and
\begin{equation}\label{eqn:Phi_lambda}
    \Phi_\tau (s) = \frac{\Gamma\left(\frac{1}{2}\right)\Gamma\left(-\kappa_++s\right)\Gamma\left(-\kappa_-+s\right)}{\Gamma(s)^2}.
\end{equation}
\end{proposition}
\begin{proof}
    By applying \cite[7.512-(10) with $\alpha=\kappa_+,\beta=\kappa_-,\gamma=\tfrac{1}{2},\sigma=s$]{GR_Table}, we have
    \begin{align*}
    &\int_0^\infty (r^2+1)^{-s}{_2F_1}\left(\kappa_+,\kappa_-,\frac{1}{2},-r^2\right) dr\\
    =&\frac{1}{2}\int_0^\infty x^{-\frac{1}{2}}(x+1)^{-s}{{}_2F_1}\left(\kappa_+,\kappa_-,\frac{1}{2},-x\right) dx\\
    =&\frac{\Gamma\left(\frac{1}{2}\right)\Gamma\left(-\kappa_++s\right)\Gamma\left(-\kappa_-+s\right)}{2\Gamma\left(s\right)^2},
    \end{align*}
    where $\mathrm{Re}(s) > \frac{1}{4}+\frac12|\Im(\tau)|$. This implies that if we let $\phi_s(u) = u^{-s}$ with $\mathrm{Re}(s) > \frac{1}{2}$, then
    \[
    (H_\tau \phi_s)(u) = \Phi_\tau(s) \phi_s(u).
    \]
    Therefore, we have
    \[
    (H_\tau k)(u) = H_\tau \frac{1}{2\pi i}\int_{(\sigma)} (Mk)(s) u^{-s} ds = \frac{1}{2\pi i}\int_{(\sigma)} (Mk)(s)\Phi_\tau(s) u^{-s} ds = (M_\sigma^{-1} \Phi_\tau M k)(u)
    \]
    provided that $\sigma> \frac{1}{4}+\frac12|\Im(\tau)|$. Since $|\Im(\tau)|<1/2$, the formula is valid for all $\sigma>1/2$.
\end{proof}

\begin{remark}\label{rem:role of Mpsi(0)=0}
    Let us briefly discuss the role of the condition $M\psi(0)=0$, for which we previously simply cited \cite[Theorem~2]{Biro_Gen_STF}. Under the assumption $\psi\in C_c^\infty((0,\infty))$, its Mellin transform $M\psi$ is entire. 
    Since $\Gamma(s)^2$ has a double pole at $s=0$, the condition $M\psi(0)=0$ reduces the pole of $Mk(s)$ at $s=0$ to at most a simple pole. Recall that $k=M_\sigma^{-1}(Mk)$ for $\sigma>0$.

    When $M\psi(0)=0$, shifting the contour $(\sigma)$ to $(-1/2)$ and picking up the residue at $s=0$ gives
    \[
        k(u)
        =
        \underbrace{\frac{(M\psi)'(0)}{\sqrt{\pi}\Gamma(-\kappa_+)\Gamma(-\kappa_-)}}_{={\rm Res}_{s=0}Mk(s)} + O_{\psi,\tau}(u^{1/2})
    \]
    as $u\to0^+$. Thus $k$ extends continuously to $u=0$. In particular, the ordinary Selberg identity term $k(0)\operatorname{Area}(Y)$ is well-defined. In the weighted trace formula considered here, the identity contribution is
    $k(0)\int_Y\varphi\,d\mu=0$.

    On the other hand, if $M\psi(0)\neq0$, then $Mk(s)$ has a double pole at $s=0$. When the contour is shifted past $s=0$, this double pole produces a $\log(1/u)$ term in $k(u)$ as $u\to0^+$. Thus $k(u)$ is logarithmically singular at $u=0$.    
\end{remark}

\subsection{Selberg/Harish-Chandra transform}
By using the Selberg/Harish-Chandra transform (in the Selberg trace formula), we obtain the explicit formula for
\[
    \mathcal{H}_\tau:\psi\mapsto h,
\]
the transform between the geometric and spectral test functions in Theorem~\ref{thm:ZTF}, which we also refer to as the Selberg/Harish-Chandra transform.

\begin{lemma}\label{lem:psi_to_h}
Suppose $\psi\in C_c^\infty((0,\infty))$. Suppose that $t\in \{ z\in \mathbb{C} : |\Im(z)|<1/2\}$ and $1/2+|\Im(t)|<\sigma<1$. Then, we have
\[
    h(t) =  \frac{1}{2\pi i}\int_{(\sigma)}\Theta^t_\tau(s) M \psi(s) ds,
\]
and
\begin{equation}\label{eqn:Theta}
    \Theta_\tau^t (s)
    =
    \frac{2^{2-2s} \sqrt{\pi}}{\sin(\pi s)}\frac{\Gamma\left(-\frac{1}{2}+it+s\right)\Gamma\left(-\frac{1}{2}-it+s\right)}{\Gamma\left(-\frac{1}{2}\left(\frac{1}{2}+i\tau\right)+s\right)\Gamma\left(-\frac{1}{2}\left(\frac{1}{2}-i\tau\right)+s\right)}\cosh(\pi t).
\end{equation}
\end{lemma}
\begin{proof}
It follows from $\psi\in C_c^\infty((0,\infty))$ that $M\psi$ is an entire function and for all $N>0$, we have
\[
    |M\psi(x+iy)|\ll_{\psi,x,N}(1+|y|)^{-N}.
\]
Recall the Selberg/Harish-Chandra transformation formula and the definition of $Q$ in \eqref{eqn:Selberg/Harish-Chandra}. We have
\begin{align}
\begin{split}\label{eqn:q and k prep}
    Q(v)
    &= \int_0^\infty    k(u+v) u^{-\frac{1}{2}}du\\
    &= v^{\frac{1}{2}}\int_0^\infty    k((u+1)v) u^{-\frac{1}{2}}du\\
    &= v^{\frac{1}{2}}\frac{1}{2\pi i}\int_0^\infty \left(\int_{(\sigma)}    \Phi_\tau^{-1}(s)M\psi(s)((u+1)v)^{-s} u^{-\frac{1}{2}}ds\right)du\\
    &= v^{\frac{1}{2}}\frac{1}{2\pi i} \int_{(\sigma)} \Phi_\tau^{-1}(s)M\psi(s)v^{-s} \left(\int_0^\infty   (u+1)^{-s} u^{-\frac{1}{2}}du\right) ds.
\end{split}
\end{align}
Since $\sigma=\Re(s)>1/2$, by \cite[8.380-(3)]{GR_Table}, we have
\[
    \int_0^\infty (u+1)^{-s} u^{-\frac{1}{2}}du
    =B\left(s-\tfrac12,\tfrac12\right)
    =\frac{\Gamma\left(s-\frac{1}{2}\right)\Gamma\left(\frac{1}{2}\right)}{\Gamma(s)}.
\]
Then, by using \eqref{eqn:Phi_lambda}, we obtain
\begin{equation}\label{eqn:q to Mpsi}
    Q(v) = v^{1/2} M_{\sigma}^{-1} \left(M\psi(s) \,\frac{\,\Gamma(s)\Gamma(s-\frac{1}{2})}{\Gamma(-\kappa_++s)\Gamma(-\kappa_-+s)}\right)(v).
\end{equation}

From \eqref{eqn:Selberg/Harish-Chandra}, we have
\[
    h(t)
    =
    \int_{-\infty}^{\infty} Q\left(\left(2\sinh\frac r2\right)^2\right)e^{irt}\,dr.
\]
With the substitution $e^{r/2}=a$, this becomes
\[
    h(t)
    =
    2\int_0^\infty Q\left(\left(a-\frac1a\right)^2\right) a^{2it}\frac{da}{a}.
\]
Using \eqref{eqn:q to Mpsi}, we obtain
\[
    h(t)
    =
    2\frac{1}{2\pi i} \int_{(\sigma)} M\psi(s) \frac{\Gamma(s)\Gamma(s-\frac12)}{\Gamma(s-\kappa_+)\Gamma(s-\kappa_-)} J(s,t)\,ds,
\]
where $\frac{1}{2}<\sigma$ and 
\[
J(s,t)
:=
\int_0^\infty
\left|a-\frac1a\right|^{1-2s}
a^{2it}\frac{da}{a}.
\]
Since $\frac12+|\Im(t)|<\Re(s)<1$, this integral for $J(s,t)$ is absolutely convergent.
Splitting at $a=1$ and using $a\mapsto a^{-1}$ in the second integral gives
\[
\begin{aligned}
J(s,t)
&=
\int_0^1
\left(\frac1a-a\right)^{1-2s}
\left(a^{2it}+a^{-2it}\right)
\frac{da}{a}  \\
&=
\int_0^1
(1-a^2)^{1-2s}
a^{2s-1}
\left(a^{2it}+a^{-2it}\right)
\frac{da}{a}.
\end{aligned}
\]
After the substitution $x=a^2$, we get
\[
J(s,t)
=
\frac12
\int_0^1
(1-x)^{1-2s}
x^{s-\frac12}
\left(x^{it}+x^{-it}\right)
\frac{dx}{x}.
\]
Therefore, we have
\[
\begin{aligned}
J(s,t)
&=
\frac12 B\left(2-2s,s-\frac12+it\right)
+
\frac12 B\left(2-2s,s-\frac12-it\right) \\
&=
\Gamma(2-2s)
\frac{
\Gamma(s-\frac12+it)\Gamma(s-\frac12-it)
\cosh \pi t
}{
\Gamma(\frac32-s)\Gamma(s-\frac12)
},
\end{aligned}
\]
and thus
\[
    h(t) = \frac{1}{2\pi i} \int_{(\sigma)} \Theta_\tau^t(s)M\psi(s)\,ds,
\]
where
\[
    \Theta_\tau^t(s)
    =
    \frac
    {
    2\Gamma(s)\Gamma(2-2s) \Gamma(s-\frac12+it) \Gamma(s-\frac12-it)\cosh\pi t
    }
    {
    \Gamma(\frac32-s)\Gamma(s-\kappa_+)\Gamma(s-\kappa_-)
    }.
\]
By using the reflection formula \eqref{eqn:reflection} and the duplication formula \eqref{eqn:duplication}, one can obtain
\[
    \frac{\Gamma(2-2s)\Gamma(s)}{\Gamma(\frac{3}{2}-s)}=\frac{2^{1-2s}\sqrt{\pi}}{\sin\pi s}.
\]
Hence,
\[
    \Theta_\tau^{t} (s) = \frac{2^{2-2s} \sqrt{\pi}}{\sin\pi s}\frac{\Gamma\left(-\frac{1}{2}+it+s\right)\Gamma\left(-\frac{1}{2}-it+s\right)}{\Gamma\left(-\frac{1}{2}\left(\frac{1}{2}+i\tau\right)+s\right)\Gamma\left(-\frac{1}{2}\left(\frac{1}{2}-i\tau\right)+s\right)}\cosh\pi t.\qedhere
\]
\end{proof}

\subsection{Inverse Selberg/Harish-Chandra transform}
Later, we need the inverse Selberg/Harish-Chandra transform formula. Let us express $M\psi$ in terms of $h$.

\begin{lemma}\label{lem:h_to_Mpsi}
Suppose that $h(t)$ is even and holomorphic on the strip $\{t:|\Im(t)|<\frac12+\epsilon_0\}$, where $\epsilon_0\ge0$, such that for every $a\in (0,\frac12+\epsilon_0)$, we have $|h(t)|\ll_a(1+|t|)^{-2}$ uniformly in the strip $|\Im(t)|\le a$. Then, $M\psi(s)$ is holomorphic on the strip $\{ s:\frac12<\Re(s)<1+\epsilon_0\}$ and
\begin{equation}\label{eqn:h_to_Mpsi_1}
    M\psi(s)=-\frac{2^{2s-1}}{\pi^{3/2}}\Gamma(s-\kappa_+)\Gamma(s-\kappa_-)\left(\frac{1}{2\pi i}\int_{(-\eta)}A(s')\, ds'\right),
\end{equation}
where $\eta\in(\Re(s)-\frac12,\frac12+\epsilon_0)$ and
\[
    A(s'):=\pi \,h(-is')\frac{\Gamma(\tfrac12-s-s')s'}{\Gamma(\frac12+s-s')}.
\]
Moreover, if $\frac12<\Re(s)<1+\min\{\epsilon_0,\frac12\}$ with $s\neq 1$, then $M\psi(s)$ can also be expressed by
\begin{equation}\label{eqn:h_to_Mpsi_2}
    M\psi(s)
    =\frac{2^{2s-1}}{\pi^{3/2}}\Gamma(s-\kappa_+)\Gamma(s-\kappa_-)\sin(\pi s)
    \left(2\,{\rm Res}_{s'=\frac12-s}F(s')-
    \frac{1}{2\pi i}\int_{(0)} F(s') ds' \right),
\end{equation}
where
\[
    F(s'):=h(-is')\Gamma(\tfrac12-s-s')\Gamma(\tfrac12-s+s')s' \sin(\pi s').
\]
\end{lemma}
\begin{proof}
Recall $H_\tau k=\psi$. By Proposition \ref{prop:philambda}, we have $M\psi(s)=\Phi_\tau(s)\,Mk(s)$.
By \eqref{eqn:q and k prep},  we have
\begin{align*}
    M(v^{-1/2}Q(v))(s)
    &= \int_0^\infty \left(\int_0^\infty    k((u+1)v) u^{-\frac{1}{2}}du \right)v^{s-1}dv\\
    &= \int_0^\infty \left(\int_0^\infty    k((u+1)v) v^{s-1}dv \right)u^{-\frac{1}{2}}du.
\end{align*}
By the substitution $w=(u+1)v$, 
\[
    \int_0^\infty    k((u+1)v) v^{s-1}dv=(u+1)^{-s}\int_0^\infty k(w) w^{s-1}dw=(u+1)^{-s}Mk(s).
\]
Hence, by \cite[8.380-(3) with $x=\tfrac12$ and $y=s-\tfrac12$]{GR_Table}, we obtain
\[
    M(v^{-1/2}Q(v))(s)
    = \,Mk(s)\int_0^\infty (1+u)^{-s}u^{-1/2}du
    =\,Mk(s)\,B(\tfrac12,s-\tfrac12)
    =\sqrt{\pi}\,\frac{\Gamma(s-\tfrac12)}{\Gamma(s)}\,Mk(s),
\]
which is valid for $\Re(s)>1/2$. Hence
\begin{equation}\label{eqn:Mk by Mq}
    Mk(s)=\frac{1}{\sqrt{\pi}}\frac{\Gamma(s)}{\Gamma(s-\tfrac12)}M(v^{-1/2}Q(v))(s).
\end{equation}
Now, we consider $M(v^{-1/2}Q(v))(s)$.
\begin{align*}
    M(v^{-\frac{1}{2}}Q(v))(s)
    &= \int_0^\infty Q(v) v^{s-\frac{1}{2}}\frac{dv}{v}\\
    &=4\int_{0}^\infty Q\left(\left(2\sinh \frac{r}{2}\right)^2\right) \left(2\sinh \frac{r}{2}\right)^{2(s-\frac{1}{2}-1)} \sinh \frac{r}{2} \cosh \frac{r}{2} dr\\
    &=4\int_{0}^\infty g(r) \left(2\sinh \frac{r}{2}\right)^{2(s-\frac{1}{2}-1)} \sinh \frac{r}{2} \cosh \frac{r}{2} dr.
\end{align*}
By using
\[
    g(r)=\frac{1}{2\pi}\int_\mathbb{R}h(t)e^{-irt}dt=\frac{1}{2\pi}\int_\mathbb{R}h(t)e^{irt}dt
\]
and the substitution $e^{r}=a$, we obtain
\begin{align*}
    M(v^{-\frac{1}{2}}Q(v))(s)
    &=\frac{2}{\pi} \int_{0}^\infty \int_{-\infty}^\infty  h(t) e^{irt} \left(2\sinh \frac{r}{2}\right)^{2(s-\frac{1}{2}-1)} \sinh \frac{r}{2} \cosh \frac{r}{2} dtdr\\
    &=\frac{1}{2\pi} \int_{1}^\infty \int_{-\infty}^\infty  h(t) a^{it} \left(a^{1/2}-\frac{1}{a^{1/2}}\right)^{2s-3} \left(a-\frac{1}{a}\right) dt\frac{da}{a}\\
    &=\frac{1}{2\pi} \int_{0}^1 \int_{-\infty}^\infty  h(t) a^{-it} \left(\frac{1}{a^{1/2}}-a^{1/2}\right)^{2s-3} \left(\frac{1}{a}-a\right) dt\frac{da}{a}\\
    &=\frac{1}{2\pi} \int_{0}^1 \int_\mathbb{R}  h(t) a^{-\frac12-s-it}(1-a)^{2s-2}(1+a)\, dtda.
\end{align*}

We want to apply the Fubini theorem, but the integral
\[
    I(s,t):=\int_0^1 a^{-\frac12-s-it}(1-a)^{2s-2}(1+a)\,da
\]
diverges when $t\in \mathbb{R}$. It absolutely converges when $\Im(t)>\Re(s)-\tfrac12$. Note that $\int_{\mathbb{R}}h(x+iy)a^{-i(x+iy)} dx$ is absolutely convergent for $|y|<1/2+\epsilon_0$. Accordingly, we shift the contour for $t$ from $\mathbb{R}$ to $\mathbb{R}+i\eta$ with $\Re(s)-1/2<\eta<1/2+\epsilon_0$ and obtain
\[
    M(v^{-\frac{1}{2}}Q(v))(s)=\frac{1}{2\pi} \int_{0}^1 \int_{\mathbb{R}+i\eta}  h(t) a^{-\frac12-s-it}(1-a)^{2s-2}(1+a)\, dtda.
\]
Since $\eta>\Re(s)-1/2$ and $\Re(s)>1/2$, for $t\in \mathbb{R}+i\eta$, we have
\begin{align*}
    I(s,t)
    &=\int_0^1 a^{-\frac12-s-it}(1-a)^{2s-2}(1+a)\,da\\
    &=\int_0^1 a^{-\frac12-s-it}(1-a)^{2s-2}\,da+\int_0^1 a^{\frac12-s-it}(1-a)^{2s-2}\,da\\
    &=B(\tfrac12-s-it,2s-1)+B(\tfrac32-s-it,2s-1).
\end{align*}
By using the identity $B(z,w)=\frac{\Gamma(z)\Gamma(w)}{\Gamma(z+w)}$ and the reflection formula 
\eqref{eqn:reflection}, for $t\in \mathbb{R}+i\eta$, we have
\begin{align*}
    I(s,t)
    &=\Gamma(2s-1)\left(\frac{\Gamma(\frac12-s-it)}{\Gamma(-\frac12+s-it)}+\frac{\Gamma(\tfrac32-s-it)}{\Gamma(\frac12+s-it)}\right)\\
    &=-2it
    \frac{\Gamma(\tfrac12-s-it)\Gamma(2s-1)}{\Gamma(\frac12+s-it)}\\
    &=-\frac{2it}\pi \Gamma(\tfrac12-s-it)\Gamma(\tfrac12-s+it)\Gamma(2s-1)\sin\left(\pi(\tfrac12-s+it)\right).
\end{align*}

Hence, by the Fubini theorem and the change of coordinate $s'=it$, we obtain
\[
    M(v^{-\frac{1}{2}}Q(v))(s)
    =\frac{1}{2\pi} \int_{\mathbb{R}+i\eta}  h(t) I(s,t)\,dt
    =\frac{i}{\pi^2}\Gamma(2s-1)\int_{(-\eta)}A(s') \,ds',
\]
where
\begin{align*}
    A(s')
    &:=h(-is')\Gamma(\tfrac12-s-s')\Gamma(\tfrac12-s+s')s'\sin(\pi(\tfrac12-s+s'))\\
    &=\pi \,h(-is')\frac{\Gamma(\tfrac12-s-s')s'}{\Gamma(\frac12+s-s')}.
\end{align*}
Then the formula \eqref{eqn:h_to_Mpsi_1} follows from \eqref{eqn:duplication}, \eqref{eqn:Mk by Mq}, and Proposition \ref{prop:philambda}.

If $\Re(s)\in(1/2,3/2)$, then shifting the contour from $(-\eta)$ to $(0)$ picks up the residue only at $s'=\frac12-s$. Hence,
\[
    \frac{1}{2\pi i}\int_{(-\eta)}A(s') \,ds'= -{\rm Res}_{s'=\frac12-s}A(s')+\frac{1}{2\pi i}\int_{(0)}A(s')\, ds'.
\]
Since
\[
    \sin(\pi(\tfrac12-s+s'))=\cos(\pi(s-s'))=\cos(\pi s)\cos(\pi s')+\sin(\pi s)\sin(\pi s')
\]
and $h$ is even, it follows that
\[
    \int_{(0)} A(s')\,ds'
    =\sin(\pi s)\int_{(0)}F(s')\, ds',
\]
where
\[
    F(s')=
    h(-is')\Gamma(\tfrac12-s-s')\Gamma(\tfrac12-s+s')s' \sin(\pi s').
\]
It is easy to show that
\[
    {\rm Res}_{s'=\frac12-s}A(s')=2\sin(\pi s)\,{\rm Res}_{s'=\frac12-s}F(s')
\]
except at $s\in\mathbb{Z}_{> 0}$.
Therefore, we obtain \eqref{eqn:h_to_Mpsi_2}.
\end{proof}

\section{Uniform estimates on matrix coefficients}\label{sec:Triple_Product}

As motivation, first consider the arithmetic surface
\(Y=\mathrm{PSL}_2(\mathbb Z)\backslash\mathbb H\), and let
\(\varphi_i\) be \(L^2\)-normalized Maass--Hecke cusp forms
with spectral parameters \(\tau_i\).
Combining Watson's formula with standard arithmetic inputs, including \cite{Hoffstein_Lockhart}, and the explicit representation of the Archimedean factor $L_\infty$ in terms of Gamma functions, we find that there exists a sufficiently large constant $A>0$ such that
\begin{multline*}
\int_Y \varphi_1\varphi_2\varphi_3 d\mu\ll_Y \prod_{j=1}^3(1+|\tau_j|)^A  \\
\times \exp\left(-\frac{\pi}{4} \left(|\tau_1+\tau_2+\tau_3|+|-\tau_1+\tau_2+\tau_3|+|\tau_1-\tau_2+\tau_3|+|\tau_1+\tau_2-\tau_3|-2|\tau_1|-2|\tau_2|-2|\tau_3|\right)\right)
\end{multline*}
uniformly in $\tau_1$, $\tau_2$, and $\tau_3$.
This arithmetic discussion is only motivational. The proof below is geometric and uses no Hecke operators or arithmetic input. From this perspective, for an arbitrary compact hyperbolic surface $Y$, the best bound in terms of the exponential factor one may expect for $\langle \varphi\phi,\phi \rangle$ is
\begin{equation}\label{uniform}
\langle \varphi\phi,\phi \rangle \ll_Y ((1+|\tau|)(1+|t|))^N \exp\left(-\frac{\pi}{4} \left(|2t+\tau|+|2t-\tau|-4|t|\right)\right)
\end{equation}
for some $N>0$, where $\tfrac14+\tau^2$ and $\tfrac14+t^2$ denote the eigenvalues of $\varphi$ and $\phi$, respectively.

Note that when $\varphi$ is fixed, this is trivially true because
\[
\abs{\langle \varphi\phi,\phi \rangle} \leq \|\varphi\|_{L^\infty} \langle \phi,\phi \rangle = \|\varphi\|_{L^\infty} = O(1).
\]
On the other hand, when $\phi$ is fixed, this is a theorem due to Sarnak \cite{Sarnak_Int_of_Prod_of_Eftns}.

In this section, we prove \eqref{uniform} for an arbitrary compact hyperbolic surface $Y$. We expect that the proof could be generalized to arbitrary dimension, but we do not pursue this in this article. To be precise, we prove Theorem \ref{thm:TripleProduct}:
\[
\int_Y \varphi(z)|\phi(z)|^2\,d\mu(z)
\ll_Y
(1+\tau)\bigl(1+\log(1+\tau)\bigr)
\exp\left(
-\frac{\pi}{4}
\left( |2t+\tau|+|2t-\tau|-4t \right)
\right),
\]
as $t,\tau \to \infty$. 

For fixed $\phi$, the estimate of
\cite[Section 0.2, Corollary]{Bernstein_Reznikov} (see also \cite{Sarnak_Int_of_Prod_of_Eftns}) gives the desired result.
Thus it remains to consider real $t$ above a fixed threshold.

To prove the estimate, we follow the geometric approach pioneered by Sarnak \cite{Sarnak_Int_of_Prod_of_Eftns} (see also \cite{Kiral_Subc} for an application and more explicit computations). The overarching strategy relies on reducing the global integral over the surface $Y$ to local estimates on the universal cover $\mathbb{H}$. 

By unfolding the integral, the problem transforms into bounding the localized matrix coefficients of the principal series representations of $\text{PSL}_2(\mathbb{R})$. The exponential decay of the triple product is intrinsically governed by the decay of these coefficients along the physical geodesics. 

The main technical hurdle is achieving strict uniformity in both spectral parameters $t$ and $\tau$ (note that in Sarnak's case $t$ is fixed). To handle this, we will isolate the matrix coefficient into a specific spectral sum, denoted by $\mathcal{B}(r)$. Bounding $\mathcal{B}(r)$ uniformly requires a discrete asymptotic analysis of conical functions. To achieve the desired bounds we will make crucial use of Dunster's work \cite{Dunster_Concial} on asymptotic expansions of conical functions where both parameters can be large.

\subsection{Uniform estimates on matrix coefficients}

We follow the initial steps of Sarnak \cite{Sarnak_Int_of_Prod_of_Eftns} (see also \cite[Sections 8--10]{Kiral_Subc} for an explicit exposition in the two-dimensional case). Let $k(z, z') = k(d(z,z'))$ be a point-pair invariant on $\mathbb{H}\times \mathbb{H}$ that is rapidly decreasing at infinity. The associated symmetric automorphic integral kernel is defined by
\[
    K(z, z') = \sum_{\gamma \in \Gamma} k(\gamma z, z').
\]
Spectrally expanding $K$, we obtain
\begin{equation}\label{FourierK}
    K(z, z')= \sum_{j=0}^\infty h(\tau_j)\phi_j(z)\overline{\phi_j(z')},
\end{equation}
where the eigenvalues are $\lambda_j = \frac{1}{4}+\tau_j^2$ and $h$ is the Selberg/Harish-Chandra transform of $k$. We consider the operator $K$ acting on the square of our eigenfunction $\phi$:
$$K\abs{\phi}^2(z):= \int_{Y} K(z, z')\abs{\phi(z')}^2d\mu_{z'}.$$
Using the spectral expansion in \eqref{FourierK}, we have
$K\abs{\phi}^2(z)=\sum_{j=0}^\infty h(\tau_j)\phi_j(z)\langle \abs{\phi}^2, \phi_j \rangle.$
To estimate the triple products \(\langle |\phi|^2,\phi_j\rangle\) from
this relation, we use Parseval's formula:
\[
\|K|\phi|^2\|_2^2
=
\sum_{j=0}^\infty
|h(\tau_j)\langle |\phi|^2,\phi_j\rangle|^2.
\]
We localize the spectral parameter near \(T\geq2\) by taking
\[
h_T(\tau)=e^{-(\tau-T)^2}+e^{-(\tau+T)^2}.
\]
Since \(h_T(\tau_j)\gg1\) when \(|\tau_j-T|\leq1\), we obtain

\[
\sum_{|\tau_j-T|\leq1}
|\langle |\phi|^2,\phi_j\rangle|^2
\ll_Y
\|K_T|\phi|^2\|_2^2
\ll_Y
\|K_T|\phi|^2\|_\infty^2.
\]

In the subsequent sections we are interested in the case when $T > 2t+ O(1)$. 
The complementary range \(T\leq 2t+2\) requires no exponential decay:
in this range the exponential factor in \eqref{uniform} is bounded below by
an absolute constant. The standard $L_\infty$ bound (see e.g. \cite{SoggeZelditch})

\begin{equation}\label{Sobolev}
|\langle |\phi|^2,\phi_j\rangle|
\leq
\|\phi_j\|_\infty\|\phi\|_2^2
\ll_Y
(1+\tau_j)^{1/2}
\end{equation}
is therefore sufficient after increasing the polynomial exponent in
\eqref{uniform}. Thus it remains to prove the sharp estimate in the hard
range \(T>2t+O(1)\), which follows from the following supremum-norm bound.

\begin{proposition}\label{prop:K_inf_bound}
Let $t\geq t_0(Y)$ be real, where $t_0(Y)$ is a sufficiently
large fixed constant, and let $T>2t+2$.
For the localization $h_T$ defined above, we have
\[
\|K_T|\phi|^2\|_\infty
\ll_Y
T(1+\log T)
\exp\left(-\frac{\pi}{2}T+\pi \abs{t}\right),
\]
uniformly for real $t\geq t_0(Y)$ and $T>2t+2$.
\end{proposition}

\begin{proof}[Proof of Theorem \ref{thm:TripleProduct} assuming Proposition
\ref{prop:K_inf_bound}]
The case in which $t$ is below the fixed threshold has already been
handled by the Bernstein--Reznikov estimate above.

Now suppose that $t\ge t_0(Y)$ is real. If $\tau\le 2t+2$, then the
exponential factor in Theorem~\ref{thm:TripleProduct} is bounded below
by an absolute constant, and \eqref{Sobolev} gives the desired estimate.
If $\tau>2t+2$, take $T=\tau$ in Proposition~\ref{prop:K_inf_bound}.
Since $h_T(\tau)\gg1$, Parseval gives the claimed bound.
\end{proof}

We defer the proof of Proposition \ref{prop:K_inf_bound} to Section \ref{subsec:Proof_Prop_K_inf_bound}, after establishing the necessary estimates in the remainder of this section.

Let $z' \in \mathbb{H}$ be a point where the maximum of $|K_T \abs{\phi}^2|$ is attained. Unfolding the kernel, we have:
$$ K_T \abs{\phi}^2(z') = \int_Y K_T(z', z)\abs{\phi}^2(z)d\mu_z = \int_{\mathbb{H}}k_T(z', z)\abs{\phi}^2(z)d\mu_z. $$
The eigenfunction $\phi \in L^2(Y)$ for $\Delta$ with eigenvalue $-(1/4+t^2)$ satisfies $\Delta\phi = - (\frac{1}{4} + t^2)\phi$. 
To evaluate this integral, we expand $\phi$ into geodesic polar coordinates $(r, \theta)$ centered at $z'$. We obtain the following:
$$ K_T \abs{\phi}^2(z') = 2\pi \int_0^\infty k_T(r) \left( \frac{1}{2\pi} \int_0^{2\pi} \abs{\phi(r, \theta)}^2 d\theta \right) \sinh r \, dr. $$
The inner integral represents the average of $\abs{\phi}^2$ over a circle of radius $r$, which we analyze using the following expansion.

\begin{lemma}[{Geodesic Polar Coordinates \cite[p. 142]{Terras_HarmonicAnal}}]
\label{lem:PolarExp}
Fix $z_0 \in \mathbb{H}$ and consider geodesic polar coordinates $(r, \theta)$ centered at $z_0$. Any \(L^2\)-eigenfunction \(\phi=\phi_t\) satisfying \(\Delta\phi=-(1/4+t^2)\phi\) admits the expansion
$$\phi(r, \theta)= \sum_{m\in \mathbb{Z}} a_\phi(m)C_t(|m|, \cosh r ) e^{im\theta},$$
where the normalized functions $C_t(|m|,x)$ are defined in terms of the associated Legendre function of the first kind (also referred to as the conical function) $P_{\nu}^{\mu}(x)$ by
$$C_t(|m|,x) = \frac{|\Gamma(1/2+|m|+it)|}{|\Gamma(1/2+it)|} P^{-|m|}_{-1/2+it}(x) \quad \text{for } x > 1.$$
\end{lemma}

Applying Parseval's identity on the circle to the expansion in Lemma
\ref{lem:PolarExp}, define first the circular average in the geodesic-radius
variable \(r\):
\begin{equation}\label{Adef}
A(r)
:=
\frac{1}{2\pi}\int_0^{2\pi}|\phi(r,\theta)|^2\,d\theta.
\end{equation}
Since the radial factor only depends on \(|m|\), it is convenient to group the
\(m\) and \(-m\) angular coefficients. For \(m\geq0\), define
\[
D_\phi(0):=|a_\phi(0)|^2,
\qquad
D_\phi(m):=|a_\phi(m)|^2+|a_\phi(-m)|^2
\quad(m\geq1).
\]
For real \(r\geq0\), Parseval gives
\begin{equation}\label{eq:A_real}
A(r)
=
\sum_{m\geq0}
D_\phi(m)|C_t(m,\cosh r)|^2.
\end{equation}

The Fourier coefficients $D_\phi(m)$ are controlled by the following $L^2$-estimate:
\begin{lemma}[{\cite[Lemma 3.4]{JJ_Sharp}}]\label{L2_coeff}
For all $T\geq 1$, we have
\begin{equation}\label{eqn:L2_Coeff}
\sum_{0\leq m < T} D_\phi(m) \ll_Y T+t.
\end{equation}
\end{lemma}

\begin{definition}[Holomorphic Extension of $A(r)$]
Let $S$ be a horizontal strip defined by
\[
    S:=\{r\in\mathbb C:|\Im r|<\pi/2\}.
\]
For \(m\geq0\), \(\sigma\in\mathbb R\), and \(r\in S\), define
\begin{equation}\label{eq:calC_def}
    \mathcal C_{\sigma,m}(r)
    :=
    \frac{|\Gamma(\frac12+m+i\sigma)|}{|\Gamma(\frac12+i\sigma)|\Gamma(m+1)}
    \tanh(r/2)^m
    {}_2F_1\left( \frac12-i\sigma,\frac12+i\sigma, m+1; -\sinh^2(r/2) \right).
\end{equation}
For real \(r>0\), this agrees with the normalized conical factor:
\[
\mathcal C_{\sigma,m}(r)=C_\sigma(m,\cosh r).
\]

We define the holomorphic extension of the circular average by
\begin{equation}\label{eq:mathcalB_def}
\mathcal B(r)
:=
\sum_{m\geq0}
D_\phi(m)
\mathcal C_{t,m}(r)\mathcal C_{-t,m}(r),
\qquad r\in S.
\end{equation}
For real \(r\geq0\), this agrees with \(A(r)\), since $
\mathcal C_{-t,m}(r)=\overline{\mathcal C_{t,m}(r)}.$
Thus
\[
\mathcal B(r)=A(r)
\qquad (r\geq0).
\]
\end{definition}

\begin{remark}
The function \(\mathcal B\) is holomorphic in \(S\). Indeed,
\(\tanh(r/2)\) is holomorphic and satisfies
\[
|\tanh(r/2)|<1
\]
on compact
subsets of \(S\), while the hypergeometric argument
\(-\sinh^2(r/2)\) avoids the branch cut \([1,\infty)\). The coefficient
bound in Lemma \ref{L2_coeff}, together with normal convergence of the
series on compact subsets of \(S\), gives holomorphicity. Moreover,
\(\mathcal B\) is even in \(r\), because
\[
\mathcal C_{\sigma,m}(-r)=(-1)^m\mathcal C_{\sigma,m}(r).
\]
\end{remark}

\subsection{Integral Representation and the Geometric Integral}

To evaluate the operator $K_T \abs{\phi}^2(z')$ explicitly, we use the relationship between the test function $h_T$ and the point-pair invariant $k_T$. Recall that for a localization $h_T(\tau)=e^{-(\tau-T)^2}+e^{-(\tau+T)^2}$, the associated function $g_T$ is given by the Fourier transform:
\begin{equation}
    g_T(u) = \frac{1}{2\pi}\int_{-\infty}^\infty h_T(\tau)e^{-i \tau u} d\tau = \frac{1}{\sqrt{\pi}}\cos(uT)e^{-u^2/4}.
\end{equation}
Following \eqref{eqn:Selberg/Harish-Chandra}, the point-pair invariant $k_T(r)$ is expressed via the transform $Q(e^u+e^{-u}-2) = g_T(u)$ as:
\begin{equation}\label{eq:k_integral_rep}
    k_T(r) = -\frac{1}{\pi \sqrt{2}}\int_r^\infty \frac{g'_T(w)}{\sqrt{\cosh w - \cosh r}} dw,
\end{equation}
where $r = d(z,z')$ is the hyperbolic distance. Differentiating $g_T$ yields:
\begin{equation}\label{derg}
    g'_T(w) = -\frac{1}{\sqrt{\pi}} \left( T \sin(wT) + \frac{w}{2}\cos(wT) \right) e^{-w^2/4}.
\end{equation}
Substituting the geodesic polar expansion of $\phi$ into the integral $I = \int_{\mathbb{H}} k_T(z', z)|\phi(z)|^2 d\mu_z$, we obtain:
\begin{align*}
    I &= 2\pi \int_0^\infty k_T(r) \mathcal{B}(r) \sinh r \, dr.
\end{align*}
Using the representation in \eqref{eq:k_integral_rep} and changing the order of integration, the integral $I$ becomes:
\begin{equation}\label{eq:I_H_abelform}
I
=
\sqrt{\frac2\pi}
\int_0^\infty
H(w)e^{-w^2/4}
\left(T\sin(Tw)+\frac w2\cos(Tw)\right)\,dw,
\end{equation}
where $H(w)$ is an Abel-type transform of the spectral sum $\mathcal{B}$ defined by
\begin{equation}\label{eq:H_def}
    H(w) := \int_0^w \frac{\mathcal{B}(r) \sinh r}{\sqrt{\cosh w - \cosh r}} dr.
\end{equation}

The strategy for proving Proposition \ref{prop:K_inf_bound} relies on shifting the integration in \eqref{eq:I_H_abelform} into the complex $w$-plane. This is justified by the following analytic properties of $H(w)$:

\begin{lemma}\label{lem:H_holomorphic}
    The function $H(w)$ is an odd function of $w$ and is holomorphic in the strip $|\Im(w)| < \pi/2$. 
\end{lemma}
\begin{proof}
    By the transformation $r = w\xi$, we may write
    \begin{equation}
        H(w) = \frac{1}{\sqrt{2}}\int_0^1 \frac{\mathcal{B}(w\xi) \sinh(w\xi)}{\sqrt{\frac{\sinh(w\frac{1+\xi}{2})}{w} \frac{\sinh(w\frac{1-\xi}{2})}{w}}} d\xi.
    \end{equation}
    The integrand is holomorphic for $|\Im(w)| < \pi/2$ and uniformly integrable in $\xi$. Since $\mathcal{B}(r)$ is an even function of $r$, it follows from the definition that $H(w)$ is odd.
\end{proof}

We will require sharp bounds for $H(w)$ on a shifted contour. To motivate the contour shift in the first place we prove the following crude bound for $\mathcal{B}(r).$

\begin{lemma}[Crude bound for \(\mathcal B\)]\label{trivialB}
For all $r\in \mathbb{C}$ with \(|\Im r|<\pi/2\), we have
\[
    |\mathcal B(r)|
    \ll_t
    \left(1-\left|\tanh\frac r2\right|^2\right)^{-1}.
\]
\end{lemma}

\begin{proof}
Put $
    q=\tanh(z/2)$ and $w=q^2.$ Since $|\Im z|<\pi/2$, we have $|q|<1$.

We first record the crude radial-factor estimate
\begin{equation}\label{eq:calC_crude}
|\mathcal C_{\sigma,m}(z)|
\ll_\sigma
(m+1)^{-1/2}|q|^m,
\qquad m\geq0.
\end{equation}
Indeed, by the hypergeometric formula for \(\mathcal C_{\sigma,m}\) (see \cite[Eq. 14.3.8]{NIST}) and
Pfaff's transformation (see \cite[Eq. 15.8.1]{NIST}), with \(a=\frac12-i\sigma\), we have
\[
\mathcal C_{\sigma,m}(z)
=
\frac{|\Gamma(\frac12+m+i\sigma)|}
     {|\Gamma(\frac12+i\sigma)|\Gamma(m+1)}
q^m
(1-w)^a
{}_2F_1(a,m+a;m+1;w).
\]
Euler's integral formula (see \cite[Eq. 15.6.1]{NIST}) gives
\[
{}_2F_1(a,m+a;m+1;w)
=
\frac{\Gamma(m+1)}
     {\Gamma(m+a)\Gamma(1-a)}
\int_0^1
u^{m+a-1}(1-u)^{-a}(1-wu)^{-a}\,du.
\]
For \(|w|<1\) and \(0\leq u\leq1\),
\[
|1-wu|\geq1-u,
\qquad
|1-w|\leq2|1-wu|.
\]
Hence
\[
\left|(1-w)^a(1-wu)^{-a}\right|\ll_\sigma1.
\]
Therefore
\[
\begin{aligned}
\left|
(1-w)^a{}_2F_1(a,m+a;m+1;w)
\right|
&\ll_\sigma
\frac{\Gamma(m+1)}
     {|\Gamma(m+a)|}
\int_0^1u^{m-1/2}(1-u)^{-1/2}\,du \\
&=
\frac{\Gamma(m+1)}
     {|\Gamma(m+\frac12-i\sigma)|}
\frac{\Gamma(m+\frac12)\Gamma(\frac12)}{\Gamma(m+1)}
\ll_\sigma1.
\end{aligned}
\]
Together with Stirling's formula,
\[
\frac{|\Gamma(\frac12+m+i\sigma)|}{\Gamma(m+1)}
\ll_\sigma
(m+1)^{-1/2},
\]
this proves \eqref{eq:calC_crude}.

Using the definition \eqref{eq:mathcalB_def} and applying
\eqref{eq:calC_crude} with \(\sigma=t\) and \(\sigma=-t\), we get
\[
|\mathcal B(z)|
\ll_t
\sum_{m\geq0}
D_\phi(m)
\frac{|q|^{2m}}{m+1}.
\]
By Lemma \ref{L2_coeff}, for fixed \(t\) we have $
D_\phi(m)\ll_t m+1.$ Thus
\[
|\mathcal B(z)|
\ll_t
\sum_{m\geq0}|q|^{2m}
=
\frac{1}{1-|q|^2}.
\]
Since \(q=\tanh(z/2)\), this gives the claimed bound.
\end{proof}

\subsection{
\texorpdfstring
{Uniform Estimates for $\mathcal{B}(z)$ via Dunster's Expansions}
{Uniform Estimates for B(z) via Dunster's Expansions}
}

The trivial bound for $\mathcal{B}(z)$ in Lemma \ref{trivialB} is sufficient for the initial contour shift, but to obtain sharp bounds for $H(w)$, we will need bounds uniform in the parameters $m$ and $t$, both of which can be large. The main ingredient to achieve the desired estimates for $\mathcal{B}(z)$ will be Dunster's work \cite{Dunster_Concial} on uniform asymptotic expansions of conical functions with both parameters large. After asymptotically expanding the conical function in terms of $J$-Bessel and $K$-Bessel functions, we will approximate those special functions in terms of exponentials. With a careful analysis of the involved arguments we can thus bound $P^{-\abs{m}}_{-1/2+it}(z),$ which in turn yields sharp bounds for $\mathcal{B}(z).$

\subsection{Conical factors on shifted contours}

Fix $T_0\geq2$ sufficiently large for the appendix estimates
with $A=B=1$, and enlarge $t_0(Y)$ so that
$2t_0(Y)+2\geq T_0$. In the remaining hard range $T>2t+2$,
we therefore have $T\geq T_0$. We set
\[
\sigma_T:=\frac{\pi}{2}-\frac1T,
\qquad
\mathcal L_T^\pm:=\{x\pm i\sigma_T:x\in\mathbb R\}.
\]
We write
\[
Z=\cosh r.
\]
Thus \(r\) denotes the complexified geodesic radius, while \(Z\) denotes the
argument of the conical function. See Figure \ref{fig:r and z planes}.
\begin{figure}
	\centering
    \begin{subfigure}[b]{0.3\textwidth}
        \centering
            \includegraphics[width=0.9\linewidth]{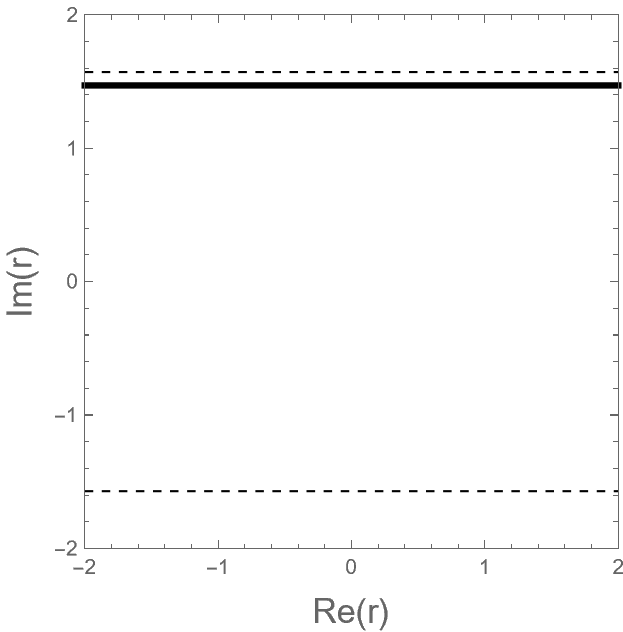}
    \end{subfigure}
    \hspace{20pt}
    \begin{subfigure}[b]{0.3\textwidth}
        \centering
            \includegraphics[width=0.9\linewidth]{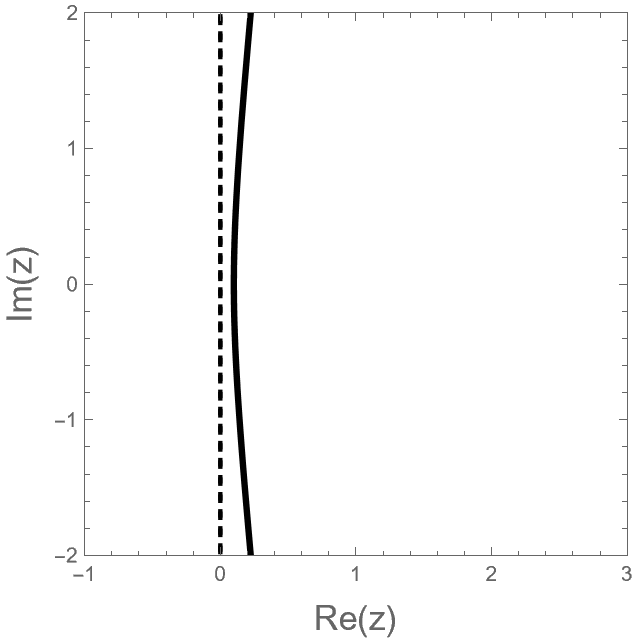}
    \end{subfigure}
    \caption{The $r$- and $z$-planes, where $z=Z=\cosh(r)$. The locus $\Im(r)=\pi/2-0.1$ is drawn in bold.}
    \label{fig:r and z planes}
\end{figure}

For the remainder of this section, we work in the range
$t\geq t_0(Y)\geq1$ and $T\geq T_0$ fixed above.
For $m\geq1$, the $J$-range is $1\leq m\leq t$, while the
$K$-range is $m>t$. The angular mode $m=0$ is treated separately
below. Bounded real and complementary spectral parameters are
covered by the reduction at the beginning of the section.

For real $m,t>0$ and $r\in\mathcal L_T^+\cup\mathcal L_T^-$,
put $Z=\cosh r$ and define
\begin{equation}\label{eq:p_R_def}
\begin{aligned}
p_{m,t}(r)&:=\sqrt{m^2-t^2\sinh^2 r} \quad \text{and} \quad 
R_{m,t}(Z)&:=
\left|\frac{mZ-p_{m,t}(r)}{mZ+p_{m,t}(r)}\right|.
\end{aligned}
\end{equation}
The square root is chosen to have positive real part. This is
well-defined on the shifted contours, since
$\Re(m^2-t^2\sinh^2 r)>0$ for $T\geq2$. The definitions allow
real $m>0$; the angular modes themselves remain integers.

For the phase functions below, put
\[
s_\alpha(Z):=\sqrt{Z^2-1-\alpha^2},
\qquad
Z_\beta:=\sqrt{1+\beta^{-2}},
\qquad
q_\beta(Z):=\sqrt{1+\beta^2-\beta^2Z^2}.
\]
The branch of $s_\alpha$ is positive for real
$Z>\sqrt{1+\alpha^2}$, and the branch of $q_\beta$ is positive
for real $0<Z<Z_\beta$. They are continued into the two open
quadrants $\Re Z>0$, $\pm\Im Z>0$. On the shifted contours,
$q_\beta$ has positive real part.

The phase functions in Lemma~\ref{lem:Dunster_mappings} are
normalized by $I_\alpha(\sqrt{1+\alpha^2})=0$ and
$U_\beta(Z_\beta)=0$. Their logarithms and inverse trigonometric
functions are continued from the corresponding real intervals
$Z>\sqrt{1+\alpha^2}$ and $1<Z<Z_\beta$, respectively.
Boundary values on a cut are always taken from the specified quadrant.

We first apply the special-function estimates with
$r=x\pm i\sigma_T$ and $x\geq0$. For $x>0$, the point
$Z=\cosh r$ lies in the upper half-plane for the plus sign and
in the lower half-plane for the minus sign. At the endpoints
$r=\pm i\sigma_T$, where $Z=\cos\sigma_T$, every branch-dependent
function $G(Z)$ is interpreted by the corresponding one-sided boundary value
\[
G^\pm(\cos\sigma_T)
:=
\lim_{x\downarrow0}G\bigl(\cosh(x\pm i\sigma_T)\bigr).
\]
The final absolute-value estimates on the full shifted contours
will follow from the symmetry $r\mapsto-r$.

\begin{proposition}[Conical factor bounds on the shifted contours]
\label{prop:C_shifted_bounds}
Let $r\in\mathcal L_T^+\cup\mathcal L_T^-$ and put $Z=\cosh r$.
Assume $T\geq T_0$, $t\geq1$, and $m\geq1$ is an integer. Then
\begin{equation}\label{eq:C_shifted_unified}
\begin{aligned}
|\mathcal C_{t,m}(r)\mathcal C_{-t,m}(r)|
\ll e^{\pi t}\frac{R_{m,t}(Z)^m}{|p_{m,t}(r)|}
\ll e^{\pi t}\frac{R_{m,t}(Z)^m}{m+t}.
\end{aligned}
\end{equation}
\end{proposition}

The mode $m=0$ is estimated separately below. The proof of Proposition \ref{prop:C_shifted_bounds} is given at the end of this subsection, after the necessary uniform expansion of conical functions due to Dunster, and some estimates for Bessel functions.

\begin{lemma}[{\cite[eqs.~3.1 and 4.14]{Dunster_Concial} and Appendix \ref{AppendixConical}}]
\label{lem:Dunster_conical}
Fix $A,B>0$. Let $m\geq1$, $t\geq1$, and
$T\geq T_0(A,B)$, where $T_0(A,B)$ is sufficiently large.
Let $r=x\pm i\sigma_T$ with $x\geq0$, and put $Z=\cosh r$.

\begin{enumerate}
\item[\textup{(J)}]
If \(\alpha=m/t\le A\), then
\[
P^{-m}_{-1/2+it}(Z)
=
\frac{|\Gamma(\frac12+it)|}{|\Gamma(\frac12+m+it)|}
\left(\frac{\zeta-\alpha^2}{Z^2-1-\alpha^2}\right)^{1/4}
J_m(t\zeta^{1/2})
\left(1+O_A(t^{-1})\right).
\]

\item[\textup{(K)}]
If \(\beta=t/m\le B\), then
\[
P^{-m}_{-1/2+it}(Z)
=
\left(\frac2\pi\right)^{1/2}
\frac{1}{|\Gamma(\frac12+m+it)|}
\left(\frac{\eta-\beta^2}{1+\beta^2-\beta^2Z^2}\right)^{1/4}
K_{it}(m\eta^{1/2})
\left(1+O_B(m^{-1})\right).
\]
\end{enumerate}
The parameters \(\zeta\) and \(\eta\) are defined by the integral
relations below. The lower-contour estimates follow from the
upper-contour estimates by complex conjugation.
\end{lemma}

\begin{lemma}[Phase functions]
\label{lem:Dunster_mappings}
Let \(\alpha,\beta>0\).

For the \(J\)-range, define
\[
I_\alpha(Z)
:=
\int_{\sqrt{1+\alpha^2}}^Z
\frac{(u^2-1-\alpha^2)^{1/2}}{u^2-1}\,du.
\]
Then \(\zeta\) is determined by
\[
\int_{\alpha^2}^{\zeta}
\frac{(\xi-\alpha^2)^{1/2}}{2\xi}\,d\xi
=
I_\alpha(Z).
\]
Moreover, on the chosen branch,
\[
I_\alpha(Z)
=
\log\bigl(s_\alpha(Z)+Z\bigr)
-\frac12\log(1+\alpha^2)
-\alpha\arctan\left(\frac{s_\alpha(Z)}{\alpha Z}\right),
\]
where \(s_\alpha(Z)=\sqrt{Z^2-1-\alpha^2}\).

For the \(K\)-range, put $
\xi=(Z^2-1)^{-1}$
and define
\[
U_\beta(Z)
:=
\int_{\beta^2}^{\xi}
\frac{(s-\beta^2)^{1/2}}{2s(1+s)^{1/2}}\,ds.
\]
Then the variable \(\eta\) is determined by
\begin{equation}\label{UBeta}
\int_{\beta^2}^{\eta}
\frac{(s-\beta^2)^{1/2}}{2s}\,ds
=
U_\beta(Z).
\end{equation}
Finally, on the chosen branch,
\begin{equation}\label{eq:Ubeta_explicit}
U_\beta(Z)
=
\frac12
\log\left(\frac{Z+q_\beta(Z)}{Z-q_\beta(Z)}\right)
-
\beta\arctan\left(\frac{q_\beta(Z)}{\beta Z}\right).
\end{equation}
\end{lemma}

\begin{proof}
The formula for \(I_\alpha\) is checked by differentiating the right-hand side and noting that both sides vanish at \(Z=\sqrt{1+\alpha^2}\). For \(U_\beta\), use the substitution
\[
v=\left(\frac{s-\beta^2}{1+s}\right)^{1/2}.
\]
Then
\[
s=\frac{\beta^2+v^2}{1-v^2},
\qquad
ds=\frac{2v(1+\beta^2)}{(1-v^2)^2}\,dv,
\]
and hence
\[
\frac{(s-\beta^2)^{1/2}}{2s(1+s)^{1/2}}\,ds
=
\left(\frac{1}{1-v^2}
-
\frac{\beta^2}{\beta^2+v^2}\right)\,dv.
\]
We can write
\[
U_\beta(Z)
=
\operatorname{arctanh}v
-
\beta\arctan\left(\frac v\beta\right).
\]
At the endpoint \(s=\xi=(Z^2-1)^{-1}\), one has
\[
v^2
=
\frac{\xi-\beta^2}{1+\xi}
=
\left(\frac{q_\beta(Z)}{Z}\right)^2.
\]
Using the branch for which \(v=q_\beta(Z)/Z\) on \(1<Z<Z_\beta\), and then
continuing along the chosen contour, gives \eqref{eq:Ubeta_explicit}.
\end{proof}

Now that we have approximated the conical functions in terms of $J$-Bessel and $K$-Bessel functions, we will give sharp upper bounds for those special functions in terms of exponentials. We are estimating the special functions on a contour that avoids the transition point, where such expansions would otherwise become problematic.

\begin{lemma}[\(J\)-Bessel contribution]
\label{lem:J_bessel_contribution}
Let $m\geq1$, $t\geq1$, and $\alpha=m/t\in(0,A]$.
Assume $T\geq T_0(A)$, let $r=x\pm i\sigma_T$ with $x\geq0$,
and put $Z=\cosh r$. Then
\[
\left|
\left(\frac{\zeta-\alpha^2}{Z^2-1-\alpha^2}\right)^{1/4}
J_m(t\zeta^{1/2})
\right|
\ll_A
\frac{1}{\sqrt m}
\frac{\alpha^{1/2}}{|Z^2-1-\alpha^2|^{1/4}}
\left|\exp(\mp i t I_\alpha(Z))\right|,
\]
where the minus sign is used for $r=x+i\sigma_T$, and the plus sign
for $r=x-i\sigma_T$.
\end{lemma}

\begin{proof}

Recall that \(t=m/\alpha\). We use the exponential form of the
large-order \(J\)-Bessel expansion with simplified error terms by \cite[Eq. 5.19]{DunsterGilSegura_BesselbyAiry} (see also Lemma \ref{lem:J-Bessel by Exp}):
\begin{equation}\label{eq:bessel_debye}J_m\left(\frac{m}{\alpha}\zeta^{1/2}\right) = \left(\frac{\alpha^2}{\alpha^2-\zeta}\right)^{1/4} \frac{1}{\sqrt{2\pi m}} \exp\left(-m \int_{\zeta^{1/2}/\alpha}^1\frac{\sqrt{1-u^2}}{u} du\right)(1+O_A(m^{-1})).\end{equation}
We focus on the integral in the exponent. Using the substitution $\xi = \alpha^2 u^2$, we get
\begin{equation}\label{aftersubstitution}
    -\int_{\zeta^{1/2}/\alpha}^1 \frac{\sqrt{1-u^2}}{u} \, du =  \frac{1}{\alpha}\int^{\zeta}_{\alpha^2} \frac{\sqrt{\alpha^2-\xi}}{2\xi} \, d\xi.
\end{equation}

For $x>0$, the upper and lower continuations give
$\sqrt{\alpha^2-\xi}=\mp i\sqrt{\xi-\alpha^2}$, respectively.
Thus \eqref{aftersubstitution} equals
\begin{equation}
\mp \frac{i}{\alpha} \int_{\alpha^2}^{\zeta} \frac{\sqrt{\xi-\alpha^2}}{2\xi} \, d\xi = \mp \frac{i}{\alpha}I_{\alpha}(Z).
\end{equation}

Substituting this result back into the exponent of \eqref{eq:bessel_debye}, we get $\exp(\mp i \frac{m}{\alpha} I_{\alpha}(Z))$. The final result follows upon taking into account the pre-factor $\left(\frac{\zeta-\alpha^2}{Z^2-1-\alpha^2}\right)^{1/4}$, taking into account the error and putting absolute values.

\end{proof}

\begin{lemma}[\(K\)-Bessel contribution]
\label{lem:K_bessel_contribution}
Let $m\geq1$, $t\geq1$, and $\beta=t/m\in(0,1]$.
Let $r=x\pm i\sigma_T$ with $x\geq0$, and put $Z=\cosh r$.
We have
\[
\left|
\left(\frac{\eta-\beta^2}{1+\beta^2-\beta^2Z^2}\right)^{1/4}
K_{it}(m\eta^{1/2})
\right|
\ll
\frac{e^{-\pi t/2}}{\sqrt t}
\frac{\beta^{1/2}}{|1+\beta^2-\beta^2Z^2|^{1/4}}
|\exp(-mU_\beta(Z))|.
\]
\end{lemma}

\begin{proof}
   The proof is similar to the proof of Lemma \ref{lem:J_bessel_contribution}.
   We recall that $m=t/\beta$. Using the asymptotic expansion of the $K$-Bessel function given in Lemma \ref{KBesselExp}, we obtain

        $$K_{it}\Big(t \frac{\eta^{1/2}}{\beta}\Big) = \Big(\frac{\pi}{2t}\Big)^{1/2}\Big(\frac{\beta^{1/2}}{(\eta-\beta^2)^{1/4}}\Big)e^{-\frac{\pi}{2}t}\exp\Big(-t \int_1^{\eta^{\frac{1}{2}}/\beta} \frac{\sqrt{s^2-1}}{s}ds\Big)(1+O(t^{-1})).$$
        Focusing on the integral in the exponential we see that 
        $$\int_1^{\eta^{\frac{1}{2}}/\beta} \frac{\sqrt{s^2-1}}{s}ds=\frac{1}{\beta}\int_{\beta^2}^\eta \frac{(s-\beta^2)^{1/2}}{2s}ds.$$ Using the integral relation \eqref{UBeta} we further simplify        $$\frac{1}{\beta}\int_{\beta^2}^\eta \frac{(s-\beta^2)^{1/2}}{2s}ds =\frac{1}{\beta}U_\beta(Z).$$
        
        The lemma follows by applying absolute values, absorbing the error term in the implicit constant and simplifying. 
\end{proof}

We now analyze the exponential factor arising in Lemma \ref{lem:J_bessel_contribution} in more detail.

\begin{lemma}[\(J\)-range exponential factor]
\label{lem:J_exponential_ratio}
Let $\alpha,t>0$, put $m=\alpha t$, and let
$r=x+i\varepsilon\sigma_T$, where $x\geq0$ and
$\varepsilon\in\{-1,1\}$. Put $Z=\cosh r$ and use the
corresponding branch of $I_\alpha$. Then
\[
\left|\exp(-i\varepsilon t I_\alpha(Z))\right|
\leq e^{\pi t/2}R_{m,t}(Z)^{m/2}.
\]
\end{lemma}

\begin{proof}
Suppose first that $x>0$, and write $s=s_\alpha(Z)$.
The points $Z$ and $s$ lie in the first quadrant if
$\varepsilon=1$, and in the fourth quadrant if $\varepsilon=-1$.
Consequently,
\[
p_{m,t}(r)=-i\varepsilon t s,
\qquad
R_{m,t}(Z)=
\left|\frac{\alpha Z+i\varepsilon s}
           {\alpha Z-i\varepsilon s}\right|.
\]
The logarithm of $Z+s$ in the formula for $I_\alpha$ is its
principal branch, since $Z+s$ stays in the right half-plane.
Using the logarithmic formula for the arctangent, in the form
\[
\Im\arctan v
=-\frac12\log\left|\frac{1+iv}{1-iv}\right|,
\]
we obtain
\[
\varepsilon\Im I_\alpha(Z)
=\varepsilon\Arg(Z+s)+\frac{\alpha}{2}\log R_{m,t}(Z).
\]
Since $0<\varepsilon\Arg(Z+s)<\pi/2$, exponentiation gives
\[
\left|\exp(-i\varepsilon t I_\alpha(Z))\right|
=\exp\bigl(\varepsilon t\Im I_\alpha(Z)\bigr)
\leq e^{\pi t/2}R_{m,t}(Z)^{m/2}.
\]
At $x=0$, writing $c_T=\cos\sigma_T$, the boundary values are
\[
s_\alpha^\varepsilon(c_T)
=i\varepsilon\sqrt{1+\alpha^2-c_T^2}.
\]
Both give the same positive value of $p_{m,t}(i\varepsilon\sigma_T)$
and hence the same $R_{m,t}(c_T)$. Thus the estimate at $x=0$ follows by taking the one-sided limit from the upper or lower quadrant, according to the sign of
$\varepsilon$.
\end{proof}

The following elementary geometric observation will be useful for controlling
the relevant branches in the $K$-Bessel range and, later, the ratio
$R_{m,t}(Z)$.

\begin{lemma}[Square-root geometry in the right half-plane]
\label{lem:sqrt_geometry}
Let \(c>0\), and let
\[
p_c(Z):=\sqrt{c^2-Z^2}
\]
be the branch which is positive on \(0<Z<c\), with boundary values obtained
by continuation from the upper or lower half-plane. If \(\Re Z>0\) and
\(Z\notin[c,\infty)\), then
\[
\Re\left(\frac{Z}{p_c(Z)}\right)>0.
\]
Consequently,
\[
\Re\left(\frac{p_c(Z)}{Z}\right)>0.
\]
For the branch of \(\arctan\) obtained by continuation from the positive real
axis, this implies
\[
\left|\Re\arctan\left(\frac{p_c(Z)}{Z}\right)\right|
\leq \frac{\pi}{2}.
\]
The same assertions hold for the corresponding boundary values on
\(0<Z<c\).
\end{lemma}

\begin{proof}
Write \(Z=\rho e^{i\theta}\), with
\(-\pi/2<\theta<\pi/2\). If \(0<\theta<\pi/2\), then
\[
c^2-Z^2=c^2+\rho^2e^{i(2\theta-\pi)}
\]
has argument between \(2\theta-\pi\) and \(0\). Hence
\[
\theta-\frac{\pi}{2}<\Arg p_c(Z)<0,
\]
and therefore
\[
0<\Arg\left(\frac{Z}{p_c(Z)}\right)<\frac{\pi}{2}.
\]
If \(-\pi/2<\theta<0\), the same argument gives
\[
-\frac{\pi}{2}
<
\Arg\left(\frac{Z}{p_c(Z)}\right)
<
0.
\]
The real interval \(0<Z<c\) is immediate from the chosen branch. This proves
\(\Re(Z/p_c(Z))>0\). Taking reciprocals gives
\(\Re(p_c(Z)/Z)>0\). Finally, the branch of \(\arctan\) continued from the
positive real axis maps the right half-plane into the strip
\(|\Re w|<\pi/2\), giving the last assertion.
\end{proof}

Next, we analyze the case where $0 < \beta = t/m$ is bounded, i.e. the conical function is approximated by the $K$-Bessel function. 

\begin{lemma}[The \(K\)-Bessel exponential factor]
\label{lem:K_exponential_ratio}
Let $0<\beta\leq1$, let $t>0$, and put $m=t/\beta$.
Let $r=x\pm i\sigma_T$ with $x\geq0$, put $Z=\cosh r$, and
use the corresponding branch of $U_\beta$. Then
\[
|\exp(-mU_\beta(Z))|
\leq e^{\pi t/2}R_{m,t}(Z)^{m/2}.
\]
\end{lemma}

\begin{proof}
On the chosen branches, $p_{m,t}(r)=m q_\beta(Z)$, so
\[
R_{m,t}(Z)=\left|\frac{Z-q_\beta(Z)}{Z+q_\beta(Z)}\right|.
\]
Taking real parts in \eqref{eq:Ubeta_explicit} gives
\[
-\Re U_\beta(Z)
=\frac12\log R_{m,t}(Z)
+\beta\Re\arctan\left(\frac{q_\beta(Z)}{\beta Z}\right).
\]
Here $\Re\log w=\log|w|$ is independent of the logarithm branch.
Moreover, $q_\beta(Z)/\beta=\sqrt{Z_\beta^2-Z^2}$, so
Lemma~\ref{lem:sqrt_geometry} gives
\[
\Re\left(\frac{q_\beta(Z)}{\beta Z}\right)>0.
\]
The arctangent is the branch agreeing with the real arctangent
on $1<Z<Z_\beta$. Since its argument remains in the right
half-plane, this is the principal branch, and its real part
lies between $0$ and $\pi/2$. Thus
\[
-\Re U_\beta(Z)
\leq\frac12\log R_{m,t}(Z)+\frac{\pi\beta}{2}.
\]
Exponentiating and using $m\beta=t$ proves the estimate.
At $x=0$, the estimate follows by taking the one-sided limit from
the upper or lower half-plane. Since only $\Re\log w=\log|w|$ enters the estimate, the two boundary values give the same bound.
\end{proof}

\begin{proof}[Proof of Proposition \ref{prop:C_shifted_bounds}]
The defining hypergeometric formula \eqref{eq:calC_def} gives
\[
\mathcal C_{-t,m}(r)=\mathcal C_{t,m}(r),
\qquad
\mathcal C_{t,m}(-r)=(-1)^m\mathcal C_{t,m}(r).
\]
Also, $p_{m,t}(-r)=p_{m,t}(r)$ and $\cosh(-r)=\cosh r$.
We may therefore first assume $\Re r\geq0$.

Suppose $1\leq m\leq t$, and put $\alpha=m/t$. Combining
Lemma~\ref{lem:Dunster_conical} in the $J$-range with
Lemmas~\ref{lem:J_bessel_contribution} and
\ref{lem:J_exponential_ratio}, we obtain
\[
|\mathcal C_{t,m}(r)|
\ll e^{\pi t/2}
\frac{R_{m,t}(Z)^{m/2}}
     {\sqrt t\,|Z^2-1-\alpha^2|^{1/4}}
=e^{\pi t/2}
\frac{R_{m,t}(Z)^{m/2}}{|p_{m,t}(r)|^{1/2}}.
\]

If $m>t$, put $\beta=t/m$. Combining
Lemma~\ref{lem:Dunster_conical} in the $K$-range with
Lemmas~\ref{lem:K_bessel_contribution} and
\ref{lem:K_exponential_ratio}, and using
$|\Gamma(\tfrac12+it)|^2=\pi/\cosh(\pi t)$, we obtain
\[
|\mathcal C_{t,m}(r)|
\ll e^{\pi t/2}
\frac{R_{m,t}(Z)^{m/2}}
     {\sqrt m\,|1+\beta^2-\beta^2Z^2|^{1/4}}
=e^{\pi t/2}
\frac{R_{m,t}(Z)^{m/2}}{|p_{m,t}(r)|^{1/2}}.
\]
Squaring and using the symmetry $r\mapsto-r$ proves the first
bound in \eqref{eq:C_shifted_unified} on the full shifted contours.

Finally, writing $r=x\pm i\sigma_T$ and using $T\geq2$, we have
\[
\Re(\sinh^2r)
=\cos(2\sigma_T)\sinh^2x-\sin^2\sigma_T
\leq-\cos^2(1/2).
\]
Consequently,
\[
|p_{m,t}(r)|^2
\geq\Re\bigl(p_{m,t}(r)^2\bigr)
\geq m^2+\cos^2(1/2)t^2
\gg(m+t)^2.
\]
This gives the second bound in \eqref{eq:C_shifted_unified}.
\end{proof}

\begin{lemma}[Ratio bounds for the summation]\label{lem:ratio_bounds_for_B}
Let $r\in\mathcal L_T^+\cup\mathcal L_T^-$, put $Z=\cosh r$, and
fix $t>0$. For every real $u>0$, we have
\[
0<R_{u,t}(Z)<1,
\]
and $u\mapsto R_{u,t}(Z)$ is strictly decreasing. In particular,
for every integer $m\geq1$,
\[
R_{m,t}(Z)^m\leq
\begin{cases}
1, & 1\leq m\leq t,\\
R_{t,t}(Z)^m, & m>t,
\end{cases}
\]
where $0<R_{t,t}(Z)<1$.
\end{lemma}

\begin{proof}
For real $u>0$, write $p=p_{u,t}(r)$ and
$c=\sqrt{1+(u/t)^2}$. Since $p=t\sqrt{c^2-Z^2}$ on the chosen
branch, Lemma~\ref{lem:sqrt_geometry} gives $\Re(Z/p)>0$. Hence
\[
|uZ+p|^2-|uZ-p|^2=4u\Re(Z\overline p)>0,
\]
which gives $R_{u,t}(Z)<1$. Also, $Z^2\neq1$ on the shifted
contours, so
\[
p^2-u^2Z^2=(u^2+t^2)(1-Z^2)\neq0.
\]
Thus $R_{u,t}(Z)>0$ and its logarithm is well-defined.

Keeping $t$ and $r$ fixed, we now differentiate the elementary ratio
with respect to the real variable $u$. Since $\Re(p_{u,t}(r)^2)>0$
for all $u>0$, the chosen square root is smooth in $u$ and
$\partial p/\partial u=u/p$. Therefore
\[
\begin{aligned}
\frac{\partial}{\partial u}\log R_{u,t}(Z)
&=\Re\left(
\frac{Z-u/p}{uZ-p}-\frac{Z+u/p}{uZ+p}
\right)\\
&=-\frac{2t^2}{u^2+t^2}\Re\left(\frac{Z}{p}\right)<0.
\end{aligned}
\]
Finally,
\[
R_{t,t}(Z)
=
\left|
\frac{Z-\sqrt{2-Z^2}}
     {Z+\sqrt{2-Z^2}}
\right|.
\]
The two bounds for $R_{m,t}(Z)^m$ follow from
$R_{m,t}(Z)<1$ and monotonicity, respectively.
\end{proof}

\begin{lemma}[Horizontal bound for \(\mathcal B\)]\label{Bhor}
Let \(r=x+iy\), where
$
|y|=\frac{\pi}{2}-\frac1T.
$
Then
\[
|\mathcal B(r)|
\ll
e^{\pi |t|}\bigl(|x|+\log T+1\bigr).
\]
\end{lemma}

\begin{proof}
Put \(Z=\cosh r\). By the definition of \(\mathcal B\),
\[
|\mathcal B(r)|
\leq
\sum_{m\geq 0}
D_\phi(m)
|\mathcal C_{t,m}(r)\mathcal C_{-t,m}(r)|.
\]
We split $
|\mathcal B(r)|\leq S_r^1+S_r^2,$
where
\[
S_r^1
:=
D_\phi(0)|\mathcal C_{t,0}(r)|^2
+
\sum_{1\leq m\leq t}
D_\phi(m)|\mathcal C_{t,m}(r)|^2,
\]
and
\[
S_r^2
:=
\sum_{m>|t|}
D_\phi(m)
|\mathcal C_{t,m}(r)\mathcal C_{-t,m}(r)|.
\]

For the zero mode, the transformation with $\alpha=0$ gives
$\zeta=r^2$. Thus Lemma~\ref{lem:Conical by JBessel} yields
\[
    |\mathcal C_{t,0}(r)|
    \ll
    \left|\frac{r}{\sinh r}\right|^{1/2}|J_0(tr)|.
\]
The standard fixed-order Bessel estimate
\[
    |J_0(z)|\ll |z|^{-1/2}e^{|\Im z|},
    \qquad |\arg z|\leq\pi/2,\quad |z|\geq1,
\]
follows from \cite[Section 10.17]{NIST}.
Applying this for $\Re r\geq0$ and then using the evenness of
$\mathcal C_{t,0}$ for $\Re r<0$, we obtain
\[
    |\mathcal C_{t,0}(x\pm i\sigma_T)|^2
    \ll
    \frac{e^{2t\sigma_T}}{t|\sinh(x\pm i\sigma_T)|}
    \ll
    \frac{e^{\pi t}}{1+t}.
\]
Since $D_\phi(0)=|\phi(z_0)|^2\ll_Y1+t$, the zero mode contributes
$O_Y(e^{\pi t})$.

By Proposition~\ref{prop:C_shifted_bounds},
Lemma~\ref{lem:ratio_bounds_for_B}, and the zero-mode estimate above,
\[
S_r^1
\ll_Y e^{\pi t}
\left(1+\frac1t\sum_{1\leq m\leq t}D_\phi(m)\right)
\ll_Y e^{\pi t},
\]
where the last step uses Lemma~\ref{L2_coeff}.

For the remaining modes, the same proposition and ratio lemma give
\[
S_r^2
\ll e^{\pi t}\sum_{m>t}
D_\phi(m)\frac{R_{m,t}(Z)^m}{m+t}
\leq e^{\pi t}\sum_{m>t}
D_\phi(m)\frac{R_{t,t}(Z)^m}{m}.
\]
By Lemma~\ref{L2_coeff},
\[
\sum_{t<m\leq u}D_\phi(m)\ll_Y u
\qquad(u\geq1),
\]
where the sum is empty when $u\leq t$. Since
$0<R_{t,t}(Z)<1$, partial summation therefore yields
\[
S_r^2
\ll_Y e^{\pi t}\sum_{m\geq1}\frac{R_{t,t}(Z)^m}{m}
=
e^{\pi t}\log\left(\frac1{1-R_{t,t}(Z)}\right).
\]

It remains to estimate \(1-R_{t,t}(Z)\). Write
\[
M(r):=R_{t,t}(Z)
=
\left|
\frac{Z-\sqrt{2-Z^2}}
     {Z+\sqrt{2-Z^2}}
\right|,
\qquad Z=\cosh r.
\]
By Lemma \ref{lem:ratio_bounds_for_B}, \(M(r)<1\). Hence
\[
1-M(r)\geq \frac{1-M(r)^2}{2}.
\]
If
\[
q:=\frac{p_{t,t}(r)}{t}=\sqrt{2-Z^2},
\]
where the square root has positive real part, then
\[
1-M(r)^2
=
\frac{4\Re(Z\overline q)}{|Z+q|^2}.
\]
Writing \(\delta=1/T\), so that
\(\sigma_T=\pi/2-\delta\), we have
\[
Z=\cosh x\sin\delta \pm i\sinh x\cos\delta .
\]
A direct calculation from \(q^2=2-Z^2\), with \(\Re q>0\), gives
\[
|Z+q|^2\ll e^{2|x|},
\qquad
\Re(Z\overline q)\gg \sin\delta\gg T^{-1}.
\]
Thus $
1-M(r)\gg T^{-1}e^{-2|x|}$ and consequently,
\[
\log\left(\frac1{1-R_{t,t}(Z)}\right)
\ll
|x|+\log T+1.
\]
We therefore obtain the bound
\[
S_r^2
\ll
e^{\pi |t|}
\bigl(|x|+\log T+1\bigr).
\]
Combining this with the estimate for \(S_r^1\) proves the lemma.
\end{proof}

\begin{lemma}[Vertical contribution to \(H\)]\label{lem:H_vertical_direct}
Let
\[
\sigma_T:=\frac{\pi}{2}-\frac1T,
\qquad
w=x+i\sigma_T.
\]
Define
\[
H_{\mathrm{vert}}(w)
:=
-1\int_0^{\sigma_T}
\frac{\mathcal B(iy)\sin y}
     {\sqrt{\cosh w-\cos y}}
\,dy.
\]
Then
\[
|H_{\mathrm{vert}}(w)|
\ll
e^{\pi |t|}(1+\log T).
\]
The same bound holds for the lower shifted line.
\end{lemma}

\begin{proof}
Put $\sigma=\sigma_T$ and $c_T=\cos\sigma$.
We first use the maximum-modulus principle and deduce a bound for $\mathcal B$ on the vertical segment
from Lemma~\ref{Bhor}. Indeed, the function $F(r):=e^{-r^2}\mathcal B(r)$ 
is holomorphic in $|\Im r|<\pi/2$. On the two lines
$r=u\pm i\sigma$, Lemma~\ref{Bhor} gives
\[
    |F(u\pm i\sigma)|
    \ll_Y
    e^{\pi t}e^{-u^2+\sigma^2}
    (|u|+\log T+1)
    \ll_Y e^{\pi t}(1+\log T).
\]
Moreover, for $|v|\leq\sigma$,
\[
    1-\left|\tanh\frac{u+iv}{2}\right|^2
    =
    \frac{2\cos v}{\cosh u+\cos v}.
\]
Hence Lemma~\ref{trivialB} implies
\[
    |\mathcal B(u+iv)|\ll_{t,T}e^{|u|}.
\]
Consequently, $F(\pm R+iv)\to0$ uniformly for $|v|\leq\sigma$
as $R\to\infty$.

Applying the maximum-modulus principle on rectangles and then
letting $R\to\infty$, we obtain
\[
    |\mathcal B(iy)|
    =e^{-y^2}|F(iy)|
    \ll_Y e^{\pi t}(1+\log T),
    \qquad |y|\leq\sigma.
\]

For $w=x+i\sigma$ and $0\leq y\leq\sigma$, we have
\[
    |\cosh w-\cos y|\geq \cos y-c_T.
\]
Indeed, the square of the left-hand side is increasing as a
function of $\cosh x\geq1$, and its value at $x=0$ is
$(\cos y-c_T)^2$. Therefore
\[
\begin{aligned}
    |H_{\mathrm{vert}}(w)|
    &\ll_Y e^{\pi t}(1+\log T)
       \int_0^\sigma
       \frac{\sin y}{(\cos y-c_T)^{1/2}}\,dy \\
    &=2e^{\pi t}(1+\log T)(1-c_T)^{1/2}\\
    &\ll_Y e^{\pi t}(1+\log T).
\end{aligned}
\]
The lower shifted line is treated in the same way.
\end{proof}

\begin{lemma}[Shifted-line bound for \(H\)]\label{lem:H_shifted_bound}
Let
\[
\sigma_T:=\frac{\pi}{2}-\frac1T,
\qquad
w=x+i\sigma_T.
\]
Then
\[
|H(w)|
\ll
e^{\pi |t|}e^{|x|}
\bigl(|x|+\log T+1\bigr).
\]
The same bound holds on the lower shifted line.
\end{lemma}

\begin{proof}
We prove the upper-line estimate. The lower-line estimate is identical.

Deform the path defining \(H(w)\) into the vertical segment from \(0\) to
\(i\sigma_T\) and the horizontal segment from \(i\sigma_T\) to
\(x+i\sigma_T\):
\[
H(w)=H_{\mathrm{vert}}(w)+H_{\mathrm{hor}}(w).
\]
By Lemma \ref{lem:H_vertical_direct},
\[
|H_{\mathrm{vert}}(w)|
\ll
e^{\pi |t|}(1+\log T).
\]

For the horizontal part, write \(r=u+i\sigma_T\), with \(0\leq u\leq x\) if
\(x\geq0\), and reverse the path if \(x<0\). Using Lemma \ref{Bhor}, the
factorization
\[
\cosh w-\cosh r
=
2\sinh\left(\frac{w+r}{2}\right)
 \sinh\left(\frac{w-r}{2}\right),
\]
and the estimates
\[
\left|\sinh\left(\frac{w+r}{2}\right)\right|\gg1,
\qquad
|\sinh(u+i\sigma_T)|\ll e^{|u|},
\]
one obtains
\[
|H_{\mathrm{hor}}(w)|
\ll
e^{\pi |t|}
\int_0^{|x|}
\frac{(u+\log T+1)e^u}
     {|\sinh((|x|-u)/2)|^{1/2}}
\,du .
\]
As before, the last integral is
\[
\ll
e^{|x|}\bigl(|x|+\log T+1\bigr),
\]
because the singularity at \(u=|x|\) is integrable and the remaining part is
exponentially decaying after the change of variables \(v=|x|-u\). Hence
\[
|H_{\mathrm{hor}}(w)|
\ll
e^{\pi |t|}e^{|x|}
\bigl(|x|+\log T+1\bigr).
\]
Combining the vertical and horizontal estimates gives the claim.
\end{proof}

\subsection{Proof of Proposition \ref{prop:K_inf_bound}}\label{subsec:Proof_Prop_K_inf_bound}
In this subsection, we finally prove Proposition \ref{prop:K_inf_bound}.

\begin{proof}[Proof of Proposition \ref{prop:K_inf_bound}]

We now complete the proof of Proposition \ref{prop:K_inf_bound}. Recall that our goal is to uniformly bound
\begin{equation}\label{eq:I_H_final}
 I=K_T \abs{\phi}^2(z')
=
\sqrt{\frac{2}{\pi}}
\int_0^\infty
H(w)e^{-w^2/4}
\left(
T\sin(Tw)+\frac w2\cos(Tw)
\right)
\,dw.
\end{equation}

Since \(H\) is odd, the integrand in \eqref{eq:I_H_final} is even. Hence
\[
I
=
\frac{1}{\sqrt{2\pi}}
\int_{-\infty}^{\infty}
H(w)e^{-w^2/4}
\left(
T\sin(Tw)+\frac w2\cos(Tw)
\right)
\,dw.
\]
Rewriting the sine and cosine functions in terms of the exponential function, we obtain
\[
I
=
\frac{1}{\sqrt{2\pi}}(I_+ + I_-),
\]
where
\[
I_+
:=
\int_{-\infty}^{\infty}
H(w)e^{-w^2/4}
\left(\frac{T}{2i}+\frac w4\right)e^{iTw}
\,dw
\]
and
\[
I_-
:=
\int_{-\infty}^{\infty}
H(w)e^{-w^2/4}
\left(-\frac{T}{2i}+\frac w4\right)e^{-iTw}
\,dw.
\]

Put $\sigma_T:=\frac{\pi}{2}-\frac1T.$ Using the holomorphicity of $H$ (see Lemma \ref{lem:H_holomorphic}), we shift the contour in \(I_+\) to the line $
\Im w=\sigma_T$
and the contour in \(I_-\) to the line $
\Im w=-\sigma_T.$

\begin{comment}
Indeed, applying Cauchy's theorem on rectangles and then letting the horizontal
length tend to infinity, the vertical side integrals vanish: on
\(w=R+iy\), with \(|y|\leq\sigma_T\),
\[
|e^{-w^2/4}|
=
e^{-(R^2-y^2)/4}
\ll e^{-R^2/4},
\]
while the shifted-line bound for \(H\) gives at most exponential growth
\(e^{|R|}\) times a polynomial factor. The Gaussian therefore dominates.
\end{comment}
On the upper shifted line \(w=x+i\sigma_T\), we have $|e^{iTw}|=e^{-\frac{\pi}{2} T+1}$ and $|e^{-w^2/4}|\ll
e^{-x^2/4}$. Together with the result of Lemma \ref{lem:H_shifted_bound},
\[
|H(x+i\sigma_T)|
\ll
e^{\pi |t|}e^{|x|}
\bigl(|x|+\log T+1\bigr),
\]
we therefore obtain
\[
\begin{aligned}
|I_+|
&\ll
e^{-\frac{\pi}{2} T+\pi |t|}
\int_{-\infty}^{\infty}
e^{-x^2/4+|x|}
(T+|x|+1)
\bigl(|x|+\log T+1\bigr)
\,dx.
\end{aligned}
\]
The same argument on the lower shifted line gives the identical estimate for
\(I_-\). Since the Gaussian integral has bounded polynomial moments, we obtain
\[
|I|
\ll
e^{-\frac{\pi}{2} T+\pi |t|}
T
\bigl(1+\log T\bigr).
\]

Since \(z'\) was chosen so that
\(|K_T|\phi|^2(z')|=\|K_T|\phi|^2\|_\infty\), this proves
\[
\|K_T|\phi|^2\|_\infty
\ll
T(1+\log T)
\exp\left(-\frac{\pi}{2}T+\pi t\right).
\]
This completes the proof of Proposition \ref{prop:K_inf_bound}.
\end{proof}

\section{Effective equidistribution of closed geodesics}\label{sec:Equidistrib}
One of the trickiest aspects of Zelditch's Trace Formula is that the spectral side
\[
    \langle \varphi\phi_j,\phi_j\rangle \frac{1}{2\pi i}\int_{(\sigma)}\Theta_\tau^{t_j}(s) M \psi(s) ds,
\]
where $\frac{1}{2}<\sigma<1$, can be exponentially large in terms of $\tau$. To see this, we shift the contour $(\sigma)$ to the left, picking up various residues of the integrand. One can show that the residues
\[
\langle \varphi\phi_j,\phi_j\rangle  M \psi\left(\frac{1}{2}\pm it_j -m\right) {\rm Res}_{s=\frac{1}{2}\pm it_j -m} \Theta_\tau^{t_j}(s)
\]
for $m=0,1,2,\ldots$ grow at most polynomially in $|\tau|$, hence these are harmless. However, the residues
\[
\langle \varphi\phi_j,\phi_j\rangle  M \psi\left( -m\right) {\rm Res}_{s=-m} \Theta_\tau^{t_j}(s)
\]
for $m=0,1,2,\ldots$ can become exponentially large in $|\tau|$ if $|t_j| > |\tau|/2$ even if we apply the optimal bound for $\langle \varphi\phi_j,\phi_j\rangle$ in Theorem \ref{thm:TripleProduct}.

The geometric side of Zelditch's Trace Formula grows at most polynomially in $\tau$, meaning that there is an exponential cancellation in the summation on the spectral side. See the proof of Proposition \ref{prop:Zelditch_Equidist}. No direct argument will capture such an extreme cancellation.

Accordingly, for some large $M>0$, we will express the problematic residue terms 
\begin{equation}\label{eqn:formidable terms}
    \sum_{j=1}^\infty \sum_{m=0}^{M} \langle \varphi\phi_j,\phi_j\rangle  M \psi\left( -m\right) {\rm Res}_{s=-m} \Theta_\tau^{t_j}(s)    
\end{equation}
on the spectral side as a geometric summation over the hyperbolic conjugacy classes on the geometric side, and then argue that this grows indeed polynomially fast in $\tau$.

To be precise, in the following subsections, we first invert the function $t \mapsto M \psi\left(-m\right) {\rm Res}_{s=-m} \Theta^{t}_\tau(s)$ on the spectral side to a function $\psi_m: u \mapsto \psi_m(u)$ on the geometric side using the inverse Selberg/Harish-Chandra transform. We then define $\psi_\tau := \psi - \sum_{m=1}^{M} \psi_m$ so that the Selberg/Harish-Chandra transform $h_\tau$ of $\psi_\tau$ does not contain the problematic terms in Equation \eqref{eqn:formidable terms}. In Section \ref{subsec:Spec_Side}, we show the spectral side has only polynomial growth in $\tau$.

On the geometric side, however, the new term $\psi_m(u)$ decays at order $u^{-1}$ (Lemma~\ref{lem:psi_m}), which leads to the failure of absolute convergence (Lemma~\ref{lem:wPrimLengthSum} with $\delta=1$). Nevertheless, we establish conditional convergence in Proposition~\ref{prop:Weighted_Length_Sum_marginal}. The proof relies on Proposition~\ref{prop:Zelditch_Equidist}, which is a version of Zelditch's equidistribution result with an error estimate uniform in $\tau$.

\subsection{\texorpdfstring{Geometric test function $\psi:=\psi_{T,\epsilon}$}{Geometric test function psi}}
For a large parameter $T>0$, we define
\[
    \psi_T(x) :=
    \left\{
        \begin{aligned}
            &1 &&\text{ if } e^{T/2}<x<e^T\\
            &-2 && \text{ if } e^{T/4}<x<e^{T/2}\\
            &0 &&\text{ otherwise}.
        \end{aligned}
    \right.
\]
Let $\varrho \in C_0^\infty(\mathbb{R}_+)$ be a non-negative function with ${\rm supp}(\varrho)\subset (1/e,e)$ such that
\[
    \int_0^\infty \varrho(x) \frac{dx}{x}=1.
\]
Then, for $\epsilon>0$, we define
\begin{equation}\label{eq:psi_Tepsilon}
    \psi_{T,\epsilon} (x) =\frac{1}{\epsilon} \int_0^\infty \psi_T (x/y) \varrho (y^{1/\epsilon}) \frac{dy}{y}.
\end{equation}
Then
\[
    M\psi_{T,\epsilon}(s)=M\psi_T(s)\,M\varrho(\epsilon s).
\]

It follows that
\begin{equation}\label{eqn:psi_T,epsilon}
    \psi_{T,\epsilon}(x)
    \begin{cases}
        =-2& x\in[e^{T/4+\epsilon},e^{T/2-\epsilon}]\\
        =1 & x\in[e^{T/2+\epsilon},e^{T-\epsilon}]\\
        \in[-2,1] & x\in[e^{T/4-\epsilon},e^{T/4+\epsilon}]\cup[e^{T/2-\epsilon},e^{T/2+\epsilon}]\cup [e^{T-\epsilon},e^{T+\epsilon}]\\
        =0 & {\rm otherwise}.
    \end{cases}
\end{equation}
We assume $T$ is large and $\epsilon$ is small enough so that the intervals in the above equation have pairwise disjoint interiors. 

If $s=0$, then
\begin{equation}\label{eqn:Mpsi_T_0}
    |M\psi_T(0)|= 0,
\end{equation}
and if $s\neq 0$, then
\begin{equation}\label{eqn:Mpsi_T}
    |M\psi_T(s)|= \left|\frac{e^{sT}-3e^{sT/2}+2e^{sT/4}}{s}\right|
        \ll
    \left\{
    \begin{aligned}
        &\frac{e^{\Re(s)T}}{|s|}&\text{if } \Re(s)\ge0\\
        &\frac{e^{\Re(s)T/4}}{|s|}&\text{if } \Re(s)<0
    \end{aligned}   
    \right.
\end{equation}
In particular, for $s\in \mathbb{R}_+$, we have
\begin{equation}\label{eqn:Mpsi_T_detail}
    M\psi_T(s)=\frac{e^{sT}}{s}+\frac{1}{s}O(e^{sT/2}).
\end{equation}

\begin{lemma}\label{lem:varrho}
    There exists $\varrho \in C_0^\infty (\mathbb{R}_+)$ with ${\rm supp}(\varrho)\subset (1/e,e)$ and $\int_0^\infty \varrho(x) \frac{dx}{x}=1$ such that for all $n\ge 1$ and $s\in \mathbb{C}$ with $[s]_n\neq 0$, we have
    \[
        \|\varrho^{(n)}\|_\infty \ll \left(n(\log (n+1))^2\right)^n
    \]
    and
    \[
        |M\varrho(s)|
        \ll \frac{1}{|[s]_n|}e^{|\Re(s)+n|} \left(n(\log (n+1))^2\right)^n
        \ll_n \frac{1}{|[s]_n|}e^{|\Re(s)|}.
    \]
    Moreover, for all $|s|<1$, we have
    \[
        M\varrho(s)= 1+ O(|s|),
    \]
    and, for all $s\in \mathbb{C}$, we have
    \[
        |M\varrho(s)|\ll e^{|\Re(s)|}.
    \]
    All implied constants are absolute unless their dependence is explicitly indicated.
\end{lemma}
\begin{proof}
    By \cite[Theorem 1.3.5] {Hormander_PDE_bumpftn_bdd}, for any positive non-increasing sequence $\{a_k\}_{k\ge0}$ with $\sum_{k\ge0} a_k < e-1/e$, there exists a non-negative function $f\in C_c^\infty(\mathbb{R}_+)$ with ${\rm supp}(f)\subset (1/e,e)$, $\int_\mathbb{R_+} f\,dx=1$ and $\|f^{(n)}(x)\|_\infty\le 2^n/(a_0a_1\dots a_n)$. Define $\varrho(x):=xf(x)$. Since $(xf(x))^{(n)}=xf^{(n)}(x)+nf^{(n-1)}(x)$, we have
    \[
        \|\varrho^{(n)}(x)\|_\infty\le \frac{2^n(e+na_n/2)}{a_0a_1\dots a_n}\ll \frac{2^nn}{a_0a_1\dots a_n}.
    \]
    
    Note that
    \[
        \sum_{k=3}^\infty \frac{1}{k(\log k)^{2}}
        \,<\,
        \int_2^\infty \frac{1}{x(\log x)^{2}} dx
        \,=\,
        \frac{1}{\log 2}
        \,<\,
        e-\frac1e.
    \]
    By taking $a_{k}= \frac{1}{(k+3)(\log (k+3))^2}$ for $k\ge 0$ and using the Stirling approximation, for all $n\ge 1$, we obtain
    \begin{align*}
        \log \|\varrho^{(n)}\|_\infty &\le n \log2+\log n +\log ((n+3)!)+2\sum_{k=3}^{n+3}\log \left(\log k\right)+O(1)\\
        &\le n\log 2+\log n+\left(n\log(n+1)+\frac{7}{2}\log(n+1)-n\right)+ 2(n+1)\log(\log(n+1))+O(1)\\
        &\le n\log (n+1)+ 2n\log(\log (n+1))+O(1),
    \end{align*}
    and thus
    \[
        \|\varrho^{(n)}\|_{\infty}\ll \left(n(\log (n+1))^2\right)^n.
    \]
    By integrating by parts \(n\) times, we obtain
    \[
        |M\varrho(s)|
        =
        \left|\int_{1/e}^e \frac{\varrho^{(n)}(x)}{[s]_n} x^{s+n-1}\,dx \right|
        \le
        \frac{\|\varrho^{(n)}\|_\infty}{|[s]_n|} \int_{1/e}^e x^{\Re(s)+n-1}\,dx
        \le
        2\frac{\|\varrho^{(n)}\|_\infty}{|[s]_n|} e^{|\Re(s)+n|}.
    \]
    On the other hand, for $|s|<1$, we have
    \[
        |M\varrho(s)-1|
        =\left|\int_{1/e}^e (x^s-1)\varrho(x)\frac{dx}{x}\right|
        \le \int_{1/e}^e |e^{s\log x}-1|\,\varrho(x)\frac{dx}{x}
        \ll \int_{1/e}^e |s\log x|\,\varrho(x)\frac{dx}{x}
        \le |s|.
    \]
    The bound $|M\varrho(s)|\ll e^{|\Re(s)|}$ for all $s\in \mathbb{C}$ is immediate.
\end{proof}

\subsection{\texorpdfstring{Estimate of ${\rm Res}_{s=\frac12-m\pm it_j}\Theta^{t_j}_\tau(s)M\psi_{T,\epsilon}(s)$}{Estimate of Res Theta Mpsi}}
In this section, we estimate 
\[
    {\rm Res}_{s=\frac12-m\pm it_j}\Theta^{t_j}_\tau(s)M\psi_{T,\epsilon}(s),
\]
where $0\le m \le  |\tau|$, which will be used in subsequent sections. The $\epsilon>0$ is sufficiently small so that $\epsilon(1+m)<1$.

Recall
\[
    \Theta^t_\tau(s)=\frac{2^{2-2s} \sqrt{\pi}}{\sin(\pi s)}\frac{\Gamma\left(-\frac{1}{2}+it+s\right)\Gamma\left(-\frac{1}{2}-it+s\right)}{\Gamma\left(-\frac{1}{2}\left(\frac{1}{2}+i\tau\right)+s\right)\Gamma\left(-\frac{1}{2}\left(\frac{1}{2}-i\tau\right)+s\right)}\cosh(\pi t).
\]
Since $\Theta^0_\tau(s)$ has a double pole at $s=\frac12-m$, we separately discuss the cases $t_j\neq 0$ and $t_j=0$.

\begin{lemma}\label{lem:res Theta at s=1/2+it_j-m, t_j nonzero}
We have
\[
    |M\psi_{T,\epsilon}(\tfrac12\pm it_j-m)|
    =
    \begin{cases}
        O(e^{(\frac12 \mp \Im(t_j))T})& m=0\\
        \frac{1}{m}O(e^{\frac14(\frac12-m\mp \Im(t_j))T}) & m>0.
    \end{cases}
\]
For $t_j\in i\mathbb{R}$ with $t_j\neq 0$, we have
\[
    |{\rm Res}_{s=\frac{1}{2}\pm it_j -m}\Theta^{t_j}_\tau(s)|
    \ll
    \frac{e^{2m}}{\sqrt{1+m}}(1+\tau)^{2m+1/2+2\Im(t_j)}e^{\pi\tau/2}.
\]
On the other hand, for $t_j\in \mathbb{R}_+$ with $t_j\neq 0$, we have
\[
    |{\rm Res}_{s=\frac{1}{2}\pm it_j -m}\Theta^{t_j}_\tau(s)|
    \ll
    \frac{1}{\sqrt{1+m}} \exp(-\tfrac{\pi}{2}(2|t_j|-|t_j+\tfrac\tau2|-|t_j-\tfrac\tau2|))\times
        \begin{cases}
            \left(\frac{12e^2(1+t_j)}{1+m}\right)^m &  \tau<2t_j\\
            \sqrt{1+\tau}\left(\frac{2e^2
            (1+\tau)^2}{1+m}\right)^m & \tau\ge 2t_j.
        \end{cases}
\]
In particular, for $m=0$ and $t_j\neq 0$, we have
\[
    \left|{\rm Res}_{s=\frac12\pm it_j}\Theta_{\tau}^{t_j}(s)\right|
    \ll
    \begin{cases}
        (1+|\tau|)^{1/2+2\Im(t_j)}e^{\pi|\tau|/2}& t_j\in i\mathbb{R}\\
        (1+|\tau|)^{1/2} \exp\left(-\tfrac{\pi}{2}(2|t_j|-|t_j+\tfrac12\tau|-|t_j-\tfrac12\tau|)\right) & t_j\in \mathbb{R}
    \end{cases}.
\]
\end{lemma}
\begin{proof}
By \eqref{eqn:Mpsi_T}, we have
\[
    M\psi_T(\tfrac12\pm it_j-m)
    =
    \begin{cases}
        O\left(e^{(\frac12\mp \Im(t_j))T}\right)&m=0\\
        \frac{1}{m}O\left(e^{\frac14(\frac12-m\mp \Im(t_j))T}\right)&m>0.
    \end{cases}
\]
By Lemma \ref{lem:varrho},
\[
    |M\varrho(\epsilon(\tfrac12\pm it_j-m))|\le e^{|\epsilon(\frac12-m\mp \Im(t_j))|}=O(1).
\]
Hence, we obtain the desired estimate of $M\psi_{T,\epsilon}$.

For $t_j\neq 0$, the residue of the simple poles at $s=\tfrac12\pm it_j-m$ is given by
\begin{equation}\label{eqn:res_1/2+it_j-m_raw}
    {\rm Res}_{s=\frac12\pm it_j-m} \Theta_\tau^{t_j}(s)
    =
    \frac{2^{1+2m\mp 2it_j}\sqrt{\pi}}{m!}
    \frac{\Gamma(-m\pm 2it_j)}{\Gamma\left(\frac14-m\pm i\left(t_j-\frac{\tau}{2}\right)\right)\Gamma\left(\frac14-m\pm i\left(t_j+\frac{\tau}{2}\right)\right)},
\end{equation}
where we use $\cosh \pi t_j=(-1)^m\sin\pi(\frac12\pm it_j-m)$.

Since $|t_j|\gg_X 1$, the reflection formula \eqref{eqn:reflection} yields
\begin{align*}
    &|{\rm Res}_{s=\frac{1}{2}\pm it_j -m}\Theta^{t_j}_\tau(s)|\\
    &\ll_X \frac{1}{\sqrt{1+m}}\left(\frac{4e}{1+m}\right)^m \exp(\pi(|t_j+\tfrac\tau2|+|t_j-\tfrac\tau2|-2|t_j|))
    \underbrace{\frac{|\Gamma(\frac34+m \pm i(t_j+\frac\tau2))\Gamma(\frac34+m \pm i(t_j-\frac\tau2))|}{|\Gamma(1+m\pm 2i t_j)|}}_{:=A}.
\end{align*}

\noindent{\it Case 1: $t_j\in i\mathbb{R}$.}
Note that $0\le|t_j|<1/2$. Hence, if $|\tau|=O(1)$, as well, then all the terms are $O(1)$. Hence we suppose $\tau\gg1$.

Since $|t_j|=O(1)$, we have
\[
    \exp(\pi(|t_j+\tfrac\tau2|+|t_j-\tfrac\tau2|-2|t_j|))\asymp e^{\pi\tau}.
\]
By Stirling's approximation \eqref{eqn:Stirling}, we have
\begin{align*}
    A
    &=\frac{|\Gamma(\frac34+m \pm \Im(t_j) + i\tau/2)|^2}{|\Gamma(1+m\pm 2\Im(t_j))|}\\
    &\asymp \frac{((m+\frac34\pm \Im(t_j))^2+\tau^2/4)^{m+1/4\pm \Im(t_j)}}{(m+1\pm 2\Im(t_j))^{m+1/2\pm 2\Im(t_j)}}\exp\left(-m-\tau\arctan\left(\frac{\tau}{2(\frac34+m\pm \Im(t_j)
    )}\right)\right).
\end{align*}
Note that
\[
    \arctan\left(\frac{\tau}{2(\frac34+m\pm \Im(t_j)
    )}\right)
    =\frac\pi2-\arctan\left(\frac{2m+\frac{3}{2}\pm 2\Im(t_j)}{\tau}\right)
    \ge\frac\pi2-\frac{2m+5/2}{\tau},
\]
and
\[
    (m+1)^{m+\alpha}\asymp(m+1\pm2\Im(t_j))^{m+\alpha}\asymp (m+\tfrac34\pm \Im(t_j))^{m+\alpha}
\]
for $\alpha\in(0,1)$. Hence, we have
\[
    A\ll (m+1)^{m}\left(1+\frac{\tau^2}{4(m+1)^2}\right)^{m+1/4\pm \Im(t_j)} e^{m-\pi\tau/2}.
\]
Note that
\[
    1+\frac{\tau^2}{4(m+1)^2}\le 1+\frac{\tau^2}{4}\le \frac{(1+\tau)^2}{4}    
\]
for $\tau\ge 3/2$. Therefore,
\[
    A\ll \left(\frac{m+1}{4}\right)^m\left(1+\tau\right)^{2m+1/2+2\Im(t_j)} e^{m-\pi\tau/2}.
\]

\noindent{\it Case 2: $t_j\in \mathbb{R}$.}

Suppose $\tau\in \mathbb{R}$.
\begin{align*}
    A
    &\ll
    \underbrace{\frac{|\left(\frac34+m+i(t_j+\tfrac12\tau)\right)\left(\frac34+m+i(t_j-\tfrac12\tau)\right)|^{m+1/4}}{|1+m+2it_j|^{m+1/2}}}_{:=A_1}\\
    &\hspace{10pt}\times
    \underbrace{\exp\left(-m-(t_j+\tfrac12\tau)\arctan\tfrac{t_j+\tau/2}{3/4+m}-(t_j-\tfrac{1}{2}\tau)\arctan\tfrac{t_j-\tau/2}{3/4+m}+2t_j\arctan\tfrac{2t_j}{1+m}\right)}_{:=A_2}.
\end{align*}

Note that
\[
    A_1
    \ll
    \frac{1}{\left((1+m)^2+4t_j^2\right)^{1/8}}
    \left(\frac{(\tfrac34+m)^2+t_j^2+(\tau/2)^2}{\sqrt{(1+m)^2+4t_j^2}}\right)^{m+1/4}.
\]

If $\tau<2t_j$, we have
\[
    \frac{(\tfrac34+m)^2+t_j^2+(\tau/2)^2}{\sqrt{(1+m)^2+4t_j^2}}
    \le \frac{(\tfrac34+m)^2}{2t_j}+t_j
    \le \frac{9}{32t_j}+\frac32+\tau+t_j\le 3t_j+2,
\]
where $t_j\ge 9/16$. The case when $t_j,\tau=O(1)$ is trivial. 

If $\tau\ge 2t_j$, then we have
\[
    \frac{(\tfrac34+m)^2+t_j^2+(\tfrac\tau2)^2}{\sqrt{(1+m)^2+4t_j^2}}
    \le \left(\tfrac34+m\right)+\frac\tau4+\left(\frac{\tau^2/4}{1+m}\right)
    \le (\tau/2+2)^2.
\]
Hence,
\[
    A_1
    \ll
    \begin{cases}
        3^m(1+t_j)^{m} & \tau<2t_j\\
        2^{-m}(1+\tau)^{2m+1/2}& \tau\ge 2t_j.
    \end{cases}
\]

On the other hand, for $x,y>0$, we have
\[
    \frac{\pi x}{2}-y\le x\arctan\frac xy = \frac{\pi x}{2}-x\arctan\frac{y}{x}\le \frac{\pi x}2.
\]
It therefore follows that
\[
    A_2
    \ll \exp\left(\frac\pi2\left(-|t_j+\tfrac12\tau|-|t_j-\tfrac12\tau|+|2t_j|\right)+m\right).
\]

Suppose $\tau\in i\mathbb{R}$. Then $m=0$ and
\[
    A
    =\frac{|\Gamma(\frac34\pm\frac12\Im(\tau)+it_j) \Gamma(\frac34\mp\frac12\Im(\tau) +i t_j)|}{|\Gamma(1\pm 2i t_j)|}=O(1).\qedhere
\]
\end{proof}
\begin{lemma}\label{lem:res Theta at s=1/2+it_j-m, t_j zero}
    For $|\tau|\ge m\ge 0$, we have
    \[
        |{\rm Res}_{s=\frac12-m}M\psi_{T,\epsilon}(s)\Theta^0_\tau(s)|
        \ll
        \frac{e^{2m}(1+\tau)^{1/2+2m}e^{\pi|\tau|/2}}{\sqrt{1+m}}
        \times
        \begin{cases}
            \left(\log(1+|\tau|)+T\right)e^{\frac12T} &m=0\\
            \frac{1}{m}\left(\log(1+|\tau|)+T\right)e^{\frac14(\frac12-m)T} & m>0.
        \end{cases}
    \]
\end{lemma}
\begin{proof}
The function $\Theta^0_\tau(s)$ has a double pole at $s=\tfrac12-m$ with 
\[
    {\rm Res}_{s=\frac12-m}M\psi_{T,\epsilon}(s)\,\Theta^0_\tau(s)
    =
    \lim_{t\to 0}
    \underbrace{\left(
        {\rm Res}_{s=\frac12-m+it}M\psi_{T,\epsilon}(s)\,\Theta^{t}_\tau(s)+{\rm Res}_{s=\frac12-m-it}M\psi_{T,\epsilon}(s)\,\Theta^{t}_\tau(s)
    \right)}_{:=R(t)}.
\]
By \eqref{eqn:res_1/2+it_j-m_raw}
\[
    R(t)
    =
    \frac{2^{1+2m}\sqrt{\pi}}{m!}
    \sum_{\pm}\frac{2^{\mp 2it}\Gamma(-m\pm 2it)\,M\psi_{T,\epsilon}(\frac12-m\pm it)}{\Gamma\left(\frac14-m\pm i\left(t-\frac{\tau}{2}\right)\right)\Gamma\left(\frac14-m\pm i\left(t+\frac{\tau}{2}\right)\right)}.
\]
By using
\begin{align*}
    2^{-2z}&=1-2(\log2)z+O(z^2)\\
    \Gamma(-m+2z)&=\frac{(-1)^m}{m!}\left(\tfrac12z^{-1}+\psi(m+1)+O_m(z)\right)\\
    \frac1{\Gamma\left(\frac14-m\pm i\frac{\tau}{2}+z\right)}&=\frac1{\Gamma\left(\frac14-m\pm i\frac{\tau}{2}\right)}\left(1-\psi(\tfrac14-m\pm i\tfrac{\tau}{2})z+O_{m,\tau}(z^2)\right)\\
    M\psi_{T,\epsilon}(\tfrac12-m+z)&=M\psi_{T,\epsilon}(\tfrac12-m)+(M\psi_{T,\epsilon})'(\tfrac12-m)\,z+O_{T,\epsilon,m}(z^2)
\end{align*}
for $\psi(z)=\Gamma'(z)/\Gamma(z)$ and $z=\pm it$,
we obtain
\begin{align*}
    &\lim_{t\to 0} R(t)
    =
    \frac{(-1)^m2^{1+2m}\sqrt{\pi}}{(m!)^2}\times\\
    &\quad
    \left(
    \frac{\left(2\psi(m+1)-2\log2-\psi(\tfrac14-m+ i\tfrac{\tau}{2})-\psi(\tfrac14-m- i\tfrac{\tau}{2})\right)M\psi_{T,\epsilon}(\frac12-m)+(M\psi_{T,\epsilon})'(\frac12-m)}{\Gamma\left(\frac14-m+i\frac{\tau}{2}\right)\Gamma\left(\frac14-m-i\frac{\tau}{2}\right)}
    \right).
\end{align*}
For
\[
    I_m:=(M\psi_{T,\epsilon})'(\tfrac12-m)=\int_0^\infty(\log x)\,\psi_{T,\epsilon}(x)x^{-\tfrac12-m}dx,
\]
by using \eqref{eqn:psi_T,epsilon}, we obtain
\[
    \left|I_m\right|
    \ll \left|\int_{e^{\frac14T-\epsilon}}^{e^{T+\epsilon}}(\log x)x^{-\frac12-m}dx\right|
    \ll
    \begin{cases}
        T\,e^{\frac12T}&m=0\\
        \frac1m T\,e^{\frac14(\frac12-m)T}&m>0.
    \end{cases}
\]
On the other hand, we have
\[
    M\psi_{T,\epsilon}(\tfrac12-m)
    =
    \begin{cases}
        O(e^{\frac 12 T})&m=0\\
        \frac{1}{m}O(e^{\frac14(\frac12-m)T})&m>0.
    \end{cases}
\]
By Stirling's approximation, for all $z$ with $\arg(z)\in[-\pi+\delta,\pi-\delta]$, we have
\[
    \psi(z)=\log z+O_\delta(|z|^{-1}).
\]
Therefore,
\[
    |{\rm Res}_{s=\frac12-m}M\psi_{T,\epsilon}(s)\Theta^0_\tau(s)|
    \ll
    \frac{2^{2m}}{(m!)^2}
    \left(\frac{\log(1+|\tau|)+T}{|\Gamma\left(\frac14-m+i\frac{\tau}{2}\right)|^2}\right)
    \times
    \begin{cases}
        e^{\frac12T} &m=0\\
        \frac{1}{m}e^{\frac14(\frac12-m)T} & m>0.
    \end{cases}
\]
Note $m!\asymp (1+m)^{m+1/2}e^{-m}$. Since $m\le |\tau|$, the argument of $\frac14-m+i\frac\tau2$ is uniformly away from $\pm \pi$. Hence, the Stirling approximation \eqref{eqn:Stirling} yields
\begin{align*}
    \frac1{|\Gamma\left(\frac14-m+i\frac{\tau}{2}\right)|^2}
    &\asymp |(\tfrac14-m)+i\tfrac{\tau}{2}|^{2m+1/2}\exp\left(-2m+\tau \arg\left(\frac14-m+i\tau/2\right)\right).
\end{align*}
Since
\[
    \tau\arg\left(\frac14-m+i\tau/2\right)
    =\frac{\pi\tau}{2}+\tau\arctan\left(\tfrac{m-\frac14}{\tau/2}\right)
    \le \frac{\pi\tau}{2}+2m,
\]
we have
\[
    \frac1{|\Gamma\left(\frac14-m+i\frac{\tau}{2}\right)|^2}\
    \ll
    \frac{\left((1+m)(1+\tau)\right)^{2m+1/2}}{2^{2m}}e^{\pi\tau/2}.
\]
Then we obtain the desired estimate.
\end{proof}

\subsection{Zelditch's effective equidistribution with a uniform error bound}
In this subsection, we revisit Zelditch's proof of \eqref{eqn:Equidist_k} in \cite{Zelditch_EquidistCompact} to obtain an estimate that is uniform in $\tau$. Our argument relies on Theorem~\ref{thm:TripleProduct}, which is a new result proved in this paper.
\begin{proposition}\label{prop:Zelditch_Equidist}
    Let $X$ be a compact hyperbolic surface. Let $0=\lambda_0< \lambda_1\le \dots$ be the eigenvalues of the Laplace--Beltrami operator $-\Delta_X$ such that $\phi_j$ denotes the $L^2$-normalized eigenfunction with eigenvalue $\lambda_j$. Let $\lambda_j=\tfrac14+t_j^2$. Let $\varphi=\phi_k$ and $\tau=t_k$ with $k\ge1$. Then, for all $T\gg_X1$, we have
    \begin{align*}
        \sum_{l(\gamma)<T}\int_{\gamma_0}\varphi \,ds
        &=\sum_{j:\lambda_j<3/16}\gamma_{k,j}\langle\varphi,\phi_j^2\rangle\,e^{(\frac12+\Im(t_j))T}+e^{\pi|\tau|/2}(1+|\tau|)^3\,O_X(\sqrt{T}e^{3T/4})\\
        &= e^{\pi|\tau|/2}(1+|\tau|)^3\,O_X(\sqrt{T}e^{(1-\epsilon')T}),
    \end{align*}
    where $\epsilon'=\min\{\tfrac14,\tfrac12-|\Im(t_1)|\}$.
\end{proposition}
\begin{proof}
    We apply Theorem \ref{thm:ZTF} with $\psi=\psi_{T,\epsilon}$.
    
    We first assume $T\gg1$ and $\epsilon<1$. We will later set $\epsilon=e^{-\frac14T}$ and explain why this is the optimal choice.

    Recall
    $\|\varphi\|_\infty\ll (1+|\tau|)^{1/2}$ (\cite{SoggeZelditch}).
    For notational convenience, we set
    \[
        \nu_t := \left(2\sinh\left(t/2\right)\right)^2
        \qquad
        \text{and}
        \qquad
        \nu_\gamma := \nu_{l(\gamma)}.
    \]
    \noindent\emph{Case 1: $e^{T}\le(1+|\tau|)^{10}$}
    
    By \eqref{eqn:PrimLengthSum}, we have
    \[
        \left|\sum_{l(\gamma)<T}\int_{\gamma_0}\varphi\, ds\right|
        \,\ll\, \sum_{l(\gamma)<T}(1+|\tau|)^{1/2}\,l(\gamma_0)
        \,\ll\, (1+|\tau|)^{1/2} e^T
        \,\ll\, (1+|\tau|)^3 e^{\frac34T}.
    \]

    \noindent\emph{Case 2: $e^T > (1+|\tau|)^{10}$}
    
    Let us estimate the geometric side. By using \eqref{eqn:PrimLengthSum} and
    \[
        \nu_{t+\delta_t}=e^t
        \qquad
        \text{where }
        \delta_t:=2\arsinh(e^{t/2}/2)-t\asymp e^{-t},
    \]
    one can show that
    \begin{align*}
        &\left|
        \sum_{\{\gamma\}}  \psi_{T,\epsilon}\left(\nu_{\gamma}\right)\,\int_{\gamma_0}\varphi(s)\,ds
        -
        \sum_{\nu_\gamma < e^T}   \int_{\gamma_0}\varphi(s)\,ds\right| \\
        &\ll (1+|\tau|)^{1/2}
        \left(\sum_{\nu_\gamma <e^{T/2+\epsilon}}l(\gamma_0)+\sum_{\nu_\gamma \in(e^{T-\epsilon},e^{T+\epsilon})}l(\gamma_0)\right)\\
        &\ll (1+|\tau|)^{1/2}\left(e^{T/2}+\epsilon e^T+\sqrt{T}e^{\frac34T}\right).
    \end{align*}
    By \eqref{eq:PGT}, we also have
    \begin{align*}
        \left|\left(\sum_{\nu_\gamma<e^T}-\sum_{l(\gamma) < T}\right)\int_{\gamma_0}\varphi(s)\,ds\right|
        &\ll
        (1+|\tau|)^{1/2}\sum_{T\le l(\gamma)<T+\delta_T}l(\gamma_0)\\
        &\ll (1+|\tau|)^{1/2}\left(\delta_Te^T+\sqrt{T} e^{\frac34 T}\right)\\
        &\ll (1+|\tau|)^{1/2}\sqrt{T}e^{\frac34T}.
    \end{align*}
    Therefore, 
    \begin{equation}\label{eqn:ZTF_geom_side}
    \left|
        \mathrm{Tr}(\mathrm{Op}(\varphi)\circ T_K)
        -
        \sum_{{l(\gamma)} < T} \int_{\gamma_0}\varphi(s)ds
    \right|
    = (1+|\tau|)^{1/2} O\left(\max\left\{\epsilon e^T,\sqrt{T}\,e^{\frac34 T}\right\}\right).
    \end{equation}
In Lemma \ref{lem:ZTF_spectral_side}, we show that the spectral side yields
\begin{align}
\begin{split}\label{eqn:ZTF_spectral_side}
    \mathrm{Tr}(\mathrm{Op}(\varphi)\circ T_K)
    &=\left(\sum_{j:\,0<\lambda_j<1/4} \gamma_{k,j}\langle\varphi,\phi_j^2\rangle e^{(\frac12+\Im(t_j))T}\right)\\
    &\quad+
    (1+|\tau|)^3 \,O_X\left(\max\{\epsilon e^{(\frac12+\Im(t_1))T},Te^{\frac12 T},\epsilon^{-1}e^{\frac12 T}\}\right)\\
    &\quad+e^{\pi|\tau|/2}(1+|\tau|)^3\,O_{X,\eta}\left(\epsilon^{-(5/2+3\eta)}e^{-\eta T/4}\right)
\end{split}
\end{align}
for any $\eta\in (0,\frac12-\Im(t_1))$.

Since we have $O(\epsilon e^T)$ in \eqref{eqn:ZTF_geom_side} and $O(\epsilon^{-1} e^{\frac12 T})$ in \eqref{eqn:ZTF_spectral_side}, the choice $\epsilon=e^{-\frac14T}$ is optimal, yielding $O(e^{\frac34T})$. With this choice, we obtain $\epsilon^{-(5/2+3\eta)}e^{-\eta T/4}=e^{(\frac58+\frac12\eta)T}$. Thus, we choose $\eta\in(0,\min\{\frac14,\frac12-\Im(t_1)\})$.

Then, from \eqref{eqn:ZTF_geom_side} and \eqref{eqn:ZTF_spectral_side}, we obtain
\begin{align*}
    \sum_{l(\gamma) < T}   \int_{\gamma_0}\varphi(s)ds
    &=
    \left(\sum_{0\neq t_j \in i\mathbb{R}} \gamma_{k,j}\langle\varphi,\phi_j^2\rangle e^{(\frac12+\Im(t_j))T}\right)
    +
    (1+|\tau|)^3e^{\pi|\tau|/2}O_X\left(\sqrt{T}\, e^{\frac34T}\right).
\end{align*}
By using \eqref{eqn:Stirling_vertical}, one can show $\gamma_{k,j}\ll_j e^{\pi|\tau|/2}(1+|\tau|)^{1/2-2\Im(t_j)}$. Then, by using Theorem~\ref{thm:TripleProduct}, uniformly for all $0\neq t_j\in i\mathbb{R}$, we obtain
\begin{equation}\label{eqn:gamma_kj TriplePDT}
    |\gamma_{k,j}\langle\varphi\phi_j,\phi_j\rangle|\ll (1+|\tau|)^{3/2}.
\end{equation}
Hence, the desired estimate follows.
\end{proof}

% We remark that if $\epsilon=e^{-(\frac14-\delta)T}$ for a small $\delta>0$, then
% \begin{equation}\label{eqn:Zeldith_equidistrib_epsilon=-1/4T+delta}
%     \sum_{l(\gamma) < T}   \int_{\gamma_0}\varphi(s)ds
%     =
%     \left(\sum_{0\neq t_j \in i\mathbb{R}} \gamma_{k,j}\langle\varphi,\phi_j^2\rangle e^{(\frac12+|t_j|)T}\right)
%     +
%     (1+|\tau|)^3e^{\pi|\tau|/2}O_{X,\delta}\left(e^{(\frac34+\delta)T}\right).
% \end{equation}

\begin{lemma}\label{lem:ZTF_spectral_side}
    Equation \eqref{eqn:ZTF_spectral_side} holds.
\end{lemma}
\begin{proof}
We assume $\tau>2$. The case $|\tau|\le 2$ follows similarly.

Recall from Lemma \ref{lem:psi_to_h} that
\[
    \mathrm{Tr}(\mathrm{Op}(\varphi)\circ T_K)=\sum_{j\ge1}h(t_j)\langle\varphi,\phi_j^2\rangle,
\]
where 
\[
    h(t) =  \frac{1}{2\pi i}\int_{(\sigma)}
    \underbrace{\frac{2^{2-2s} \sqrt{\pi}}{\sin(\pi s)}\frac{\Gamma\left(-\frac{1}{2}+it+s\right)\Gamma\left(-\frac{1}{2}-it+s\right)}{\Gamma\left(-\frac{1}{2}\left(\frac{1}{2}+i\tau\right)+s\right)\Gamma\left(-\frac{1}{2}\left(\frac{1}{2}-i\tau\right)+s\right)}\cosh(\pi t)}_{=\Theta^t_\tau(s)} M \psi_{T,\epsilon}(s) ds.
\]

We shift the contour $(\sigma)$ with $1/2+\Im(t_1)<\sigma<1$ to $(-\eta)$ with $\Im(t_1)-1/2<-\eta<0$.
We pick up the residues at $s=\tfrac12\pm it_j$ and $s=0$. Since $M\psi_{T,\epsilon}(0)=0$, the residue at $s=0$ is zero. Then
\[
    h(t_j)=
    \underbrace{\sum_{\pm}{\rm Res}_{s=\frac12\pm it_j}\left(\Theta^{t_j}_\tau(s)M\psi_{T,\epsilon}(s)\right)}_{:=h_1(t_j)}
    +
    \underbrace{\frac{1}{2\pi i}\int_{(-\eta)}\Theta^{t_j}_\tau(s)M\psi_{T,\epsilon}(s)\,ds}_{:=h_2(t_j)},
\]
where
\[
    {\rm Res}_{s=\frac{1}{2}\pm it_j}\left(\Theta^{t_j}_\tau(s)M\psi(s)\right)
    =
    \underbrace{\frac{2^{1\mp2i t_j}\sqrt{\pi}\,\Gamma\left(\pm2i t_j\right)}{\Gamma\left(\frac14\pm it_j+ \frac12i\tau\right)\Gamma\left(\frac14\pm it_j-\frac12i\tau\right)}}_{={\rm Res}_{s=\frac12\pm it_j}\Theta_{\tau}^{t_j}(s)}
    M\psi_{T,\epsilon}(\tfrac12\pm it_j).
\]
When $t_j=0$, exceptionally, 
\[
    h_1(0):={\rm Res}_{s=\frac12}\left(\Theta^{0}_\tau(s)M\psi_{T,\epsilon}(s)\right).
\]
The leading term will be obtained from $h_1(t_j)$ with $t_j\in i\mathbb{R}$.
\medskip

\noindent (i) We prove
\begin{align}
\begin{split}\label{eqn:h1_sum}
    &\sum_{j\ge1} h_1(t_j)\langle\varphi,\phi_j^2\rangle\\
    &=
    \left(\sum_{0\neq t_j\in i\mathbb{R}} \gamma_{k,j}\langle\varphi,\phi_j^2\rangle e^{(\frac12+\Im t_j)T}\right)
    +
    (1+|\tau|)^3\,O_X\left(\max\{\epsilon e^{(\frac12+\Im(t_1))T},Te^{T/2},\epsilon^{-1}e^{T/2}\}\right).
\end{split}
\end{align}

{\it Case 1. $t_j\in i\mathbb{R}$.} We show
\begin{equation}\label{eqn:h_1_sum_Im}
    \sum_{t_j\in i\mathbb{R}} h_1(t_j)\langle\varphi,\phi_j^2\rangle
    =
    \left(\sum_{0\neq t_j \in i\mathbb{R}} \gamma_{k,j}\langle\varphi,\phi_j^2\rangle e^{(\frac12+\Im(t_j))T}\right)
    +
    (1+|\tau|)^3\,O_X\left(\max\{\epsilon e^{(\frac12+\Im(t_1))T}, T e^{\frac12 T}\}\right).
\end{equation}

First, we consider the case when $t_j\neq 0$. It follows from Lemma \ref{lem:res Theta at s=1/2+it_j-m, t_j nonzero} for $m=0$ that
\begin{equation}\label{eqn:res_frac12+it_j_t_j_Im}
    \left|{\rm Res}_{s=\frac12+ it_j}\Theta_{\tau}^{t_j}(s)\right|
    \ll (1+|\tau|)^{3/2}e^{\pi\tau/2}.
\end{equation}

Since $\Im(t_j)\gg_X1$, by \eqref{eqn:Mpsi_T} and \eqref{eqn:Mpsi_T_detail}, we have
\begin{align*}
    M\psi_T(\tfrac12-it_j)&= \frac{e^{(\frac12+\Im(t_j))T}}{\frac{1}{2}+\Im(t_j)}+O(e^{(\frac14+\frac12\Im(t_j))T})\\
    M\psi_T(\tfrac12+it_j)&=O(e^{(\frac12-\Im(t_j))T}).
\end{align*}

The $\epsilon$ is small enough so that $|\epsilon(\tfrac12\pm it_j)|<1$. By Lemma \ref{lem:varrho}, we obtain
\[
    M\varrho(\epsilon(\tfrac12\pm it_j) )=1+ O_{X}(\epsilon).
\]
Thus,
\begin{align}
\begin{split}\label{eqn:Mpsi_frac12pmit_j_t_j_Im}
    M\psi_{T,\epsilon}(\tfrac12-it_j)&= \frac{e^{(\frac12+\Im(t_j))T}}{\frac{1}{2}+\Im(t_j)}+O\left(\epsilon e^{(\frac12+\Im(t_j))T}+e^{(\frac14+\frac12\Im(t_j))T}\right)\\
    M\psi_{T,\epsilon}(\tfrac12+it_j)&=O\left( e^{(\frac12- \Im(t_j))T}\right).
\end{split}
\end{align}

We define
\begin{align*}
    \gamma_{k,j}
    &:=\left({\rm Res}_{s=\frac12-it_j}\Theta_\tau^{t_j}(s)\right)\frac{1}{\frac12-it_j}\\
    &= \frac{2^{1-2\Im (t_j)}\sqrt{\pi}\,\Gamma\left(2\Im (t_j)\right)}{\Gamma\left(\frac14+\Im (t_j)+\frac12i t_k\right)\Gamma\left(\frac14+\Im (t_j)-\frac12i t_k \right)\left(\frac12+\Im (t_j)\right)}.
\end{align*}
Then, by \eqref{eqn:res_frac12+it_j_t_j_Im} and \eqref{eqn:Mpsi_frac12pmit_j_t_j_Im}, for $0\neq t_j\in i\mathbb{R}$, we have
\[
    h_1(t_j)=\gamma_{k,j}e^{(\frac12+\Im (t_j))T}+(1+\tau)^{3/2}e^{\pi|\tau|/2}\,O_X\left(\max\{\epsilon e^{(\frac12+\Im (t_j))T},e^{(\frac14+\frac12\Im(t_j))T}, e^{(\frac12-\Im(t_j))T}\}\right).
\]

When $t_j=0$, by Lemma \ref{lem:res Theta at s=1/2+it_j-m, t_j zero} with $m=0$, we obtain
\[
    h_1(0)= (1+|\tau|)^{1/2}(1+\log(1+|\tau|))e^{\pi|\tau|/2} O_X\left( Te^{\frac12 T}\right).
\]
By Theorem \ref{thm:TripleProduct}, for $t_j\in i\mathbb{R}$, we have
\[
    \langle\varphi,\phi_j^2\rangle=(1+|\tau|)(1+\log(1+|\tau|)) e^{-\pi|\tau|/2}O_X(1).
\]
Then, we obtain \eqref{eqn:h_1_sum_Im}.
\medskip

{\it Case 2. $t_j\in \mathbb{R}$ with $t_j\neq 0$.} We show
\begin{equation}\label{eqn:h_1_sum_Re}
    \sum_{t_j>0}h_1(t_j)\langle\varphi,\phi_j^2\rangle
    \ll_X (1+|\tau|)^2\,\epsilon^{-1}e^{T/2}.
\end{equation}

It follows from Lemma \ref{lem:res Theta at s=1/2+it_j-m, t_j nonzero} that
\begin{equation}\label{eqn:res_frac12+it_j_t_j_Re}
    \left|{\rm Res}_{s=\frac12\pm it_j}\Theta_{\tau}^{t_j}(s)\right|
    \ll (1+|\tau|)^{1/2} \exp\left(-\tfrac{\pi}{2}(2|t_j|-|t_j+\tfrac12\tau|-|t_j-\tfrac12\tau|)\right).
\end{equation}

By \eqref{eqn:Mpsi_T}, we have
\[
    |M\psi_T(\tfrac12+it_j)|
    =(1+t_j)^{-1}O(e^{T/2}).
\]
By using Lemma \ref{lem:varrho} with $n=2$ for $t_j\ge 1$, we obtain
\[
    |M\varrho(\epsilon(\tfrac{1}{2}+it_j))|
    \ll
    \begin{cases}
        e^{|\Re\epsilon(\frac{1}{2}+it_j)|}=O_{X}(1)&t_j<1/\epsilon\\
        \epsilon^{-2}(1+ t_j)^{-2}&t_j\ge 1/\epsilon.
    \end{cases}
\]
Therefore
\begin{equation}\label{eqn:Mpsi_frac12pmit_j_t_j_Re}
    |M\psi_{T,\epsilon}(\tfrac{1}{2}+it_j)|
    =
    \begin{cases}
        (1+t_j)^{-1}O_X(e^{T/2})& t_j< 1/\epsilon\\
        \epsilon^{-2}(1+ t_j)^{-3}O_X(e^{T/2})&t_j\ge 1/\epsilon.
    \end{cases}
\end{equation}
Therefore, by \eqref{eqn:res_frac12+it_j_t_j_Re}, \eqref{eqn:Mpsi_frac12pmit_j_t_j_Re}, and Theorem \ref{thm:TripleProduct}, we obtain
\[
    \sum_{t_j>0}h_1(t_j)\langle\varphi,\phi_j^2\rangle
    \ll (1+|\tau|)^{3/2}(1+\log(1+|\tau|)) \left(
    \sum_{0< t_j <1/\epsilon}(1+t_j)^{-1}
    +
    \epsilon^{-2}\sum_{1/\epsilon\le t_j}
        (1+t_j)^{-3}\right)O_{X}(e^{T/2}).
\]
By Lemma \ref{lem:sum(1+t)^delta}, we have
\[
    \sum_{0< t_j <1/\epsilon}(1+t_j)^{-1}
    +
    \epsilon^{-2}\sum_{1/\epsilon\le t_j} (1+t_j)^{-3}
    \ll
    \epsilon^{-1},
\]
and thus we obtain \eqref{eqn:h_1_sum_Re}.
\medskip

\noindent (ii) We show
\begin{equation}\label{eqn:h_2}
    \left|\sum_{j\ge1}h_2(t_j)\langle\varphi,\phi_j^2\rangle\right| \ll_X e^{\pi|\tau|/2}(1+|\tau|)^3 \, \epsilon^{-(5/2+3\eta)}e^{-\eta T/4},
\end{equation}
where $0<\eta<1/2-\Im(t_1)$.

Let us first estimate $\Theta_\tau^{t_j}(-\eta+it')$. See \eqref{eqn:Theta} for the formula of $\Theta^t_\tau(s)$. On the vertical contour $(-\eta)$, we use \eqref{eqn:Stirling_vertical} to obtain
\begin{align}
\begin{split}\label{eqn:Theta_contour_shift}
    |\Theta_\tau^{t_j}(-\eta+it')|
    &\ll \exp(-\tfrac\pi2(|t_j+t'|+|t_j-t'|-2|t_j|-|t'+\tfrac\tau2|-|t'-\tfrac\tau2|+2|t'|))\\
    &\hspace{20pt}\times \left((1+|t_j+t'|)(1+|t_j-t'|)\right)^{-\eta-1}\left((1+|t'+\tau/2|)(1+|t'-\tau/2|)\right)^{3/4+\eta}.
\end{split}
\end{align}

We use
\[
    -\frac\pi2\left(|t_j+t'|+|t_j-t'|-2|t_j|\right)
    =
    \left\{
    \begin{aligned}
        &\pi\left(|t_j|-|t'|\right)&&|t'|\ge t_j\\
        &0&&|t'|<t_j
    \end{aligned}
    \right.
\]
and
\[
    -\frac\pi2\left(2|t'|-|t'+\tfrac\tau2|-|t'-\tfrac\tau2|\right)
    =
    \left\{
    \begin{aligned}
        &0&&|t'|\ge \tau/2\\
        &\pi\left(|\tfrac\tau2|-|t'|\right)&&|t'|<\tau/2
    \end{aligned}
    \right.
\]
to obtain
\begin{equation}\label{eqn:h_2_exp}
    \text{ (exponential~term~in~\eqref{eqn:Theta_contour_shift})} \ll e^{\pi\tau/2}.    
\end{equation}
By using
\[
    (1+|t_j+t'|)(1+|t_j-t'|)(1+|t'|)\gg (1+|t_j|)^2,
\]
and 
\[
    (1+|t'+\tau/2|)(1+|t'-\tau/2|)\ll (1+|t'|)^2(1+|\tau|)^2,
\]
we have
\begin{equation}\label{eqn:h_2_poly_large_t_j}
    \text{ (polynomial~term~in~\eqref{eqn:Theta_contour_shift})}
    \ll
    (1+|t_j|)^{-2-2\eta}
    \left(
    1+|t'|
    \right)^{5/2+3\eta}\left(
    1+|\tau|
    \right)^{3/2+2\eta}.
\end{equation}
When $|t_j|=O(1)$, we use
\[
    (1+|t_j+t'|)(1+|t_j-t'|)\gg (1+|t'|)^2
\]
and obtain
\begin{equation}\label{eqn:h_2_poly_small_t_j}
    \text{(polynomial~term~in~\eqref{eqn:Theta_contour_shift})}
    \ll_{t_j}
    (1+|t'|)^{-1/2}
    \left(
    1+|\tau|
    \right)^{3/2+2\eta}.
\end{equation}

By \eqref{eqn:Mpsi_T} and Lemma \ref{lem:varrho}, we have
\begin{align*}
    |M\psi_T(-\eta+it')|&\ll_\eta \frac1{(1+|t'|)}e^{-\eta T/4}\\
    |M\varrho(\epsilon(-\eta+it'))|\,\,&\ll_{\eta,n} \frac{1}{|[-\epsilon\eta+i\epsilon t']_n|}e^{|\epsilon \eta|}=
    \begin{cases}
        O(1)&|t'|<1/\epsilon\\
        \epsilon^{-n}(1+|t'|)^{-n}&|t'|\ge 1/\epsilon
    \end{cases}
\end{align*}
and thus
\begin{equation}\label{eqn:Mpsi_t'}
    |M\psi_{T,\epsilon}(s)|=
    \begin{cases}
        (1+|t'|)^{-1}O(e^{-\eta T/4} )&|t'|<1/\epsilon\\
        \epsilon^{-n}(1+|t'|)^{-(n+1)}O_n(e^{-\eta T/4})&|t'|\ge 1/\epsilon.
    \end{cases}
\end{equation}

\noindent {\it Case 1: $t_j\in i\mathbb{R}$}.

By \eqref{eqn:h_2_poly_small_t_j} and \eqref{eqn:Mpsi_t'}, we have
\begin{align*}
    h_2(t_j)
    &=e^{\pi|\tau|/2}(1+|\tau|)^{3/2+2\eta}\left(\int_\mathbb{R} (1+|t'|)^{-3/2}\,dt'\right)O(e^{-\eta T/4})\\
    &=e^{\pi|\tau|/2}(1+|\tau|)^{3/2+2\eta}O(e^{-\eta T/4}).
\end{align*}

\noindent {\it Case 2: $t_j\in \mathbb{R}$}

By \eqref{eqn:h_2_poly_large_t_j} and \eqref{eqn:Mpsi_t'}, we obtain
\begin{align*}
    h_2(t_j)
    &= e^{\pi|\tau|/2}(1+t_j)^{-2-\eta}(1+|\tau|)^{3/2+2\eta}\\
    &\qquad
    \times
    \left(
    \int_{|t'|<1/\epsilon} (1+|t'|)^{3/2+3\eta}\,dt'
    +
    \epsilon^{-n}\int_{|t'|\ge1/\epsilon} (1+|t'|)^{3/2+3\eta-n}\,dt'
    \right)
    O_n(e^{-\eta T/4})\\
    &= e^{\pi|\tau|/2}(1+t_j)^{-2-\eta}(1+|\tau|)^{3/2+2\eta}\,O\left(\epsilon^{-(5/2+3\eta)}e^{-\eta T/4}\right),
\end{align*}
where we take $n>7/2+3\eta$, say $n=6$.

By using $|\langle\varphi\phi_j,\phi_j\rangle|\ll (1+|\tau|)^{1/2}$ and Lemma \ref{lem:sum(1+t)^delta}, we obtain \eqref{eqn:h_2}.
\end{proof}

\subsection{Weighted length sum}
Note that
\begin{equation}\label{eqn:abs/cond_conv}
    \sum_{\{\gamma\}}e^{-\delta l(\gamma)}\int_{\gamma_0} |\varphi|\, ds
    \le (1+|\tau|)^{1/2}\left(\sum_{\{\gamma\}}e^{-\delta l(\gamma)}\,l(\gamma_0)\right).
\end{equation}
Hence, by Lemma \ref{lem:wPrimLengthSum}, the series $\sum_{\{\gamma\}}e^{-\delta l(\gamma)}\int_{\gamma_0} \varphi\, ds$ is absolutely convergent if $\delta>1$. However, when $\delta=1$, the right-hand side of \eqref{eqn:abs/cond_conv} diverges. By using Proposition \ref{prop:Zelditch_Equidist}, we show the convergence of the series $\sum_{\{\gamma\}}e^{-l(\gamma)}\int_{\gamma_0} \varphi\, ds$ when ordered by length.

\begin{proposition}\label{prop:Weighted_Length_Sum_marginal}
    We have
    \[
        \sum_{\{\gamma\}}e^{- l(\gamma)}\int_{\gamma_0} \varphi \, ds=O_X\left( (1+|\tau|)^3\right),
    \]
    where $\sum_{\{\gamma\}}=\lim_{T\to \infty}\sum_{l(\gamma)<T}$.
\end{proposition}
\begin{proof}
Assume $\tau\in \mathbb{R}$. The case when $\tau\in i\mathbb{R}$ follows similarly.

Let
\[
    L=1+\frac{\pi\tau}{2\epsilon'}+\frac{\log((1+|\tau|))}{\epsilon'},
\]
where $\epsilon'=\min\{\tfrac14,\tfrac12-|\Im(t_1)|\}$ in Proposition \ref{prop:Zelditch_Equidist}. Then, we have
\[
    \sqrt{L}e^{-\epsilon'L}\ll e^{-\pi\tau/2}.
\]

By Lemma \ref{lem:wPrimLengthSum} and $\|\varphi\|_\infty\ll (1+|\tau|)^{0.5}$, we have
\[
    \left|\sum_{l(\gamma)<L}e^{-l(\gamma)}\int_{\gamma_0}\varphi(s)\,ds\right|
    \ll(1+|\tau|)^{0.5}\sum_{l(\gamma)<L}l(\gamma_0)e^{-l(\gamma)}\ll (1+|\tau|)^{0.5}L\ll (1+|\tau|)^{1.5}. 
\]

Let $\Psi_k(T):=\sum_{l(\gamma)<T}\int_{\gamma_0}\varphi\,ds$. By Proposition~\ref{prop:Zelditch_Equidist}, we have $\Psi_k(T)=e^{\pi\tau/2}(1+|\tau|)^3O(\sqrt{T}e^{(1-\epsilon')T})$. It follows that
\[
    \sum_{l(\gamma)\ge L} e^{-l(\gamma)}\int_{\gamma_0}\varphi(s)\,ds
    \,=\,
    \int_{L}^\infty e^{-x}d\Psi_k(x)
    \,=\,
    -e^{-L}\Psi_k(L)+\int_{L}^\infty e^{-x}\Psi_k(x)\,dx.
\]
By using $\sqrt{L}e^{-\epsilon'L}\ll_X e^{-\pi\tau/2}$ and $L>1$, we obtain
\[
    e^{-L}\Psi_k(L)=O((1+|\tau|)^3),
\]
and
\[
    \int_{L}^\infty e^{-x}|\Psi_k(x)|\,dx
    \,\ll\, e^{\pi\tau/2}(1+|\tau|)^3\int_L^\infty \sqrt{x}\,e^{-\epsilon'x}dx
    \,\ll\, e^{\pi\tau/2}(1+|\tau|)^3 \sqrt{L}e^{-\epsilon'L}
    \,=\,O_X\left((1+|\tau|)^3\right).\qedhere
\]
\end{proof}

\subsection{\texorpdfstring{Transforming $t\mapsto M \psi\left( -m\right) {\rm Res}_{s=-m} \Theta^{t}_\tau(s)$ to the geometric side as $\psi_m: u\mapsto \psi_m(u)$}{Transformation of Mpsi Res to the geometric term}}\label{subsec:Invert_res}
Recall we assume $\tau>20/\delta$, $M=\lfloor\tau\rfloor$, and $\epsilon=e^{(-\frac14+\delta)T}.$

\begin{lemma}\label{lem:psi_m}
For $m\ge 1$ and
\[
    h_m(t):=\left({\rm Res}_{s=-m} \Theta^t_\tau(s)\right)M \psi_{T,\epsilon}\left( -m\right),
\]
define $\psi_m$ as the inverse Selberg/Harish-Chandra transform of $h_m(t)$. Suppose $e^{T}\ge(1+|\tau|)^A$, $A\ge32$, and $M\epsilon<1$. Then, we have
\[
    \sum_{m=1}^{M}\psi_m(u)
    =C\,u^{-1}+(1+|\tau|)^{-2}O(u^{-1.5}),
\]
where $C=O((1+|\tau|)^{-3})$. The implied constants are independent of $\tau$.
\end{lemma}
\begin{proof}
    By the formula of $\Theta^t_\tau(s)$ in Theorem \ref{thm:ZTF}, we have
    \[
        {\rm Res}_{s=-m} \Theta^t_\tau(s)
        =
        (-1)^m \frac{2^{2m+2}\Gamma(-\frac{1}{2}-m+it)\Gamma(-\frac{1}{2}-m-it)\cosh{\pi t}}{\sqrt{\pi}\,\Gamma(-\kappa_+-m)\Gamma(-\kappa_--m)},
    \]
    where $\kappa_\pm = \frac{1}{2}\left(\frac{1}{2}\pm i\tau\right)$. Note that it has simple poles at $t=\pm i/2$, and $h_m(t)$ is holomorphic on the strip $|\Im(t)|<1/2$.

    Applying \eqref{eqn:h_to_Mpsi_2} for $h_m(t)$, we obtain
    \[
        M\psi_m(s)
        =\frac{2^{2s-1}}{\pi^{3/2}}\Gamma(s-\kappa_+)\Gamma(s-\kappa_-)\sin(\pi s)
        \left(2\,{\rm Res}_{s'=\frac12-s}F_m(s')-
        \frac{1}{2\pi i}\int_{(0)} F_m(s') ds' \right)M\psi_{T,\epsilon}(-m),
    \]
    where
    \[
        F_m(s')
        :=
        \frac{(-1)^m2^{2m+2}}{\sqrt{\pi}\,\Gamma(-\kappa_+-m)\Gamma(-\kappa_--m)}G_m(s'),
    \]
    and
    \[
        G_m(s'):=\Gamma(-\tfrac{1}{2}-m+s')\Gamma(-\tfrac{1}{2}-m-s')\Gamma(\tfrac12-s-s')\Gamma(\tfrac12-s+s')s' \sin(\pi s')\cos(\pi s'),
    \]
    and
    $1/2<\Re(s)<1$.

    By Lemma \ref{lem:Diff-Contours}, we have
    \[
        \left(2\,{\rm Res}_{s'=\frac12-s}F_m(s')-
        \frac{1}{2\pi i}\int_{(0)} F_m(s') ds'\right)
        =
        \frac{(-1)^m2^{2m+2}}{\sqrt{\pi}\,\Gamma(-\kappa_+-m)\Gamma(-\kappa_--m)}
        \frac{\pi\,\Gamma(1-s)^2}{2(m!)^2(m+s)},
    \]
    and thus, by using the reflection formula \eqref{eqn:reflection}, we obtain
    \begin{align*}
        M\psi_m(s)
        &=
        \frac{2^{2s-1}}{\pi^{3/2}}\Gamma(s-\kappa_+)\Gamma(s-\kappa_-)\sin(\pi s)
        \frac{(-1)^m2^{2m+2}}{\sqrt{\pi}\,\Gamma(-\kappa_+-m)\Gamma(-\kappa_--m)}
        \frac{\pi\,\Gamma(1-s)^2}{2(m!)^2(m+s)}M\psi_{T,\epsilon}(-m)\\
        &=\left(\frac{(-1)^m 2^{2m}}{(m!)^2\Gamma(-\kappa_+-m)\Gamma(-\kappa_--m)}\right)
        \left(\frac{2^{2s}\Gamma(1-s)\Gamma(-\kappa_++s)\Gamma(-\kappa_-+s)}{\Gamma(s)(m+s)}\right)M\psi_{T,\epsilon}(-m).
    \end{align*}
    The Mellin inversion formula yields
    \[
        \psi_m(u)
        =
        \underbrace{\left(\frac{(-1)^m 2^{2m}M\psi_{T,\epsilon}(-m)}{(m!)^2\Gamma(-\kappa_+-m)\Gamma(-\kappa_--m)}\right)}_{:=P_m}
        \left(
        \frac{1}{2\pi i}
        \int_{(\sigma)}
        \underbrace{\frac{2^{2s}\Gamma(1-s)\Gamma(-\kappa_++s)\Gamma(-\kappa_-+s)}{\Gamma(s)(m+s)}u^{-s}}_{:=A_m(s)}\,ds
        \right).
    \]

    Let us first estimate $\frac{1}{2\pi i}\int_{(\sigma)}A_m(s)\,ds$.  We shift the contour $(\sigma)$ to $(1.5)$, picking up the residue at $s=1$,
    \[
        {\rm Res}_{s=1}A_m(s)=-4\frac{\Gamma(1-\kappa_+)\Gamma(1-\kappa_-)}{m+1}u^{-1}.
    \]
    On the other hand, by \eqref{eqn:Stirling_vertical}, we have
    \begin{align*}
        |A_m(1.5+it)|
        &\ll \left|\frac{\Gamma(-0.5-it)}{\Gamma(1.5+it)(m+1.5+it)}\Gamma(\tfrac54+i(t+\tfrac\tau2))\Gamma(\tfrac54+i(t-\tfrac\tau2))u^{-1.5}\right|\\
        &\ll \frac{1}{(1+|t|)^2(m+1)} \left(1+\left|t-\frac{\tau}{2}\right|\right)^{3/4}\left(1+\left|t+\frac{\tau}{2}\right|\right)^{3/4}\exp\left(-\frac{\pi}{2}\left(\left|t+\frac\tau2\right|+\left|t-\frac\tau2\right|\right)\right)u^{-1.5}.
    \end{align*}
    Note that
    \begin{align*}
         &\left(1+\left|t-\frac{\tau}{2}\right|\right)^{3/4}\left(1+\left|t+\frac{\tau}{2}\right|\right)^{3/4}\exp\left(-\frac{\pi}{2}\left(\left|t+\frac\tau2\right|+\left|t-\frac\tau2\right|\right)\right)\\
         &\ll
         \begin{cases}
             \left((1+|\frac{\tau}{2}|)^2-t^2\right)^{3/4} e^{-\frac{\pi|\tau|}{2}} & |\frac\tau2|>|t|\\
             \left((1+|t|)^2-|\frac{\tau}{2}|^2\right)^{3/4} e^{-\pi|t|} & |\frac\tau2|\le|t|
         \end{cases}\\
         &\ll (1+|\tau|)^{3/2} e^{-\frac{\pi|\tau|}{2}}.
    \end{align*}
    Hence,
    \[
        \int_{(1.5)}
        |A_m(s)|\,ds\ll \frac{(1+|\tau|)^{3/2}}{e^{\pi\tau/2}(1+m)}u^{-1.5}.
    \]
    Therefore,
    \[
        \frac{1}{2\pi i}\int_{(\sigma)}A_m(s)\,ds=B_m u^{-1}+ e^{-\pi\tau/2}(1+|\tau|)^{3/2}O(u^{-1.5}),
    \]
    where the implicit constant is absolute and
    \[
        B_m
        =4\frac{\Gamma(1-\kappa_+)\Gamma(1-\kappa_-)}{m+1}.
    \]
    By using \eqref{eqn:Stirling_vertical}, we obtain
    \[
        |B_m|\asymp(1+|\tau|)^{1/2}(1+m)^{-1}e^{-\pi\tau/2}.
    \]

    Next, we show the following estimate of the prefactor $P_m$. Since $m\le \tau$, we have
    \[
        \frac{1}{\left|\Gamma(-\kappa_+-m)\Gamma(-\kappa_--m)\right|}
        =
        \frac{\prod_{j=1}^{m}\left((j+\frac14)^2+\frac{\tau^2}4\right)}{|\Gamma(-\kappa_+)\Gamma(-\kappa_-)|}\ll \left(\frac{5}{4}\right)^m(1+|\tau|)^{2m+3/2}e^{\pi\tau/2}.
    \]
    By $\epsilon M<1$ and Lemma \ref{lem:varrho}, we have $M\varrho(-\epsilon m)\ll e^{\epsilon m}\ll 1$. Hence,
    \[
        |M\psi_{T,\epsilon}(-m)|=|M\psi_T(-m)M\varrho(-\epsilon m)| \ll \frac{e^{-mT/4}}{m}\ll \frac{1}{(1+|\tau|)^{mA/4}}.
    \]
    For all $m\ge1$, we have
    \[
        |P_m|\ll \frac{5^m}{(m!)^2}(1+\tau)^{2m+3/2-mA/4}e^{\pi\tau/2}\ll (1+\tau)^{3/2-m(A/4-2)}e^{\pi\tau/2}.
    \]

    Therefore, 
    \begin{align*}
        \sum_{m=1}^{M}\psi_m(u)
        &= \sum_{m=1}^{M}P_m\left(B_m u^{-1}+ e^{-\pi\tau/2}(1+\tau)^{3/2}O(u^{-1.5})\right)\\
        &= \left(\sum_{m=1}^{M}P_mB_m\right) u^{-1}+ \left(\sum_{m=1}^{M}|P_m|\right)e^{-\pi\tau/2}(1+\tau)^{3/2}O(u^{-1.5}).
    \end{align*}
    Since $|P_mB_m|\ll (1+|\tau|)^{2-m(A/4-2)}\ll (1+|\tau|)^{4-A/4}$ and $M\le \tau$, we have
    \[
        \left|\sum_{m=1}^{M}P_mB_m\right|= O\left((1+|\tau|)^{5-A/4}\right)
    \]
    and
    \[
        \left(\sum_{m=1}^{M}|P_m|\right)e^{-\pi\tau/2}(1+\tau)^{3/2} =O\left((1+\tau)^{6-A/4}\right).
    \]
    Since $m\ge 1$, we obtain the desired conclusion.
\end{proof}

\begin{lemma}\label{lem:Diff-Contours}
For $\frac{1}{2}<\Re(s)<1$ and
\[
    G_m(s'):=\Gamma(-\tfrac{1}{2}-m+s')\Gamma(-\tfrac{1}{2}-m-s')
    \Gamma(\tfrac12-s-s')\Gamma(\tfrac12-s+s')s' \sin(\pi s')\cos(\pi s'),
\]
we have
\[
    \frac{1}{2\pi i} \int_{(0)}
    G_m(s')\, ds'
    =2\,{\rm Res}_{s'=\frac{1}{2}-s}G_m(s')-
    \frac{\pi\,\Gamma(1-s)^2}{2(m!)^2(m+s)}.
\]
\end{lemma}
\begin{proof}
Suppose $C$ is a contour between $-i\infty$ and $i\infty$ that separates the poles of $\Gamma \left(\frac{1}{2}-s-s'\right)$, which are
\[
    \left\{-s+\tfrac{1}{2}, -s+\tfrac{1}{2}+1, \dots \right\},
\]
from the poles of $\Gamma\left(-\frac{1}{2}-m+s'\right) \Gamma\left(\frac{1}{2}-s+s'\right)$, which are
\[
    \left\{m+\tfrac{1}{2}, m+\tfrac{1}{2}-1,\dots \right\} \cup \left\{ s-\tfrac{1}{2}, s-\tfrac{1}{2}-1,\dots \right\}.
\]
Using
\begin{align*}
    s'\Gamma(-\tfrac12-m-s')
    &=-(m+\tfrac12)\Gamma(-\tfrac12-m-s')-\Gamma(1-\tfrac12-m-s')\\
    &=-\frac{(m+\frac12)\pi}{\sin(\pi(-\frac12-m-s'))\, \Gamma(\frac32+m+s')}-\frac{\pi}{\sin(\pi(\frac12-m-s'))\, \Gamma(\frac12+m+s')}\\
    &=\frac{(-1)^m(m+\frac12)\pi}{\cos(\pi s')\, \Gamma(\frac32+m+s')}-\frac{(-1)^{m}\pi}{\cos(\pi s')\, \Gamma(\frac12+m+s')},
\end{align*}
we obtain
\begin{align}
\begin{split}\label{eqn:Barnes_add_lemma_cancellation}
    &\frac{1}{2\pi i} \int_{C}
    G_m(s') ds'\\
    &= \frac{(-1)^{m}\pi}{2\pi i} \left(m+\frac{1}{2}\right)
    \int_{C}\frac{\Gamma\left(-\frac{1}{2}-m+s'\right) \Gamma\left(\frac{1}{2}-s+s'\right)\Gamma \left(\frac{1}{2}-s-s'\right)\sin{\pi s'}}{\Gamma\left(\frac{3}{2}+m+s'\right)}ds'\\
    &\hspace{10pt}-\frac{(-1)^{m}\pi}{2\pi i}
    \int_{C}\frac{\Gamma\left(-\frac{1}{2}-m+s'\right) \Gamma\left(\frac{1}{2}-s+s'\right)\Gamma \left(\frac{1}{2}-s-s'\right) \sin{\pi s'}}{\Gamma\left(\frac{1}{2}+m+s'\right)} ds'.
\end{split}
\end{align}

We show $\frac{1}{2\pi i} \int_{C} G_m(s') ds'=0$ by using the following identity \cite[Eq.\@ (22)]{Jantzen_NewProofsBarnesLem}:
\begin{equation}\label{eqn:add_lemma}
    \frac{1}{2i\pi}\int_{-i\infty}^{i\infty} e^{\pm i\pi z}\frac{\Gamma(\alpha-z)\Gamma(\beta_1+z)\Gamma(\beta_2+z)}{\Gamma(\gamma+z)}dz
    =
    e^{\pm i \pi \alpha} \frac{\Gamma(\alpha+\beta_1)\Gamma(\alpha+\beta_2)\Gamma(\gamma-\alpha-\beta_1-\beta_2)}{\Gamma(\gamma-\beta_1)\Gamma(\gamma-\beta_2)}.
\end{equation}
Since the identity holds for both $+$ and $-$ signs in the exponential powers, we can replace $e^{\pm i \pi z}$ and $e^{\pm i\pi\alpha}$ with $\sin \pi z$ and $\sin \pi\alpha$.
As discussed in \cite{Jantzen_NewProofsBarnesLem}, this integral representation remains valid whenever $\Re(\gamma-\alpha-\beta_1-\beta_2)>0$, provided that the contour $C$ separates the left sequence of poles from the right sequence of poles. Therefore, for the first integral on the right-hand side in \eqref{eqn:Barnes_add_lemma_cancellation}, we may take $\alpha=\tfrac12-s$, $\beta_1=\tfrac12-s$, $\beta_2=-\tfrac12-m$, and $\gamma=\tfrac32+m$.
For the second integral, the same argument applies with $\gamma$ replaced by $\tfrac12+m$. It is then immediate that the two resulting expressions cancel, and hence \eqref{eqn:Barnes_add_lemma_cancellation} is equal to $0$.

Hence,
\begin{align*}
    \frac{1}{2\pi i} \int_{(0)} G_m(s')
     ds'
    &=\frac{1}{2\pi i} \int_{(0)-C} G_m(s')
     ds'\\
    &=\left({\rm Res}_{s'=\frac{1}{2}-s}G_m(s')-{\rm Res}_{s'=-\frac{1}{2}+s}G_m(s')\right)
    - \sum\limits_{0\le k \le m}{\rm Res}_{s'=\frac{1}{2}+k} G_m(s')\\
    &=2\,{\rm Res}_{s'=\frac{1}{2}-s}G_m(s')
     - \sum\limits_{0\le k \le m}{\rm Res}_{s'=\frac{1}{2}+k} G_m(s').
\end{align*}
We have
\begin{align*}
    {\rm Res}_{s'=\frac{1}{2}+k}G_m(s')
    &=
    -\pi \left(k+\frac{1}{2}\right) \frac{\Gamma\left(-s-k\right) \Gamma\left(1-s+k\right)}{\left(m-k\right)! \left( 1+m+k\right)!}\\
    &=\frac{-\pi}{2(m+s)}(g_{k+1}-g_k),
\end{align*}
where
\[
    g_k:=\frac{\Gamma(1-s-k)\Gamma(1-s+k)}{(m-k)!(m+k)!}
\]
with $g_{m+1}=0$.
Therefore
\[
    \sum\limits_{0\le k \le m}{\rm Res}_{s'=\frac{1}{2}+k} G_m(s')=\frac{\pi\,\Gamma(1-s)^2}{2(m+s)(m!)^2}.\qedhere
\]
% On the other hand, by using $\Gamma(1-2s)\sin(2\pi s)=\tfrac\pi{\Gamma(2s)}$, we have
% \begin{align*}
%     2\,{\rm Res}_{s'=\frac{1}{2}-s}G_m(s')
%     &=\Gamma(-m-s)\Gamma(-1-m+s)\Gamma(1-2s)\left(s-\frac12\right)\sin(2\pi s)\\
%     &=\Gamma(-m-s)\Gamma(-1-m+s)\frac{\pi\left(s-\frac12\right)}{\Gamma(2s)}\\
%     &=\frac{\pi\Gamma(-m-s)\Gamma(-1-m+s)}{2\Gamma(2s-1)}.
% \end{align*}
% By applying the reflection formula, we factor out $\Gamma(1-s)^2$ as follows.
% \begin{align*}
%     \Gamma(-m-s)\Gamma(-1-m+s)
%     &=\frac{\pi^2}{\sin^2(\pi s)\Gamma(m+s+1)\Gamma(m-s+2)}\\
%     &=\frac{\Gamma(s)^2\Gamma(1-s)^2}{\Gamma(m+s+1)\Gamma(m-s+2)}.
% \end{align*}
% Therefore,
% \[
%     2\,{\rm Res}_{s'=\frac{1}{2}-s}G_m(s')
%     =\frac{\pi\Gamma(s)^2\Gamma(1-s)^2}{2\Gamma(2s-1)\Gamma(m+s+1)\Gamma(m-s+2)}.\qedhere
% \]
\end{proof}

\subsection{\texorpdfstring{Spectral test functions $h(t)$ having simple poles at $t=\pm i/2$}{Spectral test functions h(t) having simple poles at t=±i/2}}\label{subsec:h_has_poles_at_i/2}

As a function of $t$, the residue ${\rm Res}_{s=-m}\Theta^{t}_\tau(s)$ has simple poles at $t=\pm i/2$, violating the condition (2) of the admissibility of spectral test functions. Hence, we need to verify that the conclusion of
Theorem~\ref{thm:ZTF} remains valid for $h_m(t):=M\psi(-m)\operatorname{Res}_{s=-m}\Theta_\tau^t(s)$ and the corresponding geometric test function $\psi_m$ defined in
Section~\ref{subsec:Invert_res}.

\begin{proposition}\label{prop:ZTF_test_ftn_poles_pmi/2}
    Suppose that
    \[
        h(t)=\frac{C}{t^2+1/4}+h^{good}(t), \quad 0\neq C\in \mathbb{C}
    \]
    such that
    \begin{enumerate}
        \item $h^{good}(t)$ is an even holomorphic function on the strip $\{t:|\Im(t)|<1/2+\epsilon_1\}$, and
        \item $|h(t)|\ll (1+|t|)^{-(2+\delta)}$ as $|t|\to \infty$ within the strip $\{t:|\Im(t)|<1/2+\epsilon_1\}$, where $\delta>0$.
    \end{enumerate}
    Let $\psi$ denote the geometric test function corresponding to $h$, obtained from Lemma~\ref{lem:h_to_Mpsi}. Then Theorem~\ref{thm:ZTF} holds for $\psi$ and $h$, i.e., for a non-constant eigenfunction $\varphi$, we have
    \[
        \sum_{j\ge1}\langle\varphi\phi_j,\phi_j\rangle h(t_j)
        =
        \sum_{\{\gamma\}} \psi\left(\frac{(e^{l(\gamma)}-1)^2}{e^{l(\gamma)}}\right)\int_{\gamma_0}\varphi(s)\,ds,
    \]
    where the right-hand side conditionally converges for $\sum_{\{\gamma\}}=\lim_{T\to \infty}\sum_{l(\gamma)<T}$.
\end{proposition}
\begin{proof}
Consider a deformation family $(h_r(t))_{r\in(0,\epsilon_1)}$ of meromorphic functions on $\{t : |\Im(t)|<1/2+\epsilon_1\}$  defined by
\[
    h_r(t)=h^{bad}_{r}(t)+h^{good}(t),
\]
where
\[
    h^{bad}_{r}(t):=\frac{C}{t^2+(r+1/2)^2}.
\]
Then,
\[
    R_r:={\rm Res}_{t=i(r+1/2)}h_r(t)=-\frac{iC}{2r+1},
\]
and thus $|R_r|=O(1)$.

Fix $r\in(0,\epsilon_1)$. The function $h_r$ is holomorphic on the strip $\{t: |\Im(t)|<1/2+r\}$. By using
\begin{equation}\label{eqn:h_r-h_0}
    |h_r(t)-h(t)|=r\left|\frac{C(1+r)}{(t^2+(r+1/2)^2)(t^2+(1/2)^2)}\right|=r\,O_X\!\left((1+|t|)^{-4}\right)
\end{equation}
as $|t|\to \infty$ and assumption (2), we obtain $|h_r(t)|\ll_{r,r'}(1+|t|)^{-\min\{2+\delta,4\}}$ uniformly in the strip $|\Im(t)|<\frac12+r'$, where $r'\in(0,r)$. Hence, for all $r\in (0,\epsilon_1)$, the function $h_r$ is admissible, and thus both the Selberg and Zelditch trace formulas apply to $h_r$. Let $K_r$ denote the automorphic kernel associated with $h_r$, and let $k_r$ and $\psi_r$ denote the corresponding geometric test functions in the Selberg and Zelditch trace formulas, respectively. 

By using \eqref{eqn:h_r-h_0}, we obtain the convergence on the spectral side
\begin{equation}\label{eqn:conv_spec_side}
    {\rm Tr}({\rm Op}(\varphi)\circ T_{K_r})=\sum_{j=1}^\infty \langle \varphi\phi_j,\phi_j\rangle h_r(t_j)\to \sum_{j=1}^\infty \langle \varphi\phi_j,\phi_j\rangle h(t_j)
    \quad
    \text{as }
    r\to 0.
\end{equation}

Note that, by Proposition \ref{prop:Weighted_Length_Sum_marginal}, the series
\[
    \sum_{\{\gamma\}} \psi\left(\frac{(e^{l(\gamma)}-1)^2}{e^{l(\gamma)}}\right)\int_{\gamma_0}\varphi(s)\,ds
\]
is (conditionally) convergent, where $\sum_{\{\gamma\}}=\lim_{T\to \infty}\sum_{l(\gamma)<T}$.

Now, by \eqref{eqn:conv_spec_side}, to verify the Zelditch trace formula for $h$ and $\psi$, it suffices to show that
\begin{equation}\label{eqn:conv_geom_side}
    \lim_{r\to 0}\left(\sum_{\{\gamma\}} \psi_r\left(\frac{(e^{l(\gamma)}-1)^2}{e^{l(\gamma)}}\right)\int_{\gamma_0}\varphi(s)\,ds\right)
    =
    \sum_{\{\gamma\}} \psi\left(\frac{(e^{l(\gamma)}-1)^2}{e^{l(\gamma)}}\right)\int_{\gamma_0}\varphi(s)\,ds.
\end{equation}

We show \eqref{eqn:conv_geom_side} by verifying
\begin{enumerate}
    \item the pointwise convergence, that is,  $\lim_{r\to 0}\psi_r(u)=\psi(u)$ for all $u\in (0,\infty)$, and
    \item the uniform decay of tails for $\psi_r$, that is, 
    \[
        \sup_{r\in(0,\epsilon_1)}\left|\sum_{l(\gamma)>T} \psi_r\left(\frac{(e^{l(\gamma)}-1)^2}{e^{l(\gamma)}}\right)\int_{\gamma_0}\varphi(s)ds\right|
        \to 0
        \quad
        \text{as }
        T\to \infty.
    \]
\end{enumerate}

Let $\psi_r^{bad}$ and $\psi^{good}$ denote the geometric test functions of the Zelditch trace formula associated with $h_r^{bad}$ and $h^{good}$, respectively. Since $h_r^{bad}$ and $h^{good}$ are not admissible, we do not claim that they satisfy the trace formula. However, we can still apply the inverse Selberg/Harish-Chandra transform (Lemma~\ref{lem:h_to_Mpsi}) to $h_r^{bad}$ and $h^{good}$. In particular, we have $\psi_r^{bad}$ and $\psi^{good}$.

By Lemma~\ref{lem:psi_bad_r}, (1) follows immediately.

Let us show (2). For $r\in (0,\epsilon_1)$, by Lemma \ref{lem:psi_bad_r},
\[
    \sum_{\{\gamma\}} \psi^{bad}_r\left(\frac{(e^{l(\gamma)}-1)^2}{e^{l(\gamma)}}\right)\int_{\gamma_0}\varphi(s)\,ds
\]
is absolutely convergent. Since $h_r$ is admissible, its geometric side in the Zelditch trace formula is absolutely convergent. Therefore, 
\[
    \sum_{\{\gamma\}} \psi^{good}\left(\frac{(e^{l(\gamma)}-1)^2}{e^{l(\gamma)}}\right)\int_{\gamma_0}\varphi(s)\,ds
\]
is also absolutely convergent. Hence we may replace $\psi_r$ in the statement (2) with $\psi_r^{bad}$. We have
\begin{align*}
    \left|\sum_{l(\gamma)\ge T} \psi^{bad}_r\left(\frac{(e^{l(\gamma)}-1)^2}{e^{l(\gamma)}}\right)\int_{\gamma_0}\varphi(s)ds\right|
    &\ll_{\epsilon_1,\tau}
    \left|\sum_{l(\gamma)\ge T}e^{-(1+r)l(\gamma)}\int_{\gamma_0}\varphi(s)ds\right|+\sum_{l(\gamma)\ge T}e^{-2l(\gamma)}\int_{\gamma_0}|\varphi(s)|ds\\
    &\ll_\tau \left|\int_T^\infty e^{-(1+r)x}d\Psi_\tau(x)\right|+\sum_{l(\gamma)\ge T}l(\gamma_0)e^{-2l(\gamma)},
\end{align*}
where $\Psi_\tau(T)=\sum_{l(\gamma)<T}\int_{\gamma_0}\varphi\, ds$. By Lemma \ref{lem:wPrimLengthSum}, we have
\[
    \sum_{l(\gamma)\ge T}l(\gamma_0)e^{-2l(\gamma)}=O_X(e^{-T}).
\]
Recall from Proposition~\ref{prop:Zelditch_Equidist} that $|\Psi_\tau(T)|\ll_\tau e^{(1-\eta_X)T}$ for some $\eta_X>0$. Hence,
\[
    \left|\int_T^\infty e^{-(1+r)x}d\Psi_\tau(x)\right|
    \ll_{\tau,\epsilon_1} \left|e^{-(1+r)T}\Psi_\tau(T)\right|+ \left|\int_T^\infty e^{-(1+r)x}\Psi_\tau(x) dx\right|
    \ll_X e^{-(r+\eta_X)T},
\]
and thus we obtain
\[
    \sup_{r\in(0,\epsilon_1)}\left|\sum_{l(\gamma)>T} \psi^{bad}_r\left(\frac{(e^{l(\gamma)}-1)^2}{e^{l(\gamma)}}\right)\int_{\gamma_0}\varphi(s)\,ds\right|=O_{X,\epsilon_1,\tau}(e^{-\eta_XT}).
\]
\end{proof}

\begin{lemma}\label{lem:psi_bad_r}
For all $r\in(0,\epsilon_1)$ and
\[
    h^{bad}_{r}(t):=\frac{C}{t^2+(r+1/2)^2},
\]
where $0\neq C\in \mathbb{C}$, we have
\begin{equation}\label{eqn:psi_bad_r}
    \psi^{bad}_r(u)=\frac{C}{4\sqrt{\pi}}\left(\frac{4}{u}\right)^{1+r}
    \frac{\Gamma(1+r-\kappa_+)\Gamma(1+r-\kappa_-)}{\Gamma(2+2r)}{}_2F_1\left(1+r-\kappa_+, 1+r-\kappa_-, 2+2r ; -\frac{4}{u}\right).
\end{equation}
Consequently, for all $r\in (0,\epsilon_1)$, we have
\[
    \psi_r^{bad}(u)=C_ru^{-1-r}+O_{\epsilon_1}(|C|u^{-2})
\]
as $u\to \infty$, where
\[
    C_r\ll_{\epsilon_1} |C|\,(1+|\tau|)^{\frac12+2\epsilon_1}e^{-\pi|\tau|/2}.
\]
In particular, the implicit constants are independent of $r$.
\end{lemma}
\begin{proof}
By \eqref{eqn:h_to_Mpsi_1}, we have
\[
    M\psi_r^{bad}(s)=-\frac{2^{2s-1}}{\pi^{3/2}}\Gamma(s-\kappa_+)\Gamma(s-\kappa_-)\left(\frac{1}{2\pi i}\int_{(-\eta)}A(s')\, ds'\right),
\]
where $1/2<\Re(s)<1$, $\Re(s)-1/2<\eta<1/2$, and
\[
    A(s')= \frac{C\pi s'}{(r+1/2)^2-s'^2}\frac{\Gamma(\tfrac12-s-s')}{\Gamma(\frac{1}{2}+s-s')}.
\]
On the left of the contour $(-\eta)$, the function $A(s')$ has only one pole, at $s'=-r-\frac{1}{2}$. Therefore,
\[
    \frac{1}{2\pi i}\int_{(-\eta)}A(s')\, ds'={\rm Res}_{s'=-r-\frac{1}{2}}A(s')=-\frac{C\pi}{2}\frac{\Gamma(1+r-s)}{\Gamma(1+r+s)}.
\]
Thus,
\[
    M\psi_r^{bad}(s)=\frac{C}{\sqrt{\pi}}2^{2s-2}\Gamma(s-\kappa_+)\Gamma(s-\kappa_-)\frac{\Gamma(1+r-s)}{\Gamma(1+r+s)}.
\]
By the Mellin inversion formula,
\[
    \psi^{bad}_r(u)=\frac{C}{4\sqrt{\pi}}\frac{1}{2\pi i}\int_{(\sigma)}  \frac{\Gamma(s-\kappa_+)\Gamma(s-\kappa_-)\Gamma(r+1-s)}{\Gamma(r+1+s)}\left(\frac{4}{u}\right)^s\,ds.
\]
By the change of variable $w=s-r-1$, we obtain
\begin{align*}
    \psi^{bad}_r(u)=\frac{C}{4\sqrt{\pi}}\left(\frac{4}{u}\right)^{r+1}\frac{1}{2\pi i}\int_{(\sigma-r-1)}  \frac{\Gamma(1+r-\kappa_++w)\Gamma(1+r-\kappa_-+w)\Gamma(-w)}{\Gamma(2r+2+w)}\left(\frac{4}{u}\right)^w\,dw.
\end{align*}
Recall $\kappa_\pm=\tfrac14\pm\tfrac12i\tau$. Then, the Barnes integral formula yields \eqref{eqn:psi_bad_r}.

Let
\[
    G_r(u):=\frac{\Gamma(1+r-\kappa_+)\Gamma(1+r-\kappa_-)}{\Gamma(2+2r)}{}_2F_1\left(1+r-\kappa_+, 1+r-\kappa_-, 2+2r ; -\frac{4}{u}\right).
\]
When $u>4$, 
\[
    G_r(u)
    =\sum_{n\ge0}\frac{\Gamma(n+1+r-\kappa_+)\Gamma(n+1+r-\kappa_-)}{\Gamma(n+2+2r)}\frac{(-4)^n}{n!}u^{-n}.
\]
Since $|\Gamma(x+iy)|\le \Gamma(x)$ for all $x>0$ and $y\in \mathbb{R}$, we have
\[
    \frac{\Gamma(n+1+r-\kappa_+)\Gamma(n+1+r-\kappa_-)}{\Gamma(n+2+2r)}
    \ll
    \frac{\Gamma(n+\frac34+r+\frac12\Im(\tau))\Gamma(n+\frac34+r-\frac12\Im(\tau))}{\Gamma(n+2+2r)}
    \ll_{\epsilon_1}
    n!\,(1+n)^{-3/2}.
\]
Therefore,
\[
    G_r(u)=\frac{\Gamma(1+r-\kappa_+)\Gamma(1+r-\kappa_-)}{\Gamma(2+2r)}+O_{\epsilon_1}(u^{-1})
\]
as $u\to \infty$, where the implicit constant is independent of $r$. Then
\[
    C_r=\frac{C4^r}{\sqrt{\pi}}\frac{\Gamma(1+r-\kappa_+)\Gamma(1+r-\kappa_-)}{\Gamma(2+2r)}\ll_{\epsilon_1}C\,(1+|\tau|)^{\frac12+2\epsilon_1}e^{-\pi|\tau|/2}.
\]

\end{proof}

\begin{lemma}\label{lem:h_m satisfies conditions}
    For $m\ge 1$ and $\epsilon_1<1$, the function $h_m(t)=M\psi(-m)\operatorname{Res}_{s=-m}\Theta_\tau^t(s)$ satisfies the assumptions for $h(t)$ in Proposition~\ref{prop:ZTF_test_ftn_poles_pmi/2}.
\end{lemma}
\begin{proof}
    Note that the poles of $h_m(t)$ are at $t=\pm i(\frac12+n)$, where $n=0,1,2,\dots,m$. Hence $\pm i/2$ are the only poles in the strip $|\Im(t)|<1/2+\epsilon_1$.
    
    Since
    \begin{align*}
        \Gamma(-\tfrac{1}{2}-m+it)\Gamma(-\tfrac{1}{2}-m-it)
        &=\frac{\Gamma(\tfrac12+it)\Gamma(\tfrac12-it)}{\prod_{j=0}^m (-\frac12-j+it)(-\frac12-j-it)}\\
        &=\frac{\pi/\cosh(\pi t)}{\prod_{j=0}^m (t^2+(\frac12+j)^2)},
    \end{align*}
    we have
    \begin{align*}
        {\rm Res}_{s=-m} \Theta^t_\tau(s)
        &=
        (-1)^m \frac{2^{2m+2}\Gamma(-\frac{1}{2}-m+it)\Gamma(-\frac{1}{2}-m-it)\cosh{\pi t}}{\sqrt{\pi}\,\Gamma(-\kappa_+-m)\Gamma(-\kappa_--m)}\\
        &=\frac{(-1)^m2^{2m+2}\sqrt{\pi}}{\Gamma(-\kappa_+-m)\Gamma(-\kappa_--m)} \frac{1}{\prod_{j=0}^m (t^2+(\frac12+j)^2)}.
    \end{align*}
    Therefore, it suffices to show that
    \[
        P_m(t):=\frac{1}{\prod_{j=0}^m \left(t^2+(\frac12+j)^2\right)}
    \]
    satisfies assumptions~(1) and~(2) for $h(t)$ in Proposition~\ref{prop:ZTF_test_ftn_poles_pmi/2}. It is immediate that
    \[
        |P_m(t)|\ll_M (1+|t|)^{-(2m+2)}
    \]
    as $|t|\to\infty$ within the strip $\{t:|\Im(t)|<M\}$. Hence assumption~(2) holds. 
    Assumption~(1) follows from the partial fraction decomposition
    \[
        P_m(t)
        =
        \sum_{k=0}^m \frac{C_{m,k}}{t^2+(\frac12+k)^2},
    \]
    where
    \[
        C_{m,k}
        =
        \frac{1}{
        \prod_{\scriptstyle 0\le j\le m,\;j\neq k}
        \left((j+\frac12)^2-(k+\frac12)^2\right)
        }.\qedhere
    \]
\end{proof}

\subsection{
\texorpdfstring{Proof of Theorem \ref{theorem:main} assuming \eqref{eqn:Spec_Side}}
{Proof of Main Theorem}
}\label{subsec:proof_main_thm}

In this subsection, we prove Theorem \ref{theorem:main} assuming a technical estimate \eqref{eqn:Spec_Side}. The proof of \eqref{eqn:Spec_Side} will be given in the subsequent subsections.

Recall that we use the notation
\[
    \phi_k=\varphi
    \qquad
    \text{and}
    \qquad
    t_k=\tau.
\]

\noindent\textbf{Constants $\epsilon,\delta,T,A(\delta)$.} We will take $\delta>0$ as an arbitrarily small number and $\epsilon=e^{-(\frac14-\delta)T}$. We define
\[
    A(\delta)=\lfloor  20/\delta\rfloor.
\]

The parameter $\delta$ will be replaced with $\epsilon$ in the statement of Theorem \ref{theorem:main}.
\medskip

\noindent{\it Step 1: Estimate of $\sum_{l(\gamma)<T} \int_{\gamma_0} \phi_k ds$.}

We first show
\begin{equation}\label{eqn:Equidistribution_phi_k_polynomial}
    \sum_{l(\gamma)<T} \int_{\gamma_0} \phi_k ds
    =
    \sum_{0<\lambda_j<3/16} \gamma_{k,j} \langle\phi_k,\phi_j^2\rangle\,e^{(\frac12+\Im(t_j))T}
    +(1+|t_k|)^{A(\delta)} O_{X,\delta}\left( e^{(\frac34+\delta)T}\right).
\end{equation}

Proposition \ref{prop:Zelditch_Equidist} proves \eqref{eqn:Equidistribution_phi_k_polynomial} for finitely many $\tau$. Hence, in what follows, we may assume that $\tau>20/\delta$.
\medskip

\noindent{\it Case 1: $e^{T}< (1+|\tau|)^{A(\delta)}$.}

We have
\[
    e^{(\frac12+\Im(t_j))T}< (1+|\tau|)^{A(\delta)(\Im(t_j)-\frac14)}e^{\frac34T},
\]
and, by \eqref{eqn:gamma_kj TriplePDT}, for all $j$'s with $0<\lambda_j<1/4$,
\[
    |\gamma_{k,j}\langle\phi_k,\phi_j^2\rangle|\ll (1+|\tau|)^{3/2}.
\]

Since $1-(\Im(t_j)-\frac14)\gg_X 1$, for all sufficiently small $\delta\ll_X1$, the main term can be absorbed in the error term, that is,
\[
    \sum_{j:\lambda_j<3/16} \gamma_{k,j} \langle\phi_k,\phi_j^2\rangle\,e^{(\frac12+\Im(t_j))T}=(1+|\tau|)^{A(\delta)}O_{X}\left(e^{\frac34T}\right).
\]
By \eqref{eqn:PrimLengthSum} and $\|\varphi\|_\infty\ll_X(1+|\tau|)^{1/2}$, we have
\begin{equation}\label{eqn:Equidist_small_T}
    \left|\sum_{l(\gamma)<T} \int_{\gamma_0} \varphi(s)\, ds\right|
    \ll(1+|\tau|)^{1/2}\sum_{l(\gamma)<T} l(\gamma_0)
    \ll (1+|\tau|)^{1/2}\,e^T
    \ll (1+|\tau|)^{A(\delta)}\,e^{\frac34T}.
\end{equation}
\medskip

\noindent{\it Case 2: $e^{T}\ge (1+|\tau|)^{A(\delta)}$.}

We apply Theorem \ref{thm:ZTF} with $\psi=\psi_\tau=\psi_{T,\epsilon}-\sum_{m=1}^M \psi_m$. Note that, by Proposition~\ref{prop:ZTF_test_ftn_poles_pmi/2} and Lemma~\ref{lem:h_m satisfies conditions}, Theorem~\ref{thm:ZTF} holds for $\psi_m$ and $h_m$.

The geometric side is different from the geometric side of Proposition~\ref{prop:Zelditch_Equidist} by
\[
    \text{(Difference)}
    =
    \sum_{\{\gamma\}} \left(\sum_{m=1}^M\psi_m\left(\frac{(e^{l(\gamma)}-1)^2}{e^{l(\gamma)}}\right)\right)\int_{\gamma_0}\varphi(s)\,ds,
\]
and, by substituting $\epsilon=e^{-(\frac14-\delta)T}$ in \eqref{eqn:ZTF_geom_side}, we have
\[
    \text{(Geometric side)}=\sum_{l(\gamma) < T}   \int_{\gamma_0}\varphi(s)ds
    +
    (1+|\tau|)^{1/2}O_{X,\delta}\left(e^{(\frac34+\delta)T}\right)
    -
    (\text{Difference}).
\]
By Lemma \ref{lem:psi_m}, we have
\begin{align*}
    \text{(Difference)}
    =
    C \sum_{\{\gamma\}} e^{-l(\gamma)}\int_{\gamma_0}\varphi(s)\,ds
    +
    (1+|\tau|)^{-2}\,O\left((1+|\tau|)^{1/2}
    \sum_{\{\gamma\}} \,e^{-(1.5)l(\gamma)} l(\gamma_0)
    \right),
\end{align*}
where $C=O\left((1+|\tau|)^{-3}\right)$. We apply Proposition~\ref{prop:Weighted_Length_Sum_marginal} to the first term and Lemma~\ref{lem:wPrimLengthSum} to the big-O term and obtain
\[
    \text{(Difference)}=O_X(1).
\]

Suppose we obtain 
\begin{equation}\label{eqn:Spec_Side}
    \text{(Spectral side)} = \sum_{0<\lambda_j<3/16} \gamma_{k,j} \langle\varphi\phi_j,\phi_j\rangle\,e^{(\frac12+\Im(t_j))T} +(1+|\tau|)^3\,O_{X,\delta}\left(e^{(\frac34+\delta) T}\right),
\end{equation}
uniformly for all $e^{T}>(1+|\tau|)^{A(\delta)}$.

Then, by combining the spectral side and the geometric side, for all $e^{T}>(1+|\tau|)^{A(\delta)}$, we obtain
\begin{equation}\label{eqn:Equidist_large_T}
    \sum_{l(\gamma)<T} \int_{\gamma_0} \varphi \, ds
    =
    \sum_{0<\lambda_j<3/16} \gamma_{k,j} \langle\varphi\phi_j,\phi_j\rangle\,e^{(\frac12+\Im(t_j))T}
    +(1+|\tau|)^3 O_{X,\delta}\left(e^{(\frac34+\delta)T}\right).
\end{equation}

By combining \eqref{eqn:Equidist_small_T} and \eqref{eqn:Equidist_large_T}, we see that \eqref{eqn:Equidistribution_phi_k_polynomial} holds for all $T\gg_X1$.
\medskip

\noindent{\it Step 2: Estimate of $\sum_{l(\gamma)<T} \int_{\gamma_0} f ds$.}

For all $f\in C^\infty(X)$, by the spectral decomposition, we have
\begin{align*}
    \sum_{l(\gamma)<T} \int_{\gamma_0} f ds &= \sum_{l(\gamma)<T} \int_{\gamma_0} \left(\sum_{k\ge 0} \langle f,\phi_k \rangle \phi_k(s) \right) ds\\
    & = \left(\sum_{l(\gamma)<T} \int_{\gamma_0} 1 ds\right) \frac{\int_X f d\mu}{{\rm Vol}(X)}
    + \sum_{k\ge 1} \langle f,\phi_k \rangle \sum_{l(\gamma)<T}\int_{\gamma_0} \phi_k(s)ds.
\end{align*}
By \eqref{eqn:Equidistribution_phi_k_polynomial}, we have
\begin{align*}
    &\sum_{k\ge 1} \langle f,\phi_k \rangle \sum_{l(\gamma)<T}\int_{\gamma_0} \phi_k(s)ds\\
    &=\sum_{k\ge1} \,\sum_{j:\lambda_j<3/16} \langle f,\phi_k\rangle\gamma_{k,j} \langle\phi_k,\phi_j^2\rangle\,e^{(\frac12+\Im(t_j))T}
    + \sum_{k\ge1}(1+|t_k|)^{A(\delta)} \langle f,\phi_k\rangle O_{X,\delta}( e^{(\frac34+\delta)T}).
\end{align*}

Note that for all $f\in C^{\infty}(X)$, $k>0$, and $m\ge 0$, we have
\[
    |\lambda_k^m \langle f,\phi_k\rangle| \ll_{X,m}\|f\|_{C^{2m}}.
\]
By using \eqref{eqn:Stirling_vertical} and Theorem \ref{thm:TripleProduct}, we obtain
\[
    \gamma_{k,j}\ll_j
    \, (1+ |t_k|)^{-2\Im(t_j)+\frac12}
    \exp(\pi t_k/2)
    \quad
    {\rm and}
    \quad
    \langle\phi_k\phi_j,\phi_j\rangle
    \ll_j
    (1+|t_k|)(1+\log(1+|t_k|))
    \exp(-\pi t_k/2)
\]
as  $k\to\infty$.
Hence, by Lemma \ref{lem:sum(1+t)^delta} and $\lambda_k=1/4+t_k^2$, we obtain
\[
    E(f;j):=\sum_{k\ge1} \langle f,\phi_k\rangle \gamma_{k,j}\langle\phi_k,\phi_j^2\rangle\ll_j \sum_{k\ge1} (1+\log(1+|t_k|))(1+|t_k|)^{-5/2-2\Im(t_j)}\lambda_k^2|\langle f,\phi_k\rangle|  \ll \|f\|_{C^4}.
\]
We also have
\[
    \sum_{k\ge1}(1+|t_k|)^{A(\delta)} |\langle f,\phi_k\rangle|
    \ll
    \sum_{k\ge1}(1+|t_k|)^{-3} \left|\lambda_k^{\lceil(A(\delta)+3)/2\rceil}\langle f,\phi_k\rangle\right|
    \ll\|f\|_{C^{D(\delta)}},
\]
where $D(\delta)=2\left\lceil\frac{A(\delta)+3}{2}\right\rceil$. Note that $D(\delta)\sim 20/\delta$.

Therefore,
\begin{align*}
    \sum_{l(\gamma)<T} \int_{\gamma_0} f ds 
     - \left(\sum_{l(\gamma)<T} \int_{\gamma_0} 1 ds\right) \frac{\int_X f d\mu}{{\rm Vol}(X)}&=\sum_{j:\lambda_j<3/16} E(f;j)\,e^{(\frac12+\Im(t_j))T}+\|f\|_{C^{D(\delta)}}O_{X,\delta}\left(e^{(\frac34+\delta)T}\right)\\
     &\ll_{X,\delta}\|f\|_{C^{D(\delta)}}\,\max\{e^{(\frac12 +\Im(t_1))T},e^{(\frac34+\delta)T}\}.
\end{align*}
By dividing both sides by $\sum_{l(\gamma)<T} \int_{\gamma_0} 1 ds$ and using \eqref{eqn:PrimLengthSum}, we obtain Theorem~\ref{theorem:main}-(1).

For the sharpness of $e^{(\frac12+|t_1|)T}$, we only need to show that 
\[
    \sum_{j:\lambda_j=\lambda_1}E(f;j)\neq  0.
\]
Since $\gamma_{k,j}$ depends only on $t_k$ and $t_j$, we have
\begin{align*}
    \sum_{j:\lambda_j=\lambda_1}E(f;j)
    &=\sum_{j:\lambda_j=\lambda_1} \sum_{k\ge1}\langle f,\phi_k\rangle\gamma_{k,j}\langle \phi_{k},\phi_j^2\rangle\\
    &=\sum_{k\ge1} \langle f,\phi_k\rangle \gamma_{k,1} \left\langle \phi_k,\sum_{j:\lambda_j=\lambda_1} \phi_j^2\right\rangle.
\end{align*}
Therefore, if $\sum_{j:\lambda_j=\lambda_1} \phi_j^2$ is not a constant function, then there exists $k_0>0$ with $\left\langle \phi_{k_0},\sum_{j:\lambda_j=\lambda_1} \phi_j^2\right\rangle\neq0$. Then, by taking $f=\phi_{k_0}$, we have $\sum_{j:\lambda_j=\lambda_1}E(\phi_{k_0};j)\neq 0$. Note that if $\lambda_1$ is simple, then $\sum_{j:\lambda_j=\lambda_1} \phi_j^2=\phi_1^2$ is not a constant function. Hence, we obtain Theorem~\ref{theorem:main}-(2).

% {\color{red} Incorrect

% \subsection{
% \texorpdfstring{Genericity of $\sum_{j:\lambda_j=\lambda_1}\phi_j^2\neq {\rm const.}$}
% {Genericity of the simplicity of lambda1}
% }\label{subsec:generic_simple_eval}
% Note that if $\lambda_1$ is simple, then $\sum_{j:\lambda_j=\lambda_1}\phi_j^2$ is non-constant.

% % Consider the thin part $\mathcal{T}^{<\epsilon}(S_g)$ of the Teichm\"uller space $\mathcal{T}(S_g)$, consisting of compact hyperbolic surfaces of genus $g$ whose shortest closed geodesic has length $<\epsilon$. By the main theorem of \cite{SWY_Eval_Lengths}, we have $\lambda_1\ll_g\epsilon$. Hence, for sufficiently small $\epsilon>0$, we have $\Im(t_1)$ very close to $1/2$ so that $e^{-(\frac12+\Im(t_1))T}$ dominates $e^{(-\frac24+\epsilon)T}$ in the error term of  

% Eigenvalue branches depend real analytically on the hyperbolic structure; see \cite[Lemma 3.15]{Berger_LowEval}. Hence the locus
% \[
% V_{1,2}:=\{X\in\mathcal{T}(S_g):\lambda_1(X)=\lambda_2(X)\}
% \]
% is a real analytic subset of $\mathcal{T}(S_g)$. Consequently, either $V_{1,2}=\mathcal{T}(S_g)$ or $V_{1,2}$ has empty interior. By \cite[Theorem 2]{Berger_Small_Eval_Bulletin}, the former alternative does not occur, and thus $V_{1,2}$ is a proper real analytic subset of $\mathcal{T}(S_g)$.

% Therefore, the first nonzero eigenvalue $\lambda_1$ is simple on an open dense subset of the thin part $\mathcal{T}^{<\epsilon}(S_g)$. In particular,
% \[
%     \sum_{j:\lambda_j=\lambda_1}\phi_j^2
% \]
% is non-constant for hyperbolic surfaces lying in this open dense subset.
% }

\subsection{Estimate of the spectral side (\ref{eqn:Spec_Side})}\label{subsec:Spec_Side}

Recall we assume $\tau>20/\delta$ and $M=\lfloor\tau\rfloor$. By Theorem \ref{thm:ZTF},
\[
    h(t) =  \frac{1}{2\pi i}\int_{(\sigma)}\Theta_\tau^t (s) M \psi_{T,\epsilon}(s) ds,
\]
where $\frac12+\Im(t_1)<\sigma<1$ and
\[
    \Theta^t_\tau(s)=\frac{2^{2-2s} \sqrt{\pi}}{\sin(\pi s)}\frac{\Gamma\left(-\frac{1}{2}+it+s\right)\Gamma\left(-\frac{1}{2}-it+s\right)}{\Gamma\left(-\frac{1}{2}\left(\frac{1}{2}+i\tau\right)+s\right)\Gamma\left(-\frac{1}{2}\left(\frac{1}{2}-i\tau\right)+s\right)}\cosh(\pi t).    
\]

By shifting the contour $(\sigma)$ to $(-(M+\nu))$ with $\nu\in(0,\frac12-\Im(t_1))$, we have
\begin{align}
    h_\tau(t_j)
    &:=h(t_j) -\sum_{m=0}^{M} \underbrace{{\rm Res}_{s=-m} \Theta^{t_j}_\tau(s)M\psi_{T,\epsilon}\left(s\right)}_{=:h_m}\\
    &=\underbrace{\sum_{m=0}^{M}\sum_{\pm} {\rm Res}_{s=\frac{1}{2}\pm i t_j -m} \Theta^{t_j}_\tau(s)M\psi_{T,\epsilon}\left(s\right)}_{=:h_{\tau,1}(t_j)}
    + \underbrace{\frac{1}{2 \pi i} \int_{(-(M+\nu))} \Theta_\tau^{t_j} (s) M\psi_{T,\epsilon}(s) ds}_{=:h_{\tau,2}(t_j)}.
\end{align}
When $t_j=0$,
\[
    h_{\tau,1}(0):=\sum_{m=0}^{M}{\rm Res}_{s=\frac{1}{2}-m} \Theta^{0}_\tau(s)M\psi_{T,\epsilon}\left(s\right).
\]
Hence, on the spectral side, we estimate
\[
    \sum_{j\ge1} h_\tau(t_j)\langle\varphi\phi_j,\phi_j\rangle = \sum_{j\ge1} h_{\tau,1}(t_j)\langle\varphi\phi_j,\phi_j\rangle+\sum_{j\ge1} h_{\tau,2}(t_j)\langle\varphi\phi_j,\phi_j\rangle.
\]
Equation \eqref{eqn:Spec_Side} follows from Lemmas~\ref{lem:h_tau1} and~\ref{lem:h_tau2}.

\begin{lemma}\label{lem:h_tau1}
    Assume $\tau>20/\delta$ and $M=\lfloor \tau\rfloor$. For all $e^{T}>(1+|\tau|)^{A(\delta)}$, we have
    \[
        \sum_{j\ge 1} h_{\tau,1}(t_j)\langle\varphi\phi_j,\phi_j\rangle=\sum_{j:\lambda_j<3/16} \gamma_{k,j} \langle\varphi\phi_j,\phi_j\rangle\,e^{(\frac12+\Im(t_j))T} +(1+|\tau|)^3\,O_{X,\delta}(e^{(\frac34+\delta)T}),
    \]
    where $\tau=t_k$ and
    \[
        \gamma_{k,j}= \frac{2^{1-2\Im(t_j)}\sqrt{\pi}\,\Gamma\left(2\Im(t_j)\right)}{\Gamma\left(\frac14+\Im(t_j)+\frac12it_k\right)\Gamma\left(\frac14+\Im(t_j)-\frac12it_k\right)\left(\frac12+\Im(t_j)\right)}>0.
    \]
\end{lemma}
\begin{proof}

For $h_1(t_j)$ defined in Lemma \ref{lem:ZTF_spectral_side}, we have
\[
    \Delta_1(t_j):=h_{\tau,1}(t_j)-h_1(t_j)=
    \begin{cases}
        \sum_{m=1}^{M}\sum_{\pm} {\rm Res}_{s=\frac{1}{2}\pm i t_j -m}M\psi_{T,\epsilon}(s)\,  \Theta^{t_j}_\tau(s) & t_j\neq 0\\
        \sum_{m=1}^{M}
        \operatorname{Res}_{s=\frac12-m}
        \left(M\psi_{T,\epsilon}(s)\Theta_\tau^0(s)\right)
        & t_j=0.
    \end{cases}
\]
By \eqref{eqn:h1_sum} with $\epsilon=e^{-(\frac14-\delta)T}$, it suffices to show
\[
    \sum_{j\ge1} \Delta_1(t_j)\,\langle\varphi\phi_j,\phi_j\rangle=(1+|\tau|)^3\, O_X(e^{\frac34T}).
\]

\noindent {\it Case 1. $t_j\in i\mathbb{R}$.}

Recall $\Im(t_j)<1/2$ for all $t_j\in i\mathbb{R}$. By Lemmas \ref{lem:res Theta at s=1/2+it_j-m, t_j nonzero} and \ref{lem:res Theta at s=1/2+it_j-m, t_j zero}, for all $m>0$ and $t_j\in i\mathbb{R}$, including the case when $t_j=0$, we have
\[
    |{\rm Res}_{s=\frac{1}{2}\pm it_j -m} M\psi_{T,\epsilon}(s)\,\Theta^{t_j}_\tau(s)|
    \ll
    (1+\tau)^{3/2}\left(\frac{e^{2}(1+\tau)^{2}}{e^{\frac14T}}\right)^m e^{\pi\tau/2}\,Te^{(\frac18+\frac14\Im(t_j))T}.
\]

On the other hand, by Theorem \ref{thm:TripleProduct}, for all $t_j\in i\mathbb{R}$, we have
\[
    \langle\varphi,\phi_j^2\rangle\ll (1+|\tau|)(1+\log(1+|\tau|))e^{-\pi\tau/2}.
\]

Therefore, by using $e^{T}>(1+|\tau|)^{A(\delta)}$ and $\Im(t_j)<1/2$, we obtain
\begin{align*}
    \sum_{t_j\in i\mathbb{R}}\Delta_1(t_j)\langle\varphi,\phi_j^2\rangle
    &\ll_{X,\delta} \sum_{t_j\in i\mathbb{R}}(1+|\tau|)^3
    \left(
    \sum_{m=1}^{M} \left(\frac{e^2(1+|\tau|)^2}{e^{\frac14T}}\right)^{m}
    \right) Te^{\frac14T}\\
    &\ll_{X,\delta} (1+|\tau|)^3\,Te^{\frac14T}.
\end{align*}

\medskip

\noindent {\it Case 2. $t_j\in \mathbb{R}$ and $t_j\neq 0$.}

By \eqref{eqn:Mpsi_T}, for all $m>0$, we have
\[
    M\psi_T(\tfrac12\pm it_j-m)=O(e^{\frac14(\frac{1}{2}-m)T}).
\]
Recall from Lemma \ref{lem:res Theta at s=1/2+it_j-m, t_j nonzero} that, for all $m>0$, we have
\[
    |{\rm Res}_{s=\frac{1}{2}\pm it_j -m}\Theta^{t_j}_\tau(s)|\\
    \ll
    \sqrt{1+\tau} \exp(-\tfrac{\pi}{2}(2|t_j|-|t_j+\tfrac\tau2|-|t_j-\tfrac\tau2|))
            \left(\frac{12e^2\max\{(1+t_j),(1+\tau)^2\}}{(1+m)}\right)^m.
\]

When $t_j>T$, the decay $e^{-\frac14mT}$ may not cancel the term $(1+t_j)^m$ in the residue. Hence, we will obtain the decay $(1+t_j)^{-m-3}$ from $M\varrho(\epsilon s)$, instead of using an easy bound $M\varrho(\epsilon s)=O(1)$. The additional exponent $-3$ is required for the convergence of the sum over $j\ge 1$; see Lemma \ref{lem:sum(1+t)^delta}.

\noindent {\it Case 2-1: $t_j< \epsilon^{-\frac{1/4-\delta/2}{1/4-\delta}}=e^{(\frac14-\frac\delta2)T}$.}

We use an easy bound $M\varrho(\epsilon s)=O(1)$ and
\[
    M\psi_{T,\epsilon}((\tfrac{1}{2}-m+it_j))=
    O_{X,\delta}(e^{\frac14(\frac12-m)T}),
\]
which was shown in Lemma \ref{lem:res Theta at s=1/2+it_j-m, t_j nonzero}. Recall $e^{T}>(1+|\tau|)^{A(\delta)}$.
It follows that
\begin{align}
\begin{split}\label{eqn:Delta_real_1}
    |\Delta_1(t_j)|
    &\ll (1+|\tau|)\exp\left(-\tfrac{\pi}{2}(2|t_j|-|t_j+\tfrac{\tau}{2}|-|t_j-\tfrac{\tau}{2}|)\right)
    \left(
    e^{T/8}\sum_{m=1}^{M} \left(\frac{12e^2\max\{(1+\tau)^2,(1+t_j)\}}{e^{\frac14T}}\right)^{m}
    \right)\\
    &\ll_\delta (1+|\tau|)\exp\left(-\frac{\pi}{2}(2|t_j|-|t_j+\frac{\tau}{2}|-|t_j-\frac{\tau}{2}|)\right)
    e^{T/8}.
\end{split}
\end{align}
\noindent {\it Case 2-2: $t_j\ge \epsilon^{-\frac{1/4-\delta/2}{1/4-\delta}}=e^{(\frac14-\frac\delta2)T}$.}

By Lemma \ref{lem:varrho} and using $|\epsilon(\tfrac{1}{2}-m+it_j)|\ge \epsilon t_j$, for all $n\ge 2$, we have
\[
    |M\varrho(\epsilon(\tfrac{1}{2}-m+it_j))|
    \ll \frac{1}{|[\epsilon (\frac{1}{2}-m+it_j)]_n|}e^{|\epsilon(\frac{1}{2}-m)+n|} \left(n(\log (n+1))^2\right)^{n}
    \ll \left(\frac{n^{1.5}}{\epsilon t_j} \right)^n.
\]
Note that $\epsilon t_j\ge t_j^{\frac{\delta/2}{1/4-\delta/2}}\ge t_j^{2\delta}$ for $\delta<1/8$.
Let $n=m\lceil10/\delta\rceil$. Since $m\le 1+\tau$ and $\lceil10/\delta\rceil\le \tau$, we have
\[
    |M\varrho(\epsilon(\tfrac{1}{2}-m+it_j))|
    \ll_\delta \left(\frac{(1+\tau)^{15/\delta+1.5}(1+\lceil10/\delta\rceil)^{15/\delta+1.5}}{(1+t_j)^{20}} \right)^m
    \ll
    \left(\frac{(1+\tau)^{30/\delta+3}}{(1+t_j)^{20}} \right)^m.
\]
Therefore, for all $t_j\ge e^{(\frac14-\frac\delta2)T}$ and $m>0$, we have
\[
    |M\psi_{T,\epsilon}((\tfrac{1}{2}-m+it_j))|
    \ll_{X,\delta}
    e^{\frac{1}{8}T} \left(\frac{(1+\tau)^{30/\delta+3}}{(1+t_j)^{20}e^{\frac14T}} \right)^m.
\]
Since $1+t_j\ge e^{(\frac14-\frac\delta2)T}>(1+\tau)^2$, we have
\[
    |\Delta_1(t_j)| \ll_{X} (1+|\tau|)\exp\left(-\tfrac{\pi}{2}(2|t_j|-|t_j+\tfrac{\tau}{2}|-|t_j-\tfrac{\tau}{2}|)\right)
    \underbrace{
    \left(
        e^{\frac{1}{8}T}\sum_{m=1}^{M}\left(\frac{12e^2(1+\tau)^{30/\delta+3}}{(1+t_j)^{19}e^{\frac14T}}\right)^m  
    \right)
    }_{:=I_j}.
\]
Since $A(\delta)=\lfloor20/\delta\rfloor$, we have
\[
    (1+t_j)^{16}>e^{16(\frac14-\frac\delta2)T}>(1+\tau)^{A(\delta)16(\frac14-\frac\delta2)}>2\left(12e^2(1+\tau)^{30/\delta+3}\right),
\]
which yields
\[
    I_j\ll_\delta (1+t_j)^{-3}.
\]
It follows that
\begin{equation}\label{eqn:Delta_real_2}
    |\Delta_1(t_j)| \ll_{X} (1+|\tau|)(1+t_j)^{-3}\exp\left(-\tfrac{\pi}{2}(2|t_j|-|t_j+\tfrac{\tau}{2}|-|t_j-\tfrac{\tau}{2}|)\right).
\end{equation}
\medskip

Therefore, by using \eqref{eqn:Delta_real_1}, \eqref{eqn:Delta_real_2}, and Theorem \ref{thm:TripleProduct}, we obtain
\begin{align*}
    \sum_{t_j>0} \left|\Delta_{1}(t_j)\langle\varphi\phi_j,\phi_j\rangle\right|
    &=\sum_{t_j<e^{(\frac14-\frac\delta2)T}} \left|\Delta_{1}(t_j) \langle\varphi\phi_j,\phi_j\rangle\right|
    +
    \sum_{t_j\ge e^{(\frac14-\frac\delta2)T}}
    \left|\Delta_{1}(t_j)\langle\varphi\phi_j,\phi_j\rangle\right|\\
    &\ll_\delta(1+\tau)^3\,e^{T/8} \left(\sum_{t_j<e^{(\frac14-\frac\delta2)T}}1\right) + (1+\tau)^3 \sum_{t_j>e^{(\frac14-\frac\delta2)T}}(1+t_j)^{-3}.
\end{align*}
By Lemma \ref{lem:Weyl Law} and Lemma \ref{lem:sum(1+t)^delta}, we have
\[
    \sum_{t_j<e^{(\frac14-\frac\delta2)T}}1=O_{X,\delta}(e^{(\frac12-\delta)T})
    \quad{\rm and}\quad
    \sum_{t_j\ge e^{(\frac14-\frac\delta2)T}}(1+t_j)^{-3}=O_{X,\delta}(e^{-(\frac14-\frac\delta2)T}). 
\]
Therefore,
\[
    \sum_{t_j>0} \left|\Delta_{1}(t_j)\langle\varphi\phi_j,\phi_j\rangle\right|\ll_\delta (1+\tau)^3\,e^{(\frac58-\delta)T}.\qedhere
\]
\end{proof}

\begin{lemma}\label{lem:h_tau2}
    Suppose $\tau>20/\delta$. For all sufficiently small $\delta>0$ and all $e^T>(1+|\tau|)^{A(\delta)}$, we have
    \[
        \left|\sum_{j} h_{\tau,2}(t_j)\langle\varphi\phi_j,\phi_j\rangle\right|= O_{X,\delta}(1).
    \]
\end{lemma}
\begin{proof}
By Lemma \ref{lem:sum(1+t)^delta} and $|\langle\varphi\phi_j,\phi_j\rangle|\le \|\varphi\|_\infty\ll_X(1+|\tau|)^{1/2}$ (\cite{SoggeZelditch}), it suffices to show
\[
    |h_{\tau,2}(t_j)|\ll_{X,\delta}(1+\tau)^{-2}(1+|t_j|)^{-3}.
\]

By Lemma \ref{lem:psi_to_h}, we have
\[
    h_{\tau,2} (t_j)
    = \frac{\cosh\pi t_j}{2 \pi} \int_{\mathbb{R}}
    \left(
    \begin{aligned}
        &\frac{2^{2+2M+2\nu-2it'} \sqrt{\pi}}{\sin\pi(-M-\nu+it')} M\psi_{T,\epsilon}(-M-\nu+it')\\
        &\times \frac{\Gamma\left(-M-\frac{1}{2}-\nu+i(t'+t_j)\right)\Gamma\left(-M-\frac{1}{2}-\nu+i(t'-t_j)\right)}{\Gamma\left(-M-\frac{1}{4}-\nu+i\left(t'+\frac{\tau}{2}\right)\right)\Gamma\left(-M-\frac{1}{4}-\nu+i\left(t'-\frac{1}{2}\tau\right)\right)}
    \end{aligned}
    \right)
    dt',
\]
where $M=\lfloor \tau \rfloor$ and $\nu\in(0,\frac12-\Im(t_1))$.

If $t_j \in i\mathbb{R}$, then $|t_j|\le 1/2$. The approximation for small $t_j \in \mathbb{R}$ also applies to $t_j \in i\mathbb{R}$. Therefore, in what follows, we assume $t_j \in \mathbb{R}$.

\medskip
{\it (Step 1) Estimate of $|M\psi_{T,\epsilon}(-M-\nu+it')|$.}
Similarly to Case 2-2 in the proof of Lemma \ref{lem:h_tau1}, one can obtain
\begin{align}
    \left|M\psi_{T,\epsilon}(-M-\nu+it')\right|
    &\ll_{X,\delta} \left\{
    \begin{aligned}
        &\left(\frac{(1+\tau)^{30/\delta+3}}{(1+|t'|)^{20}} \right)^M e^{-\frac14(M+\nu)T}
        &|t'|\ge e^{(\frac14-\frac\delta2)T}\\
        &e^{-\frac14(M+\nu)T}& |t'| < e^{(\frac14-\frac\delta2)T}
    \end{aligned}
    \right..
\end{align}

\medskip
{\it (Step 2) Estimate of gamma functions.}

By \eqref{eqn:Stirling}, \eqref{eqn:reflection}, and $\arctan(\theta)=\frac\pi2-\arctan(1/\theta)$, we have
\begin{align*}
    &\left|\Gamma\left(-M-\tfrac{1}{2}-\nu+i(t'\pm t_j)\right)\right|\\
    &\asymp \left|\frac{1}{\sin(\pi(-M-\tfrac{1}{2}-\nu+i(t'\pm t_j)))\Gamma(M+\frac32+\nu-i(t'\pm t_j))}\right|\\
    &\ll e^{-\pi|t'\pm t_j|}\left|(M+\tfrac32+\nu)+i(t'\pm t_j)\right|^{-M-1-\nu}\exp\left(M+(t'\pm t_j)\arctan{\frac{t'\pm t_j}{M+\frac32+\nu}}\right)\\
    &\ll e^{-\frac{\pi}{2}|t'\pm t_j|} e^M\left|(M+\tfrac32+\nu)+i(t'\pm t_j)\right|^{-M-1-\nu}.
\end{align*}
By using
\[
    |(a+i(b+c))(a+i(b-c))|\ge |a|\sqrt{a^2+2b^2+2c^2} \quad \forall a,b,c\in \mathbb{R},
\]
we obtain
\begin{align*}
    \left|(M+\tfrac32+\nu)+i(t'+t_j)\right|\left|(M+\tfrac32+\nu)+i(t'- t_j)\right|
    &\ge (M+\tfrac32+\nu)\sqrt{(M+\tfrac32+\nu)^2+2t'^2+2t_j^2}\\
    &\ge (1+M)\max\{1+|t'|,1+t_j\}.
\end{align*}
It follows that
\[
    \left(\left|(M+\tfrac32+\nu)+i(t'+t_j)\right|\left|(M+\tfrac32+\nu)+i(t'- t_j)\right|\right)^{M+1+\nu}\ge (1+M)^{M+1+\nu}(1+t_j)^3(1+|t'|)^{M-2+\nu}.
\]
By using
\[
    \exp\left(\tfrac{\pi}{2}\left(2|t_j|-|t'+ t_j|-|t'-  t_j|\right)\right)
    =\left\{
    \begin{aligned}
        &\exp(\pi(t_j-|t'|))\le 1&& |t'|\ge t_j\\
        &1&& |t'|< t_j,
    \end{aligned}
    \right.
\]
we obtain
\begin{align*}
    &\exp(\pi t_j)\left|\Gamma\left(-M-\tfrac{1}{2}-\nu+i(t'+ t_j)\right)\Gamma\left(-M-\tfrac{1}{2}-\nu+i(t'- t_j)\right)\right|\\
    &\ll \exp\left(\tfrac\pi2(2|t_j|-|t'+t_j|-|t'-t_j|)\right)e^{2M}(1+M)^{-M-1-\nu}(1+t_j)^{-3}(1+|t'|)^{-M+2-\nu}\\
    &\ll e^{2M}(1+M)^{-M-1-\nu}(1+t_j)^{-3}(1+|t'|)^{-M+2-\nu}.
\end{align*}

On the other hand, since $0<\nu<1/2$, we have
\begin{align*}
    &\left|\Gamma\left(-M-\tfrac{1}{4}-\nu+i\left(t'\pm\tfrac{\tau}{2}\right)\right)\right|^{-1}\\
    &\ll |\sin(\pi(-M-\tfrac14-\nu+i(t'\pm\tfrac\tau2)))||\Gamma(\tfrac54+\nu+M-i(t'\pm\tfrac\tau2))|\\
    &\ll \exp(\pi|t'\pm\tfrac\tau2|)\left|M+\tfrac{5}{4}+\nu+i|t'\pm \tfrac{\tau}{2}|\right|^{M+\frac34+\nu} \exp\left(-M-|t'\pm \tfrac\tau2| \arctan\left(\frac{|t'\pm\frac\tau2|}{M+\frac54+\nu}\right)\right)\\
    &\ll\exp(\tfrac\pi2|t'\pm\tfrac\tau2|)\left|M+\tfrac{5}{4}+\nu+i|t'\pm \tfrac{\tau}{2}|\right|^{M+\frac34+\nu} \exp\left(-M+|t'\pm \tfrac\tau2| \arctan\left(\frac{M+\frac54+\nu}{|t'\pm\frac\tau2|}\right)\right)\\
    &\ll \exp(\tfrac\pi2|t'\pm\tfrac\tau2|)\left|M+2+i|t'\pm \tfrac{\tau}{2}|\right|^{M+\frac34+\nu}.
\end{align*}
Using $2ab\le a^2+b^2$ and $(a+b)^2\le 2a^2+2b^2$ for all $a,b>0$, we obtain
\begin{align*}
    &\left|\left(M+2+i|t'+ \tfrac{\tau}{2}|\right)\left(M+2+i|t'- \tfrac{\tau}{2}|\right)\right|^2\\
    =&\left(M+2\right)^4+2\left(M+2\right)^2\left(t'^2+\frac{\tau^2}{4}\right)+\left(t'^2-\frac{\tau^2}{4}\right)^2\\
    \le&2\left(M+2\right)^4+2\left(t'^2+\frac{\tau^2}{4}\right)^2\\
    \le&32\left(\frac{M}2+1\right)^4+2\left(2t'^4+2\left(\frac{\tau^2}{4}\right)^2\right)\\
    \le& 4(1+|t'|)^4 + 36 (1+\tfrac\tau2)^4\\
    \le& 40\left(\max\left\{1+|t'|,1+\tfrac\tau2\right\}\right)^4,
\end{align*}
where we used $M\le \tau$.

Hence, since $M=\lfloor\tau\rfloor$ and $\tau>20/\delta$, we obtain
\begin{align*}
    &\left|\Gamma\left(-M-\tfrac{1}{4}-\nu +i\left(t'+\tfrac{\tau}{2}\right)\right)\Gamma\left(-M-\tfrac{1}{4}-\nu +i\left(t'-\tfrac{\tau}{2}\right)\right)\right|^{-1}\exp\left(-\pi|t'|\right)\\
    &\ll \exp\left(\tfrac{\pi}{2}\left(|t'+ \tfrac{\tau}{2}|+|t'- \tfrac{\tau}{2}|-2|t'|\right)\right) 40^{M/2}\left(\max\left\{1+|t'|,1+\tfrac\tau2\right\}\right)^{2M+\frac32+2\nu}\\
    &\ll
    \left\{
    \begin{aligned}
        &\exp\left(\tfrac\pi2 \tau \right) 40^{M/2} \left(1+\tfrac\tau2\right)^{2M+\frac32+2\nu} && |t'|\le \tfrac\tau2\\
        &40^{M/2}\left(1+|t'|\right)^{2M+\frac32+2\nu} && |t'|>\tfrac\tau2.
    \end{aligned} 
    \right.
\end{align*}

\medskip
{\it (Step 3) Show $|h_{\tau,2}(t_j)|\ll(1+\tau)^{-2}(1+t_j)^{-3}$}

By using $\cosh(\pi t_j)\asymp e^{\pi|t_j|}$ and $|\sin\pi(\sigma+i t')|\asymp e^{\pi |t'|}$, we have
\[
    |h_{\tau,2} (t_j)|\ll \int_{\mathbb{R}}
    \left(
    \begin{aligned}
        &e^{\pi(|t_j|-|t'|)}(4^M)|M\psi_{T,\epsilon}(-M-\nu+it')|\\
        &\times \left|\frac{\Gamma\left(-M-\frac{1}{2}-\nu+i(t'+t_j)\right)\Gamma\left(-M-\frac{1}{2}-\nu+i(t'-t_j)\right)}{\Gamma\left(-M-\frac{1}{4}-\nu+i\left(t'+\frac{\tau}{2}\right)\right)\Gamma\left(-M-\frac{1}{4}-\nu+i\left(t'-\frac{1}{2}\tau\right)\right)} \right|
    \end{aligned}
    \right) dt'.
\]

For simplicity, let us denote the integrand by $(\dots)$. By (Steps 1 and 2), the estimates of the terms in the integrand depend on whether $|t'|<\tau/2$ and $|t'|<e^{(\frac14-\frac\delta2)T}$. 

We first estimate $\int_{|t'|<e^{(\frac14-\frac\delta2)T}}(...) dt'$. By $2^2e^240^{1/2}\le 200$, we obtain
\begin{align*}
    &\int_{|t'|<e^{(\frac14-\frac\delta2)T}}(...) dt'\\
    &\ll \frac{(200/(1+M))^M}{(1+t_j)^{3} e^{\frac14(M+\nu)T}}
    \left(\begin{aligned}
    &e^{\pi\tau/2}(1+\tfrac\tau2)^{2M+\frac32+2\nu}\int_{|t'|<\tau/2} (1+|t'|)^{-M+2-\nu} dt'\\
    &\quad +
    \int_{\tau/2<|t'|<e^{(\frac14-\frac\delta2)T}} \left(1+|t'|\right)^{M+\frac72+\nu} dt'\end{aligned}\right).
\end{align*}

Since $M=\lfloor\tau\rfloor$ and $\tau>20/\delta$, we have $\int_{|t'|<\tau/2} (1+|t'|)^{-M+2-\nu} dt' = O(1)$. By using $A(\delta)\sim 20/\delta$ and $e^{T}>(1+|\tau|)^{A(\delta)}$, we have
\begin{align*}
    \frac{e^{\pi\tau/2} (200/(1+M))^M(1+\frac\tau2)^{2M+\frac32+2\nu}}{(1+t_j)^{3} e^{\frac14(M+\nu)T}}
    &\le (1+t_j)^{-3} e^{-\nu T/4}\left( \frac{e^{\pi/2} 200 (1+\tfrac\tau2)^{2+\frac3M}}{(1+M)(1+|\tau|)^{A(\delta)/4}}\right)^M\\
    &\le (1+t_j)^{-3}(1+\tau)^{-2}
\end{align*}
for all sufficiently small $\delta>0$.

On the other hand,
\[
    \int_{|t'|<e^{(\frac14-\frac\delta2)T}} \left(1+|t'|\right)^{M+\frac72+\nu} dt' \ll  e^{(\frac14- \frac\delta2)(M+4)T} \int_{|t'|<e^{(\frac14-\frac\delta2)T}} 1dt'\ll e^{(\frac14- \frac\delta2)(M+5)T}.
\]
Since $M=\lfloor\tau\rfloor$ and $\tau> 20/\delta$, we have
\[
    \frac{200^Me^{(\frac14- \frac\delta2)(M+5)T}}{(1+M)^M(1+t_j)^{3} e^{\frac14(M+\nu)T}}
    < (1+t_j)^{-3}(1+\tau)^{-2} 200^M e^{(5-\nu-\tfrac\delta2(M+5))T}.
\]
Then, by $A(\delta)\sim 20/\delta$
\[
    200^Me^{(5-\nu-\delta(M+5)/2)T}<e^{(5-\nu-\delta (M+5)/4)T}\left(\frac{200}{e^{\delta T/4}}\right)^M< \frac{200}{(1+|\tau|)^{A(\delta)\delta/4}}<1
\]
for sufficiently small $\delta>0$.

Therefore, we have
\[
    \int_{|t'|<e^{(\frac14-\frac\delta2)T}}(...) dt'\ll_\delta (1+\tau)^{-2}(1+t_j)^{-3}.
\]

Next, we estimate $\int_{|t'|>e^{(\frac14-\frac\delta2)T}}(...) dt'$. Since $\nu\in(0,1/2)$, we have
\begin{align*}
    &\int_{|t'|>e^{(\frac14-\frac\delta2)T}}(...) dt'\\
    &\ll \frac{200^M(1+\tau)^{30M/\delta+3M}}{(1+M)^Me^{\frac14(M+\nu)T}}  (1+t_j)^{-3}\int_{|t'|>e^{(\frac14-\frac\delta2)T}} (1+|t'|)^{(-20M)+(-M+2-\nu)+(2M+\frac32+2\nu)} dt'\\
    &\ll (1+\tau)^{30M/\delta+3M}(1+t_j)^{-3} \int_{|t'|>e^{(\frac14-\frac\delta2)T}} |t'|^{-19M+4} dt'\\
    &\ll (1+\tau)^{30M/\delta+3M}(1+t_j)^{-3}\, e^{(-19M+5)(\frac14-\frac\delta2)T}.
\end{align*}
Note that
\[
    e^{(-19M+5)(\frac14-\frac\delta2)T}\le (1+\tau)^{A(\delta)(\frac14-\frac\delta2)(5-19M)}.
\]
Since $A(\delta)\sim 20/\delta$, it follows that
\[
    A(\delta)\left(\frac14-\frac\delta2\right)(19M-5)> 5+30M/\delta+3M
\]
for sufficiently small $\delta>0$. Hence, we obtain
\[
    \int_{|t'|>e^{(\frac14-\frac\delta2)T}}(...) dt'\ll_{\delta}(1+\tau)^{-2}(1+t_j)^{-3}.
\]
\end{proof}

\newpage
\appendix

\section{
\texorpdfstring{Lower bound of the least zeros of cylinder functions with $\alpha=1/4$}
{Lower bound of the least zeros of cylinder functions with alpha=1/4}
}\label{sec:Cylinder Ftn Zero}

For $\nu\ge0$, a {\it cylinder function} $\mathscr{C}_{\nu,\alpha}(x)$ is defined by
\[
    \mathscr{C}_{\nu,\alpha}(x)=J_\nu(x)\cos(\pi\alpha)+Y_\nu(x)\sin(\pi\alpha).
\]
We define $X_\nu>0$ as the least zero of $J_\nu(x)+Y_\nu(x)=\sqrt{2}\;\mathscr{C}_{\nu,1/4}(x)$. There exists an asymptotic expansion
\[
    X_\nu=\nu-\frac{c}{2^{1/3}}\nu^{1/3}+\frac{3}{10}\left(\frac{c}{2^{1/3}}\right)^2\nu^{-1/3}+O(\nu^{-1}),
\]
where $c\approx-0.36604655\dots$ is the negative solution of ${\rm Ai}(x)={\rm Bi}(x)$ with the least absolute value. We define $\hat{\nu}$ as the non-negative real number satisfying
\[
    \nu:=\hat{\nu}-\frac{c}{2^{1/3}}\hat{\nu}^{1/3}.
\]
Then
\[
    X_{\hat\nu}=\nu+\frac{3}{10}\left(\frac{c}{2^{1/3}}\right)^2\nu^{-1/3}+O(\nu^{-1}),
\]
and thus we obtain $X_{\hat\nu}>\nu$ for all sufficiently large $\nu>0$. In this section, we prove $X_{\hat\nu}>\nu$ for all $\nu\ge0$. Since $\nu>\hat\nu$, this also implies $X_\nu>\nu$ for all $\nu>0$.
\begin{lemma}\label{lem:lower bound of zero of cylinder ftn}
    For all $\nu\ge 0$, we have $X_{\hat{\nu}}>\nu$.
\end{lemma}
\begin{proof}
    Suppose that $X_{\hat\nu}\le \nu$ for some $\nu>0$.   Consider $X_{\hat\nu}$ and $\nu$ as functions of $\nu\ge0$. Since $X_0>0$ and $X_{\hat\nu}>\nu$ for sufficiently large $\nu>0$, there exist $\nu$'s satisfying $X_{\hat\nu}=\nu$. Let $\nu_0$ denote the largest number with $X_{\hat{\nu}_0}=\nu_0$. Then, we have
    \begin{equation}\label{eqn:Cylinder Zero_inequality_1}
        \frac{d X_{\hat\nu}}{d\hat\nu}(\hat\nu_0)\ge\frac{d\nu}{d\hat\nu}(\hat{\nu}_0)=1-\frac{c}{3\cdot 2^{1/3}}\hat{\nu}_0^{-2/3}.
    \end{equation}

    We use the identity
    \[
        \frac{\vartheta}{\sin \vartheta}=\int_0^\infty K_0(u)e^{-u\cos\vartheta}du,
    \]
    and the Watson formula \cite[(3) on p.\@ 508]{Watson_BesselFunctions}
    \[
        \frac{dX_\nu}{d\nu}=2X_\nu\int_0^\infty K_0(2X_\nu \sinh t)e^{-2\nu t}dt,
    \]
    where $K_0$ denotes the modified Bessel function of the second kind. Since $K_0(x)$ is monotonically decreasing, we have
    \begin{align}
        \begin{split}\label{eqn:Cylinder Zero_inequality_2}
            \frac{d X_{\hat\nu}}{d\hat\nu}(\hat{\nu}_0)
            &\le 2 X_{\hat{\nu}_0}\int_0^\infty
            K_0(2X_{\hat{\nu}_0} t)e^{-2\hat{\nu}_0 t}\,dt\\
            &\le \int_0^\infty
            K_0(u)e^{-u\frac{\hat{\nu}_0}{X_{\hat{\nu}_0}}}\,du\\
            &=\frac{\vartheta}{\sin \vartheta},    
        \end{split}
    \end{align}
    where
    \[
        \cos\vartheta=\frac{\hat\nu_0}{X_{\hat\nu_0}}=\frac{\hat{\nu}_0}{\nu_0}=\frac{1}{1-\frac{c}{2^{1/3}}\hat{\nu}_0^{-2/3}}<1.
    \]
    It is easy to show that
    \[
        \frac{\phi}{\sin\phi}<1+\frac{C}{3},
    \]
    where $\cos\phi=\frac{1}{1+C}$ with $C>0$. By substituting $\phi$ and $C$ with $\vartheta$ and $-\frac{c}{2^{1/3}}\hat{\nu}_0^{-2/3}>0$, we obtain a contradiction from \eqref{eqn:Cylinder Zero_inequality_1} and \eqref{eqn:Cylinder Zero_inequality_2}.

\end{proof}
\section{Estimates of conical and Bessel functions}\label{AppendixConical}
Throughout the following four subsections, we estimate conical functions by Bessel functions and then Bessel functions by exponential functions. Our approach relies on work of Dunster, as well as his collaborations with Boyd, Gil, and Segura, based on the Liouville transformation.

The most technical step in controlling the error terms is to obtain upper bounds for the total variations of certain functions along {\it progressive paths}. We provide a detailed and rigorous proof in Section \ref{subsec:Conical by J-Bessel} and omit analogous details in the subsequent subsections.

\subsection{
\texorpdfstring{Estimate of Conical functions by $J$-Bessel functions, $\alpha$ bounded}
{Estimate of Conical functions by J-Bessel functions, alpha bounded}
}\label{subsec:Conical by J-Bessel}

\begin{lemma}\label{lem:Conical by JBessel}
Fix $A>0$. For all $\alpha=m/\tau\in[0,A]$ and $z=\cosh(x+i(\pi/2-1/T))$ with $x\in \mathbb{R}$ and $T\gg_A 1$, we have
\[
    |P^{-m}_{-\frac{1}{2}+i\tau}(z)|= \left|\frac{\Gamma(1/2+i\tau)}{\Gamma(1/2+m+i\tau)} \left(\frac{\zeta-\alpha^2}{z^2-1-\alpha^2}\right)^{1/4}J_{m}(\tau\zeta^{1/2})(1+O_A(\tau^{-1}))\right|.
\]
\end{lemma}

\subsection*{Liouville transformation}
The function $w:=(z^2-1)^{1/2}P^{-m}_{-\frac{1}{2}+i\tau}(z)$ solves the equation
\begin{equation}\label{eqn:Conical_w}
    \frac{d^2w}{dz^2}=\left(\tau^2f(\alpha,z)+g(z)\right)w,
\end{equation}
where $\alpha=m/\tau$, $f(\alpha,z)=\frac{(1-z^2)+\alpha^2}{(z^2-1)^2}$, and $g(z)=-\frac{z^2+3}{4(z^2-1)^2}$.

We use the following transformation between $z$ and $\zeta$:
\begin{equation}\label{eqn:LG-transform_Conical to J}
    \int_{\alpha^2}^\zeta \frac{(\xi-\alpha^2)^{1/2}}{2\xi} d\xi =\int_{\sqrt{1+\alpha^2}}^z\frac{(t^2-1-\alpha^2)^{1/2}}{t^2-1} dt,
\end{equation}
which is in \cite[Section 6]{Dunster_Prep}; see also the explanation around \cite[Equation (3.1)]{Dunster_Concial}.

The function $\zeta(\alpha,z)$ is continuous in $\alpha$ and analytic in $z$. The joint continuity for $\alpha\in[0,A]$ and $z\in \{z:\Re(z)>0,z\neq  \sqrt{1+\alpha^2},1\}$ follows from the implicit function theorem and \eqref{eqn:LG-transform_Conical to J}. It also extends jointly continuously to the imaginary axis $\Re(z)=0$ and $z\in  \{\sqrt{1+\alpha^2},1\}$.

From the asymptotics of the integrands in \eqref{eqn:LG-transform_Conical to J} near $\zeta=0$ and $z=1$, one can obtain
\[
    \zeta(\alpha,z)=C(\alpha)(z-1)+O_A(|z-1|^2)
\]
for all $|z-1|\le r_A$, where
\[
    C(\alpha)=2(1+\alpha^2)\exp\left(\frac{2\arctan\alpha}{\alpha}-2\right)
\]
with $C(0)=2$. This yields the joint continuity of $\zeta(\alpha,z)$ at $z=1$. We have $C(\alpha)=2+O_A(\alpha^2)$. By the joint continuity of $\zeta(\alpha,z)$ and compactness, we can extend the range of validity of the estimate as follows: for all $z$ with $\Re z\ge0$ and $|z|\le R$
\begin{equation}\label{eqn:zeta by z near zeta=0}
    \zeta(\alpha,z)=2(z-1)+O_A(\alpha^2|z-1|+|z-1|^2).
\end{equation}

 A similar argument applies around $z=\sqrt{1+\alpha^2}$ and $\zeta(\alpha,\sqrt{1+\alpha^2}):=\alpha^2$ and yields the joint continuity of $\zeta(\alpha,z)$ at $z=\sqrt{1+\alpha^2}$.

Since
\[
    \frac{(\xi-\alpha^2)^{1/2}}{2\xi}= \frac{1}{2\sqrt{\xi}}+O_A(|\xi|^{-3/2}), 
    \qquad
    \frac{(t^2-1-\alpha^2)^{1/2}}{t^2-1}= \frac{1}{\sqrt{t^2-1}}+O_A(|t^2-1|^{-3/2}),
\]
we have
\begin{equation}\label{eqn:zeta and z large asymp}
    \left(\zeta(\alpha,z)\right)^{1/2}={\rm arcosh}(z)+O_A(1)=\log(z)+O_A(1),
\end{equation}
for $|z|,|\zeta|\gg1$.

% \begin{lemma}\label{lem:dZeta/dZ}
%     We have
%     \[
%         \frac{d\zeta}{dz}\Big|_{z=\sqrt{1+\alpha^2}}=2(1+\alpha^2)^{1/6}.
%     \]
% \end{lemma}
% \begin{proof}
% By differentiating \eqref{eqn:LG-transform_Conical to J} about $z$, we obtain
% \begin{equation}\label{eqn:dzeta/dz}
%     \frac{\sqrt{\zeta-\alpha^2}}{2\zeta}\frac{d\zeta}{dz}=\frac{\sqrt{z^2-1-\alpha^2}}{z^2-1}.
% \end{equation}
% Since $\sqrt{1+\alpha^2}$ in the $z$-plane corresponds to $\alpha^2$ in the $\zeta$-plane, we have
% \[
%     \frac{d\zeta}{dz}\Big|_{z=\sqrt{1+\alpha^2}}=2\lim_{z\to \sqrt{1+\alpha^2}} \frac{\sqrt{z^2-1-\alpha^2}}{\sqrt{\zeta-\alpha^2}}=\frac{2}{d\zeta/dz|_{z=\sqrt{1+\alpha^2}}\;}\lim_{z\to \sqrt{1+\alpha^2}} \frac{\sqrt{\zeta-\alpha^2}}{\sqrt{z^2-1-\alpha^2}},
% \]
% where the second equality follows from the L'hospital rule. Then, we obtain
% \[
%     \frac{d\zeta}{dz}\Big|_{z=\sqrt{1+\alpha^2}}=2(1+\alpha^2)^{1/6}.\qedhere
% \]
% \end{proof}

For $W=\left(\frac{d\zeta}{dz}\right)^{1/2}w$, the equation \eqref{eqn:Conical_w} is transformed into 
\begin{equation}\label{eqn:Conical_W}
    \frac{d^2W}{d\zeta^2}=\left\{\tau^2\left(\frac{\alpha^2}{4\zeta^2}-\frac{1}{4\zeta}\right)+\frac{\psi(\alpha,\zeta)}{\zeta}-\frac{1}{4\zeta^2}\right\}W,
\end{equation}
where
\begin{align}\label{eqn:psi_conical to J-Bessel}
\begin{split}
    \psi(\alpha,\zeta)&=\frac{1}{16}\frac{(\zeta+4\alpha^2)}{(\zeta-\alpha^2)^2}+\frac{1}{64}\frac{(\zeta-\alpha^2)(5\dot{f}^2-4f\ddot{f}-16gf^2)}{\zeta f^3}\\
    &=\frac{\zeta+4\alpha^2}{16(\zeta-\alpha^2)^2}-\frac{(z^2-1)(\zeta-\alpha^2)(\alpha^4-1+z^2(1+4\alpha^2))}{16\zeta(z^2-1-\alpha^2)^3}.
\end{split}
\end{align}
By \cite[Equation (3.1)]{Dunster_Concial} with $n=0$ and \cite[Theorem 3]{BoydDunster_EstimateLegendra}, we have
\[
    P^{-m}_{-\frac{1}{2}+i\tau}(z)=c_{1,1}\left(\frac{\zeta-\alpha^2}{z^2-1-\alpha^2}\right)^{1/4} \left[J_{m}(\tau\zeta^{1/2})+\zeta^{-1/2}\epsilon^{(0)}_1(\tau,\alpha,\zeta)\right],
\]
where, by \cite[Equation (2.8)]{Dunster_Concial},
\[
    c_{1,1}=(\tau^2+m^2)^{-m/2}\exp\left(m-\tau \arctan\left(\frac{m}{\tau}\right)\right).
\]
\begin{lemma}
Fix $A>0$. For all $m,\tau>0$ with  $m/\tau\in[0,A]$ and $\tau \gg_A 1$, we have
\[
    \left|\frac{\Gamma(1/2+i\tau)}{\Gamma(1/2+m+i\tau)}\right|=
    {(\tau^2+m^2)^{-m/2}\exp\left(m-\tau\arctan\left(\frac{m}{\tau}\right)\right)}(1+O_A(\tau^{-1})).
\]
\end{lemma}
\begin{proof}
This follows from Lemma \ref{lem:Gamma(0.5+m+itau)} and the well-known formula
\[
    \log|\Gamma(1/2+i\tau)|=-\frac{\pi}{2}\tau+\frac{1}{2}\log(2\pi)+O(\tau^{-2}).\qedhere
\]
\end{proof}

\begin{lemma}\label{lem:Gamma(0.5+m+itau)}
Fix $A,B>0$. For $m,\tau\ge0$, we have
\[
    |\Gamma(\tfrac12+m+i\tau)|=\sqrt{2\pi}(m^2+\tau^2)^{m/2}e^{-m-\tau\arctan(\tau/m)}\times
    \left\{
        \begin{aligned}
            &(1+O_A(\tau^{-1}))&&{\rm if}\; m/\tau\in[0,A],\, \tau\gg_A1\\
            &(1+O_B(m^{-1}))&&{\rm if}\;\tau/m\in[0,B],\, m\gg_B1
        \end{aligned}
    \right. .
\]
\end{lemma}
\begin{proof}
    We give a proof for $m/\tau\in[0,A]$; the case $\tau/m\in[0,B]$ follows by a similar argument.
    
    By Stirling's approximation \eqref{eqn:Stirling}, we have
    \[
        \log|\Gamma(1/2+m+i\tau)|=m \log\abs{\frac{1}{2}+m+i\tau}-\tau\,\arg(\frac{1}{2}+m+i\tau)-m-\frac{1}{2}+\frac{1}{2}\log(2\pi)+O(\tau^{-1}).
    \]
    Note that we have
    \begin{align*}
        \log\abs{\frac{1}{2}+m+i\tau}&=\frac{1}{2}\log|\tau^2+(m+1/2)^2|\\
        &=\frac{1}{2}\log(m^2+\tau^2)+\frac12\log\left(1+\frac{m+1/4}{m^2+\tau^2}\right)\\
        &=\frac12\log(m^2+\tau^2)+\frac12\frac{m}{m^2+\tau^2}+O_A(\tau^{-2}),
    \end{align*}
    and, by using the Taylor series of $\arctan(x)$ near $x=m/\tau$, we obtain
    \begin{align*}
        -\arg(\frac{1}{2}+m+i\tau)
        &=-\frac\pi2+\arctan\left(\frac{1/2+m}{\tau}\right)\\
        &=-\frac\pi2+\arctan\frac{m}{\tau}+\frac{1}{2}\frac{\tau}{m^2+\tau^2}+O_A(\tau^{-2}).
    \end{align*}
    Hence, by the cancellation $\tfrac{m^2}{m^2+\tau^2}+\tfrac{\tau^2}{m^2+\tau^2}-1=0$, we obtain
    \[
        \log|\Gamma(1/2+m+i\tau)|=\frac{m}{2}\log(m^2+\tau^2)-\tau\left(\frac{\pi}{2}-\arctan\frac{m}\tau\right)-m+\frac{1}{2}\log(2\pi)+O_A(\tau^{-1}).\qedhere
    \]
\end{proof}

Moreover, by \cite[Theorem 3]{BoydDunster_EstimateLegendra},
\begin{multline*}
    |\epsilon_1^{(0)}(\tau,\alpha,\zeta)|\le \frac{|\zeta|^{1/2}}{\tau}\left|E_{m}^{(0)}(\tau\zeta^{1/2})^{-1}M_m^{(\pm1)}(\tau\zeta^{1/2})\right|\\
    \times \mathscr{V}_{\mathscr{P}^{(0)}}\left((\zeta-\alpha^2)^{1/2}B_0(\alpha,\zeta)\right)\exp\left(\frac{\kappa}{\tau}\mathscr{V}_{\mathscr{P}^{(0)}}\left((\zeta-\alpha^2)^{1/2}B_0(\alpha,\zeta)\right)\right),
\end{multline*}
where $\kappa$ is a constant and the sign $\pm$ is determined as $+$ if $\zeta\in S_\alpha^{(0)}\cup S_\alpha^{(2)}$ and $-$ if $\zeta\in S_\alpha^{(0)}\cup S_\alpha^{(1)}$. The notation $\mathscr{V}_{\mathscr{P}^{(0)}}(f)$ denotes the total variation of $f$ along a progressive path and $\mathscr{P}^{(0)}$ denotes the set of progressive paths. 

\subsection*{Progressive paths} An {\it $R_2$ arc} is a $C^2$-embedding of an interval. A progressive path is a finite concatenation of $R_2$ arcs along which a certain function is monotonic. In the context of Lemma \ref{lem:Conical by JBessel}, the imaginary part of the function
\begin{equation}\label{eqn:Phi_m(0)}
    \Phi^{(0)}_m (\tau\zeta^{1/2})=\int_{X_{\hat{m}}}^{\tau\zeta^{1/2}} \frac{(t^2-(X_{\hat{m}})^2)^{1/2}}{t}dt
\end{equation}
is non-increasing along progressive paths, where
\[
    X_{\hat{m}}=m+O(m^{-1/3})\ge m
\]
as $m\to \infty$. The big-O is independent of $\tau$. More precisely, $X_m$ is the least positive zero of $J_m(x)+Y_m(x)$. See Section \ref{sec:Cylinder Ftn Zero}. The inequality $X_{\hat{m}}\ge m$ follows from Lemma \ref{lem:lower bound of zero of cylinder ftn}.

\begin{figure}
    \centering
       
        \includegraphics[width=0.3\textwidth]{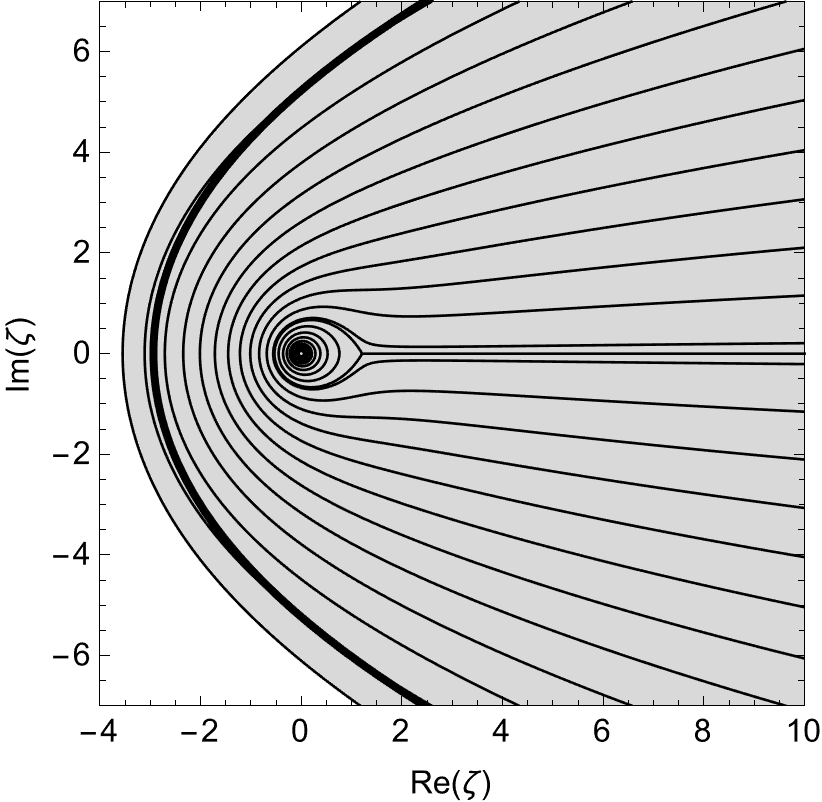}
        \caption{Level sets of $\Im(\Phi_m^{(0)})$ with the locus $\Im(r)=\pi/2-0.1$ with $\alpha=1$ in the $\zeta$-plane. The shaded region corresponds to the right half-plane $\Re(z)>0$ in the $z$-plane. We have $\Im(\Phi_m^{(0)})>0$ in the teardrop-shaped region $\mathscr{T}^{(0)}$, $\Im(\Phi_m^{(0)})=0$ on $\partial\mathscr{T}^{(0)}\cup[\hat{\alpha}^2,\infty)$, and $\Im(\Phi_m^{(0)})<0$ elsewhere.}
        \label{fig:zeta plane}    
\end{figure}

We consider the following three coordinates:
\begin{center}
    \begin{tikzcd}
    \{u:\Im(u)\in(-\frac{\pi}{2},\frac{\pi}{2})\} \arrow[rr,"z=\cosh(u)"]
    && \{z:\Re(z)>0\} \arrow[rr,"{\rm Eq.}\eqref{eqn:LG-transform_Conical to J}"]
    && \{\zeta:{\rm Figure}~\ref{fig:zeta plane}\}.
    \end{tikzcd}
\end{center}
The initial and terminal points of our progressive path are $0$ and $x+i(\pi/2-1/T)$ in the $u$-plane. If $x>0$, we define a curve $\Gamma= \Gamma_1 \cup \Gamma_2$ in the $u$-plane parametrized by
\[
\begin{aligned}
\Gamma_1(s) &= is, \qquad && s \in \left[0,\pi/2-1/T\right], \\
\Gamma_2(s) &= s + \left(\pi/2-1/T\right)i, \qquad && s \in [0,x],
\end{aligned}
\]
and consider its images in the $z$- and $\zeta$-planes. If $x<0$, then we define $\Gamma_2$ by
\[
\begin{aligned}
\Gamma_2(s) &= -s + \left(\pi/2-1/T\right)i, \qquad && s \in [0,-x].
\end{aligned}
\]
\begin{lemma}
    The curve $\Gamma$ is a progressive path. 
\end{lemma}
\begin{proof}
    Since $\Gamma_1$ and $\Gamma_2$ are obviously $R_2$-arcs, let us show $H:=\Im(\Phi_m^{(0)})$ is decreasing along $\Gamma$, i.e., $dH/ds<0$. It is easy to show $dH/ds<0$ along $\Gamma_1$, so we only discuss $\Gamma_2$. 

    Let $\hat{\alpha}=\frac{X_{\hat{m}}}{\tau}$. By Lemma \ref{lem:lower bound of zero of cylinder ftn}, we have $\hat{\alpha}>\alpha$ for all $\alpha>0$.

    By \eqref{eqn:Phi_m(0)}, we have
    \begin{align*}
        dH&=\Im \left( \tau\frac{\sqrt{\zeta-\hat{\alpha}^2}}{2\zeta}\;d\zeta\right)\\
        &=\tau\; \Im\left( \frac{\sqrt{\zeta-\hat{\alpha}^2}}{\sqrt{\zeta-\alpha^2}}\frac{\sqrt{z^2-1-\alpha^2}}{z^2-1}\;dz\right)\\
        &=\tau\; \Im\left( \frac{\sqrt{\zeta-\hat{\alpha}^2}}{\sqrt{\zeta-\alpha^2}}\frac{\sqrt{\sinh^2 u-\alpha^2}}{\sinh u}\; du\right).
    \end{align*}
    According to \cite[p.\@ 436]{BoydDunster_EstimateLegendra}, the branch of $(\zeta-\hat{\alpha}^2)^{1/2}$ is taken so that $\Im(\Phi_m^{(0)}(\tau\zeta^{1/2}))>0$ for $\zeta\in\mathscr{S}^{(0)}$. This implies that $\sqrt{-1}=-i$, and $\arg\sqrt{\zeta-\alpha^2},\arg\sqrt{\zeta-\hat{\alpha}^2},$ and $\arg\sqrt{\sinh^2u-\alpha^2}$ are contained in $[-\pi,0]$.
    
    We first consider the case $\Re(u) > 0$. In this case, we have $\Im(\zeta), \Im(z)>0$. Since $\hat{\alpha} > \alpha > 0$, it follows that
    \[
        \arg\!\left(\frac{\sqrt{\zeta - \hat{\alpha}^2}}{\sqrt{\zeta - \alpha^2}}\right) \in (0, \pi/2).
    \]
    For sufficiently large $T > 0$, we also have $\pi/4 < \arg(\sinh u) < 3\pi/4$.
    It is then straightforward to see that
    \[
        \arg\!\left(\frac{\sqrt{\sinh^2 u - \alpha^2}}{\sinh u}\right) \in (-\pi, -\pi/2).
    \]
    Since $u(s) = s + i(\pi/2 - 1/T)$, we have $du = ds$, and hence $dH/ds < 0$.

    When $\Im(\zeta), \Im(z), \Re(u)< 0$, we let $u(s) = -s + i(\pi/2 - 1/T)$. Then, compared with the previous case, the quantity
    \[
        \frac{\sqrt{\zeta-\hat{\alpha}^2}}{\sqrt{\zeta-\alpha^2}}\frac{\sqrt{\sinh^2 u-\alpha^2}}{\sinh u}
    \]
    is replaced by its complex conjugate, and $du$ changes sign. Hence, we still have $dH/ds<0$.
\end{proof}

\begin{lemma}\label{lem:prog.path_unif_away_turningPt}
    Fix a large $T_0$. For all $T\ge T_0$, there exists $\epsilon\gg_{T_0}1$ such that for all $\alpha\in[0,A]$, every point of $\Gamma_2$ in the $z$-plane is $\epsilon$-away from $\sqrt{1+\alpha^2}$, and every point of $\Gamma_2$ in the $\zeta$-plane is $\epsilon$-away from $\alpha^2$. In particular, the constant $\epsilon=\epsilon(A)$ depends only on $A$.
\end{lemma}
\begin{proof}
    Since $\Gamma_2(s)=\cosh(s+(\pi/2-1/T)i)$ in the $z$-plane is independent of $\alpha$ and close to the $y$-axis, the statement for the $z$-plane is easy to prove.

    Consider the $\zeta$-plane. By \eqref{eqn:zeta and z large asymp} and $\alpha\in[0,A]$, the $\zeta(z)$ is uniformly $\epsilon$-away from $\alpha^2$ for all $|s|> R=R(T_0)$. For $|s|\le R$, by using the compactness of $[-R,R]$, $[0,T_0^{-1}]$ and $[0,A]$, the joint continuity of $\zeta(z)$, and the fact that $\Gamma_2$ does not cross the turning point $\alpha^2$, we obtain the uniform distance $\epsilon\gg_{T_0}1$ from $\alpha^2$.
    % Note that the Liouville transformation \eqref{eqn:LG-transform_Conical to J} between $z$ and $\zeta$ depends continuously on $\alpha$. Since $\Gamma_2$ gets farther from $\alpha^2$ as $T$ increases, we may fix $T=T_0$. Since $|\Gamma_2(s)|\to \infty$ as $|s|\to \infty$, for each $\alpha\in[0,A]$, there exists $B\gg_{T_0} 1$ such that $\Gamma_2(s)$ is $\epsilon_1$-away from $\alpha^2$ for all $s$ with $|s|\ge B$. By the compactness of $[0,A]$, we can choose $B$ and $\epsilon_1$ to be independent of $\alpha$. Then by the compactness of $\{s:|s|\le B\}$ and $[0,A]$, we have $\Gamma_2(s)$ is $\epsilon_2$-away from $\alpha^2$ for all $|s|\le B$ and $\alpha\in[0,A]$. Then we define $\epsilon=\min\{\epsilon_1,\epsilon_2\}$.
\end{proof}
\medskip

\subsection*{Error estimate}
Note that in \cite[Equation (3.1)]{Dunster_Concial}, we have $A_0(\alpha,\zeta)\equiv1$ and
\[
    B_0(\alpha,\zeta)=(\zeta-\alpha^2)^{-1/2}\int_{\alpha^2}^\zeta (\xi-\alpha^2)^{-1/2}(-\psi(\alpha,\xi))\; d\xi.
\]
See \cite[p.\@ 426]{BoydDunster_EstimateLegendra} for $B_0$, where $\zeta-\alpha^2$ and $\xi-\alpha^2$ are taken in absolute value. For complex $\zeta$, we use analytic continuation.

By \cite[Equation (5.19)]{BoydDunster_EstimateLegendra} and the discussion following it, we have
\[
    \left|E_{m}^{(0)}(\tau\zeta^{1/2})^{-1}M_m^{(\pm1)}(\tau\zeta^{1/2})\right|\ll |J_m(\tau\zeta^{1/2})|.
\]

Therefore, to show Lemma \ref{lem:Conical by JBessel}, it is sufficient to show
\[
     \mathscr{V}_{\mathscr{P}^{(0)}}\left((\xi-\alpha^2)^{1/2}B_0(\xi)\right)=O_A(1).
\]
Recall the definition of our progressive path $\Gamma=\Gamma_1\cup\Gamma_2\in \mathscr{P}^{(0)}$.

Now we estimate
\[
    \mathscr{V}_{\Gamma}\left((\xi-\alpha^2)^{1/2}B_0(\xi)\right).
\]
Set $F(\xi):=(\xi-\alpha^2)^{1/2}B_0(\alpha,\xi)$. The variation $\mathscr{V}_\Gamma(f)$ of $f$ along the path $\Gamma$ satisfies  $\mathscr{V}_\Gamma(f)\le \int_\Gamma |f'|$. Hence, we have
\begin{align*}
    \mathscr{V}_{\Gamma}(F) &\le \int_{\Gamma} |\zeta-\alpha^2|^{-1/2}\,|\psi(\alpha,\zeta)|\,|d\zeta|\\
    &= \int_\Gamma \left|\frac{\zeta+4\alpha^2}{16(\zeta-\alpha^2)^{5/2}}-\frac{(z^2-1)(\zeta-\alpha^2)^{1/2}(\alpha^4-1+z^2(1+4\alpha^2))}{16\zeta(z^2-1-\alpha^2)^3}\right| |d\zeta|.
\end{align*}

We first estimate 
\[
    (I)_\alpha:=\int_{\Gamma_1} |\xi-\alpha^2|^{-1/2}\,|\psi(\alpha,\xi)|\,|d\xi|.
\]
In the $z$-plane, $\Gamma_1$ is a path from $1$ to $\sin(1/T)$. In the $z$-plane, $\Gamma_1$ is contained in $(0,1]$, and hence
\[
    |z^2-1-\alpha^2|\asymp |z^2-1|+\alpha^2
    \qquad
    \forall z\in \Gamma_1.
\]
Substituting the estimate for $\zeta$ \eqref{eqn:zeta by z near zeta=0} into \eqref{eqn:psi_conical to J-Bessel}, we obtain
\[
    |\psi(\alpha,\zeta)|=O_A(1)
    \qquad
    \forall\zeta\in \Gamma_1,
\]
and hence
\[
    (I)_\alpha\ll \int_{\Gamma_1} |\xi-\alpha^2|^{-1/2} |d\xi|=O_A(1).
\]

% (Alternative soft argument)
% Since $\zeta(z,\alpha)$ is jointly continuous on $z$ and $\alpha\in[0,A]$, so is $(\alpha,z)\mapsto \psi(\alpha,\zeta(z))$. Hence, we only need to consider the case $\alpha=0$. Then
% \[
%     (I)_0=\frac{1}{16}\int_{0}^{\zeta(\sin(1/T))}
%     \left|
%         \frac{1}{\xi^{3/2}}-\frac{1}{\xi^{1/2}(z(\xi)^2-1)}
%     \right||d\xi|.
% \]
% Here $z(\xi)$ is the point in the $z$-plane corresponding to $\xi$. Integrating \eqref{eqn:LG-transform_Conical to J} for $\alpha=0$ yields
% \[
%     \sqrt{\zeta}={\rm arcosh}(z).
% \] 
% The point $\zeta(0)$ corresponding to $z=0$ is a negative real number.

% Since $\frac1{\sinh^2 x}=\frac1{x^2}-\frac13+O(x^2)$ and $\zeta(0)<\zeta(\sin(1/T))$, we obtain
% \[
%     (I)_0 \le\frac{1}{16}\int_0^{\zeta(0)}
%     \left|
%         \frac{1}{\xi^{3/2}}-\frac{1}{\xi^{1/2}\sinh^2{\xi^{1/2}}}
%     \right||d\xi|
%     \le \frac{1}{16} \int_0^{\zeta(0)}
%     \frac{1}{3}|\xi|^{-1/2}+O(|\xi|^{1/2})\;|d\xi|
%     =O(1).
% \]

% Although $\zeta(0)$ depends on $\alpha$, we still have $\zeta(0)=O_A(1)$. When $\alpha>0$, the denominators in $(I)_\alpha$ are non-zero along $\xi\in \Gamma_1$.  Then, by continuity of $(\alpha,\zeta)\mapsto \psi(\alpha,\zeta)$ and the compactness of $[0,A]$, we can conclude that $(I)_\alpha=O_A(1)$ for all $\alpha\in[0,A]$.

Next, we consider
\[
    (II)_\alpha:=\int_{\Gamma_2} |\xi-\alpha^2|^{-1/2}\,|\psi(\alpha,\xi)|\,|d\xi|.
\]
Since $\Gamma_2$ is uniformly $\epsilon$-away from $\alpha^2$ (Lemma \ref{lem:prog.path_unif_away_turningPt}), we obtain
\[
    \left|\frac{\zeta+4\alpha^2}{16(\zeta-\alpha^2)^{5/2}}\right| \ll_A
    |\zeta+1|^{-3/2},
\]
and, by additionally using Lemma \ref{lem:prog.path_unif_away_turningPt}, we obtain 
\[
    \left|\frac{(z^2-1)(\zeta-\alpha^2)^{1/2}(\alpha^4-1+z^2(1+4\alpha^2))}{16\zeta(z^2-1-\alpha^2)^3}\right|
    \ll_A |1+\zeta|^{-1/2}|1+z|^{-2},
\]
and by \eqref{eqn:zeta and z large asymp}
\[
    |1+z|^{-2} \ll e^{-C|\zeta|^{1/2}}
\]
for some $C\gg_A1$.
Hence, we have
\[
    (II)_\alpha \ll_A \int_{\Gamma_2} |1+\xi|^{-3/2}+|1+\xi|^{-1/2}e^{-C|\xi|^{1/2}}\;|d\xi|.
\]
From \eqref{eqn:LG-transform_Conical to J}, we have
    \[
        \frac{(\zeta-\alpha^2)^{1/2}}{2\zeta}d\zeta=\frac{(z^2-1-\alpha^2)^{1/2}}{z^2-1}dz.
    \]
    Since $dz=\sinh u\,ds$ along $\Gamma_2(s)$, where $u(s)=s+i(\pi/2-1/T)$, we have
    \[
        \left|\frac{d\zeta}{ds}\right|
        =
        \left|\frac{2\zeta\,(\sinh^2u-\alpha^2)^{1/2}}{(\zeta-\alpha^2)^{1/2}\sinh u}\right|\ll_A|\zeta|^{1/2}.
    \]
    By \eqref{eqn:zeta and z large asymp}, we have $|\zeta|^{1/2}=\log|z|+O_A(1)=|s|+O_A(1)$ along $\Gamma_2$. Then,
    \[
        |d\zeta/ds|\ll_A 1+|s|.
    \]
    Hence,
    \[
        (II)_\alpha\ll_A \int_{\Gamma_2} (1+|s|)^{-2}+e^{-C|s|}\;|ds|=O(1).
    \]

\subsection{
\texorpdfstring{Estimate of $J$-Bessel functions by exponential functions}
{Estimate of J-Bessel functions by exponential functions}
}\label{subsec:J-Bessel by Exponential}

In this subsection, we use \cite[Theorem 2.1]{DunsterGilSegura_BesselbyAiry}, where the following Liouville transform is applied:
\begin{equation}\label{eqn:LG_Transform_J-Bessel by Exp}
    \xi=\frac{2}{3}\zeta^{3/2}=\int_1^z-\frac{\sqrt{1-t^2}}{t}dt=\log\left\{\frac{1+(1-z^2)^{1/2}}{z}\right\}-(1-z^2)^{1/2}.
\end{equation}
The branch is chosen such that $z\in(0,1)$ maps to $\xi(z)>0$
\cite[Section 5]{DunsterGilSegura_BesselbyAiry}, that is, $\sqrt{1-t^2},(1-z^2)^{1/2}>0$ for all $t,z\in(0,1)$.

\begin{lemma}\label{lem:J-Bessel by Exp}
    Suppose $m>0$. Let
    \[
        D_\epsilon:=\{z\in \mathbb{C}:\Re(z)>0~{\rm and}~d(z,[1,\infty))>\epsilon\}.
    \]
    For all $z\in D_\epsilon$, we have  
    \[
        J_m(mz)=\frac{1}{\sqrt{2\pi m}}\left(\frac{1}{1-z^2}\right)^{1/4}\;\exp(-m\xi)\;(1+O_\epsilon(m^{-1})),
    \]
    where $\xi$ is determined by \eqref{eqn:LG_Transform_J-Bessel by Exp}.
\end{lemma}
\begin{proof}
Suppose $m>0$. Consider the equation
\[
    \frac{d^2w}{dz^2}=\left(m^2f(z)+g(z)\right)w,
\]
where
\[
    f(z)=\frac{1-z^2}{z^2}
    \quad
    \text{and}
    \quad
    g(z)=-\frac{1}{4z^2}.
\]
By \cite[Theorem 2.1 with $n=1$]{DunsterGilSegura_BesselbyAiry}, for
\[
    W_0=\left(\frac{1-z^2}{\zeta z^2}\right)^{1/4} w_0(z),
\]
we obtain
\[
    W_0(m,\zeta)=\frac{1}{\zeta^{1/4}}e^{-m\xi}(1+\eta_{1,0}(m,z)),
\]
where
\[
    |\eta_{1,0}(m,z)|\le \frac{1}{m}\omega_{1,0}(m,z)\exp\left(\frac{1}{m}\omega_{1,0}(m,z)\right)
\]
and 
\[
    \omega_{1,0}(m,z)=2\int_{z^{(0)}}^z\left| \hat{F}_1(t)f^{1/2}(t)\;dt\right|,
\]
and, by \cite[eq.\@ (5.7)]{DunsterGilSegura_BesselbyAiry},
\[
    \hat{F}_1(z)=\frac{z^2(z^2+4)}{8(z^2-1)^3}.
\]
We take $z^{(0)}=0$. The integral is taken along a progressive path. By \cite[(5.19)]{DunsterGilSegura_BesselbyAiry}, we have
\[
    J_m(mz)=\frac{m^m}{e^m\Gamma(m+1)}\left(\frac{\zeta}{1-z^2}\right)^{1/4} W_0(m,\zeta).
\]
Since $\frac{m^m}{e^m\Gamma(m+1)}=\frac{1}{\sqrt{2\pi m}}(1+O(m^{-1}))$, it suffices to show $|\omega_{1,0}(m,z)|=O_\epsilon(1)$.

The function 
$\Re(\xi)$ is non-increasing along a progressive path. From \eqref{eqn:LG_Transform_J-Bessel by Exp}, we have
\[
    d\xi= \frac{-\sqrt{1-z^2}}{z} dz,
\]
where $\sqrt{1-z^2}>0$ for $z\in(0,1)$.
Then, we have $\Re(d\xi)<0$ along the paths
\[
\begin{aligned}
    &\Gamma_1: s\mapsto {\rm sgn}(\Im(z))\,is, &&0\le s\le |\Im z|\\
    &\Gamma_2: s\mapsto s+i\,\Im z, &&0\le s\le \Re(z),
\end{aligned}
\]
Therefore, for $z\in D_\epsilon$, the path $\Gamma:=\Gamma_1\cup \Gamma_2$ is a progressive path with $\Gamma\subset \overline{D_\epsilon}$. See Figure \ref{fig:Level_sets_JBessel_exp}.
\begin{figure}
    \centering
        \def\svgwidth{0.3\textwidth}
        %% Creator: Inkscape 1.3.2 (091e20e, 2023-11-25, custom), www.inkscape.org
%% PDF/EPS/PS + LaTeX output extension by Johan Engelen, 2010
%% Accompanies image file 'Level_sets_JBessel_exp.pdf' (pdf, eps, ps)
%%
%% To include the image in your LaTeX document, write
%%   \input{<filename>.pdf_tex}
%%  instead of
%%   \includegraphics{<filename>.pdf}
%% To scale the image, write
%%   \def\svgwidth{<desired width>}
%%   \input{<filename>.pdf_tex}
%%  instead of
%%   \includegraphics[width=<desired width>]{<filename>.pdf}
%%
%% Images with a different path to the parent latex file can
%% be accessed with the `import' package (which may need to be
%% installed) using
%%   \usepackage{import}
%% in the preamble, and then including the image with
%%   \import{<path to file>}{<filename>.pdf_tex}
%% Alternatively, one can specify
%%   \graphicspath{{<path to file>/}}
%% 
%% For more information, please see info/svg-inkscape on CTAN:
%%   http://tug.ctan.org/tex-archive/info/svg-inkscape
%%
\begingroup%
  \makeatletter%
  \providecommand\color[2][]{%
    \errmessage{(Inkscape) Color is used for the text in Inkscape, but the package 'color.sty' is not loaded}%
    \renewcommand\color[2][]{}%
  }%
  \providecommand\transparent[1]{%
    \errmessage{(Inkscape) Transparency is used (non-zero) for the text in Inkscape, but the package 'transparent.sty' is not loaded}%
    \renewcommand\transparent[1]{}%
  }%
  \providecommand\rotatebox[2]{#2}%
  \newcommand*\fsize{\dimexpr\f@size pt\relax}%
  \newcommand*\lineheight[1]{\fontsize{\fsize}{#1\fsize}\selectfont}%
  \ifx\svgwidth\undefined%
    \setlength{\unitlength}{312.28196512bp}%
    \ifx\svgscale\undefined%
      \relax%
    \else%
      \setlength{\unitlength}{\unitlength * \real{\svgscale}}%
    \fi%
  \else%
    \setlength{\unitlength}{\svgwidth}%
  \fi%
  \global\let\svgwidth\undefined%
  \global\let\svgscale\undefined%
  \makeatother%
  \begin{picture}(1,0.9768518)%
    \lineheight{1}%
    \setlength\tabcolsep{0pt}%
    \put(0,0){\includegraphics[width=\unitlength,page=1]{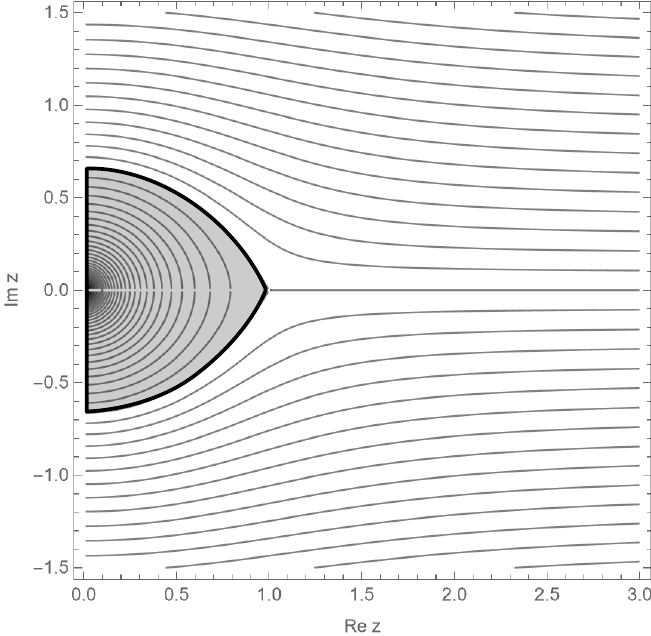}}%
    \put(0.18537367,0.51913906){\color[rgb]{0,0,0}\makebox(0,0)[lt]{\lineheight{1.25}\smash{\begin{tabular}[t]{l}$\mathscr{T}^{(0)}$\end{tabular}}}}%
  \end{picture}%
\endgroup%

        \caption{Level sets of $\Re(\xi)$ in the right half-plane of the $z$-plane. The tripod is the level set $\Re(\xi)=0$; moreover, $\Re(\xi)>0$ in $\mathscr{T}^{(0)}$ and $\Re(\xi)<0$ elsewhere.}
        \label{fig:Level_sets_JBessel_exp}
\end{figure}

Then, we obtain
\[
    \omega_{1,0}(m,z)=2\int_\Gamma\left| \frac{t(t^2+4)}{8(t^2-1)^{5/2}}\;dt \right|=O_\epsilon(1),
\]
and thus
\[
    |\eta_{1,0}(m,z)|=O_
    \epsilon(m^{-1}).\qedhere
\]
\end{proof}

\subsection{
\texorpdfstring{Expansions of Conical functions by $K$-Bessel functions, $\beta:=1/\alpha$ bounded}
{Expansions of Conical functions by K-Bessel functions, beta bounded}
}\label{subsec:Conical by K-Bessel}

Suppose $m>0$. Let $\beta:=1/\alpha=\tau/m$.

\begin{lemma}\label{lem:Conical by KBessel}
Suppose that $m>0$ and $\beta\in(0,B]$. Let $z=\cosh(x+i(\pi/2-1/T))$. For all $x\in \mathbb{R}$ and $T\gg_B 1$, we have
\[
    |P^{-m}_{-1/2+i\tau}(z)|= \left(\frac{2}{\pi}\right)^{1/2}\left|\frac{1}{\Gamma(1/2+m+i\tau)} \left(\frac{\eta-\beta^2}{1+\beta^2-z^2\beta^2}\right)^{1/4}K_{i\tau}(m \eta^{1/2})\right|\big(1+O_B(m^{-1})\big).
\]    
\end{lemma}
\begin{proof}
By the following preliminary Liouville transform
\[
    \xi=\frac{1}{z^2-1} \quad{\rm and}\quad \hat{w}:=\left(\frac{z^{1/2}}{z^2-1}\right)w,
\]
the equation \eqref{eqn:Conical_w} is transformed into
\begin{equation}\label{eqn:Conical_hatw}
    \frac{d^2\hat{w}}{d\xi^2}=\left(m^2\hat{f}(\beta,\xi)+\hat{g}(\xi)\right)\hat{w},
\end{equation}
where $\hat{f}(\beta,\xi)=\frac{\xi-\beta^2}{4\xi^2(1+\xi)}$ and $\hat{g}(\xi)=-\frac{4\xi^2+5\xi+4}{16\xi^2(1+\xi)^2}$. See \cite[Section 4]{Dunster_Concial}. We use the following Liouville transformation:
\begin{equation}\label{eqn:eta-xi transform}
    \int_{\beta^2}^\eta \frac{(t-\beta^2)^{1/2}}{2t}\,dt =\int_{\beta^2}^\xi(\hat{f}(\beta,t))^{1/2}\,dt=\int_{\beta^2}^\xi \frac{1}{2} \left(\frac{t-\beta^2}{t^2 (1+t)}\right)^{1/2}\,dt.
\end{equation}
Then, for
\[
    \hat{W}(\beta,\xi)=\left(\frac{\eta}{\xi}\right)^{1/2}\left(\frac{\xi-\beta^2}{(1+\xi)(\eta-\beta^2)}\right)^{1/4}\hat{w}(\beta,\xi),
\]
the equation \eqref{eqn:Conical_hatw} is transformed into 
\begin{equation}\label{eqn:Conical_hatW}
    \frac{d^2\hat{W}}{d\eta^2}=\left\{m^2\left(\frac{1}{4\eta}-\frac{\beta^2}{4\eta^2}\right)-\frac{1}{4\eta^2}+\frac{\hat{\psi}(\beta,\eta)}{\eta}\right\}\hat{W},
\end{equation}
where
\begin{equation}\label{eqn:hatpsi}
    \hat{\psi}(\beta,\eta)  =\frac{\eta+4\beta^2}{16(\eta-\beta^2)^2}-\frac{(\eta-\beta^2)\xi(\xi+4\beta^2\xi+\beta^4+4\beta^2)}{16\eta(\xi-\beta^2)^3}.
\end{equation}

We consider the following four coordinates:
\begin{center}
    \begin{tikzcd}
    \{u:\Im(u)\in(-\frac{\pi}{2},\frac{\pi}{2})\} \arrow[rr,"z=\cosh(u)"]
    && \{z:\Re(z)>0\} \arrow[r, "\xi=\frac{1}{z^2-1}"]
    & \{\xi:{\rm Figure}~\ref{fig:z_xi_eta_planes}\} \arrow[r]
    & \{\eta:{\rm Figure}~\ref{fig:z_xi_eta_planes}\}.
\end{tikzcd}
    
\end{center}

\begin{figure}
	\centering
    \begin{subfigure}[b]{0.3\textwidth}
        \centering
        \def\svgwidth{\textwidth}
        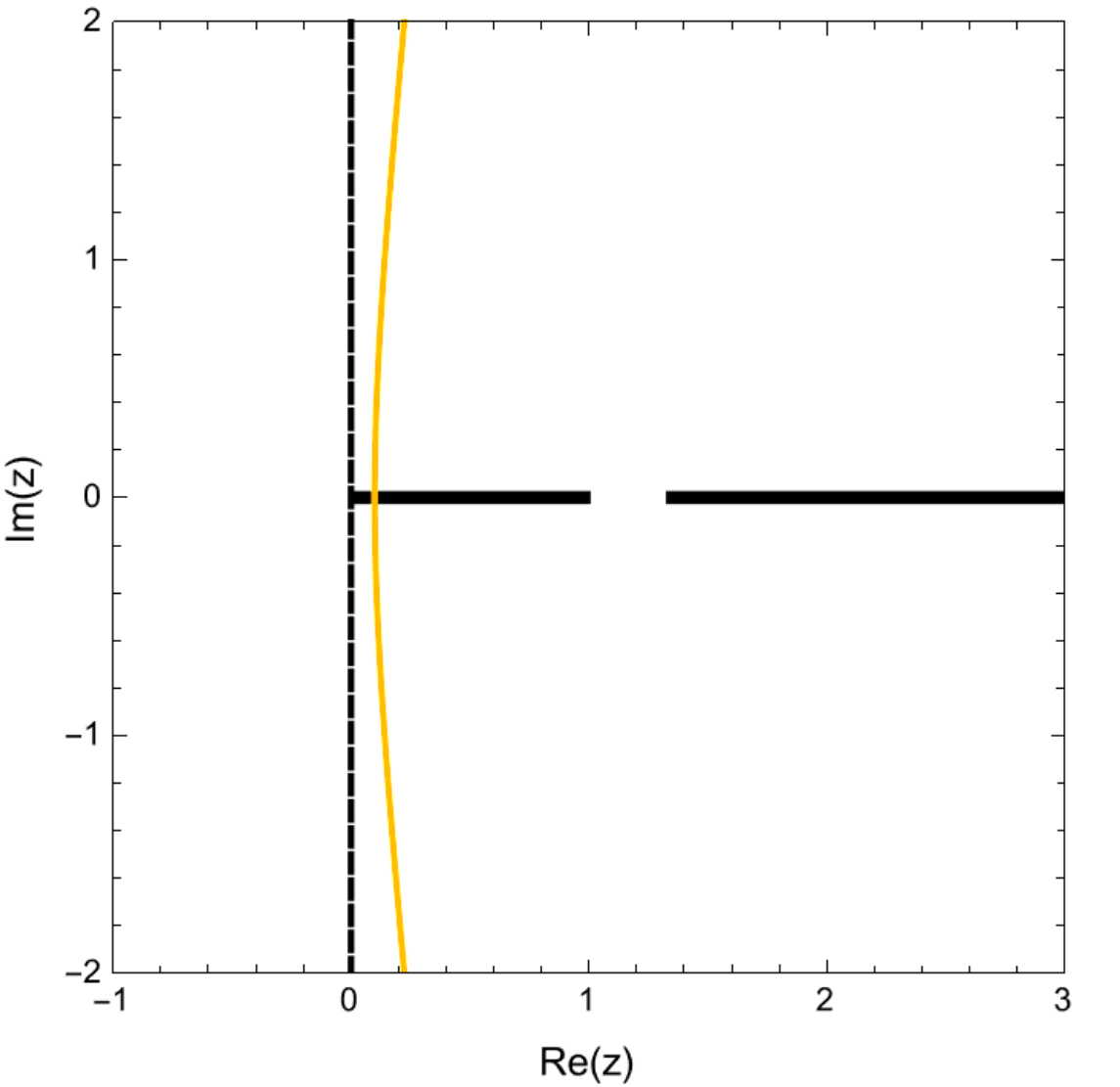
        \caption{$z$-plane}
    \end{subfigure}
    \hspace{10pt}
        \begin{subfigure}[b]{0.3\textwidth}
        \centering 
        \def\svgwidth{\textwidth}
        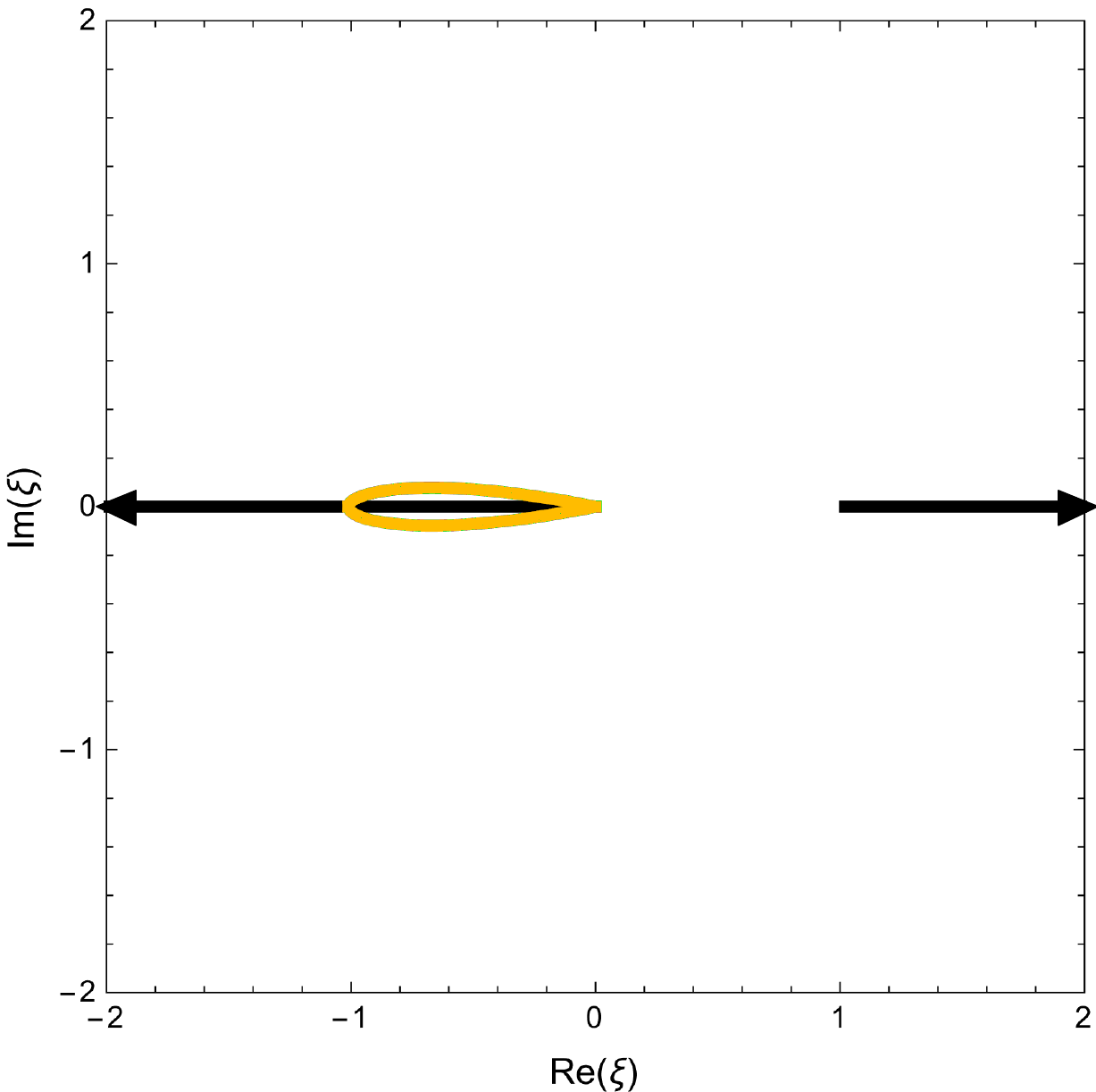
        \caption{$\xi$-plane}
    \end{subfigure}
    \hspace{10pt}
    \begin{subfigure}[b]{0.3\textwidth}
        \centering
        \def\svgwidth{\textwidth}
        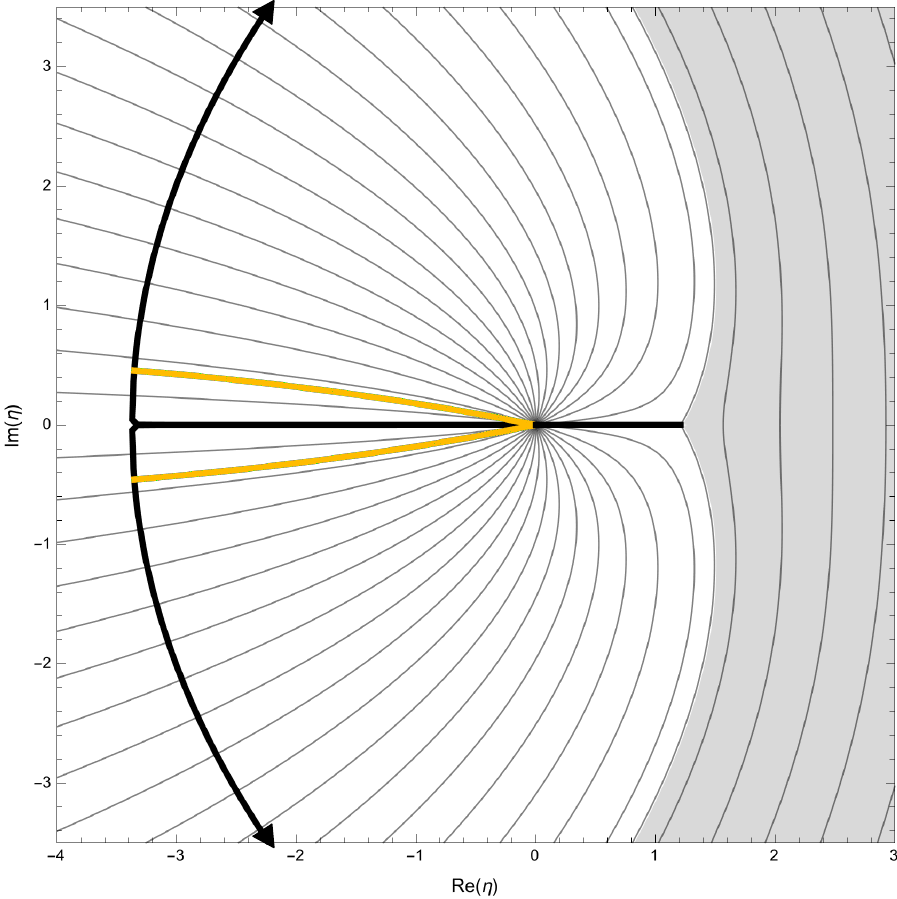
        \caption{$\eta$-plane}
    \end{subfigure}
    \caption{Coordinate changes among the $z$-, $\xi$-, and $\eta$-planes with $\beta\approx 1$. The thick curves are the image of $\Im(r)=\pi/2-0.1$, where $z=\cosh(r)$. Level curves of $\Re(\Phi^{(0)}_\tau(m\eta^{1/2}))$ are shown in the $\eta$-plane. The shaded region in the $\eta$-plane denotes the domain $\mathscr{T}^{(3)}$ where progressive paths start.}
    \label{fig:z_xi_eta_planes}
\end{figure}

By \cite[Equations (4.14),(4.15)]{Dunster_Concial} with $n=0$ and \cite[Theorem 3a]{Dunster_Prep}, we have
\[
    P^{-m}_{-\frac{1}{2}+i\tau}(z)=\hat{c}^{(1)}_{1}\left(\frac{\eta-\beta^2}{1+\beta^2-z^2\beta^2}\right)^{1/4} \left[K_{i\tau}(m\eta^{1/2})+\eta^{-1/2}\hat{\epsilon}^{(1)}_1(m,\beta,\eta)\right],
\]
where
\[
    \hat{c}_1^{(1)}=\sqrt{\frac{2}{\pi}}\frac{m^{m-1/2}e^{m\beta\arctan(\beta)}}{(m^2+\tau^2)^{m/2}\Gamma(m)},
\]
and 
\begin{multline*}
    |\hat{\epsilon}^{(1)}_1(m,\beta,\eta)|\le  \left|\mu \frac{\eta^{1/2}}{m}K_{i\tau}(m\eta^{1/2})
    \mathscr{V}_{\mathscr{P}^{(0)}}\left((\eta-\beta^2)^{1/2}B_0(\beta,\eta)\right)
    \exp\left(\frac{\kappa}{m}\mathscr{V}_{\mathscr{P}^{(0)}}\left((\eta-\beta^2)^{1/2}B_0(\beta,\eta)\right)\right)\right|.
\end{multline*}
By using Lemma \ref{lem:Gamma(0.5+m+itau)}, one can show 
\[
    \hat{c}_1^{(1)}=
    \sqrt{\frac{2}{\pi}}\,
    \frac1{|\Gamma(m+\frac12+i\tau)|}
    \left(1+O_B\!\left(\frac1m\right)\right).
\]
The superscript index $(j)$ is (1) in \cite{Dunster_Concial} and in this article, but (3) in \cite{Dunster_Prep}. Also, in \cite[Theorem 3a]{Dunster_Prep}, Dunster used $\kappa^{(j)}, \mu^{(j)}$ for $j=1,2,3$, which are real numbers, but we use $\kappa:=\max_{j=1,2,3} \kappa^{(j)}$ and $\mu:=\max_{j=1,2,3} \mu^{(j)}$. 

Note that in \cite[Equations (2.13--2.14)]{Dunster_Prep} and the accompanying discussion, we have $A_0(\beta,\eta)\equiv1$ and
\[
    B_0(\beta,\eta)=(\eta-\beta^2)^{-1/2}\int_{\beta^2}^\eta (t-\beta^2)^{-1/2}\hat{\psi}(\beta,t)\; dt.
\]
In \cite[eq. (2.13)]{Dunster_Prep} for $B_0$, the terms $\eta-\beta^2$ and $t-\beta^2$ are taken in absolute value. For complex $\eta$, we use analytic continuation.

It suffices to show that
\begin{align}\label{eqn:Variation_K-Bessel}
    \begin{split}
        &\mathscr{V}_{\mathscr{P}^{(0)}}\left(\int_{\beta^2}^\eta (t-\beta^2)^{-1/2}\hat\psi(\beta,t)\; d t\right)\\
        \le&\int_\Gamma \left|(t-\beta^2)^{-1/2}\hat\psi(\beta,t)\right| |d t|\\
        =& \int_\Gamma  \left|\frac{\eta+4\beta^2}{16(\eta-\beta^2)^{5/2}}-\frac{(\eta-\beta^2)\xi(\xi+4\beta^2\xi+\beta^4+4\beta^2)}{16\eta(\eta-\beta^2)^{1/2}(\xi-\beta^2)^3}\right||d\eta|\\
        =&O_B(1),
    \end{split}
\end{align}
where $\Gamma\in \mathscr{P}^{(0)}$ is a progressive path.

\subsection*{Progressive path}
In this case, the initial point of a progressive path is contained in $\mathscr{T}_\tau^{(3)}$, possibly at infinity. The terminal point corresponds to $u_*=x+i(\pi/2-1/T)$ in the $u$-plane. Let $\xi_*$ and $\eta_*$ denote these points in the $\xi$- and $\eta$-planes, respectively.

Along progressive paths, the function $\Re(\Phi_\tau^{(0)})$ is non-increasing, where
\begin{equation}\label{eqn:Phi_tau(0)}
    \Phi^{(0)}_\tau (m\eta^{1/2})=\int_{X_{\hat{\tau}}}^{m\eta^{1/2}} \frac{(t^2-(X_{\hat{\tau}})^2)^{1/2}}{t}dt,
\end{equation}
and $X_\tau$ denotes the largest positive real zero of $K_{i\tau}(x)-\tilde{I}_{i\tau}(x)$; see \cite[Eq.\@ (2.8), (2.20)]{Dunster_Prep}. We have $X_\tau=\tau+c(\tau/2)^{1/3}+O(\tau^{-1/3})$ \cite[Eq.\@ (2.21)]{Dunster_Prep}, where $c$ is the largest negative root of ${\rm Ai}(x)={\rm Bi}(x)$. Define $\hat{\tau}>0$ by $\tau=\hat\tau+c(\hat\tau/2)^{1/3}$. Then $X_{\hat\tau}=\tau+O(\tau^{-1/3})$. 

The condition $\Re(\Phi_\tau^{(0)}(z))>0$ for $z\in \mathscr{S}_\tau^{(3)}$ in \cite[p.~1606]{Dunster_Prep} implies that the square root of a positive real number is taken to be a positive real. Hence, we choose the branch so that $ \arg\sqrt{\cdot}\in\left(-\tfrac{\pi}{2},\tfrac{\pi}{2}\right)$.

We define a path $\Gamma=\Gamma_1\cup \Gamma_2\cup \Gamma_3$ in the $\eta$-plane satisfying the following: (i) $\Gamma_1$ is a horizontal line segment starting from a boundary point of $\mathscr{T}_\tau^{(3)}$, (ii) $\Gamma_2$ is a vertical line segment ending at the level set of $\Re(\Phi_\tau^{(0)}(m\eta^{1/2}))$ containing $\eta_*$, and (iii) $\Gamma_3$ is contained in the level curve containing $\eta_*$ such that $\Gamma_3$ ends at $\eta_*$. We can choose $\Gamma_1$ and $\Gamma_2$ so that $|\eta|\gg1$ and $|\eta-\beta^2|\gg|\eta|$ for all $\eta\in \Gamma_1\cup \Gamma_2$. Moreover, the entire path $\Gamma$ is contained in a ball of radius $O_B(1)$, and its Euclidean length is also $O_B(1)$. See Figure \ref{fig:z_xi_eta_planes}, where $\eta_*$ lies on the thick curve. Then one can show that $\Gamma$ is a progressive path. 
\medskip

A similar argument to that in Section \ref{subsec:Conical by J-Bessel} may not work because the progressive path cannot be uniformly away from the turning point $\beta^2$ in the $\eta$ plane, as $\beta\to 0$.

From \eqref{eqn:eta-xi transform}, for all $\eta$ with $|\eta|\ll_B 1$, we obtain
\begin{equation}\label{eqn:xi by eta_small}
    \xi(\beta,\eta)=\frac{\exp(R(\beta))}{1+\beta^2}\eta+O_B(\eta^2),
\end{equation}
where 
\[
    R(\beta)=2-\frac1\beta \arccos{\left(\frac{1-\beta^2}{1+\beta^2}\right)}= \frac23\beta^2+O_B(\beta^4);
\]
see \cite[(4.10),(4.11)]{Dunster_Concial}. Moreover, from \ref{eqn:eta-xi transform}, we have that $\xi(\beta,\eta)$ is jointly continuous for $\beta\in[0,B]$ and $\eta\neq0,\beta^2$. 

When $\eta\in \Gamma_1\cup \Gamma_2$, we have
\[
    \left|(\eta-\beta^2)^{-1/2}\hat{\psi}(\beta,\eta)\right|
    \ll_B |\eta|^{-3/2}+|\eta|^{-1/2}|\xi|^{-1},
\]
and thus the integral in \eqref{eqn:Variation_K-Bessel} along $\Gamma_1\cup \Gamma_2$ is uniformly $O_B(1)$.

Let us consider the integral along $\Gamma_3$. Since $\Gamma_3$ is contained in the half-planes $\Re(\xi)<0$ and $\Re(\eta)<0$, we have $|\eta-\beta^2|\asymp_B |\eta|+\beta^2$. By \eqref{eqn:xi by eta_small}, for all $\eta$ with $|\eta|<R_B$ for some $R_B\gg_B 1$, we have
\begin{equation}\label{eqn:xi(eta) small asympototic}
    \xi=\eta+O_B(\max\{\beta^2|\eta|,|\eta|^2\}).
\end{equation}
Then, by substituting \eqref{eqn:xi(eta) small asympototic} into \eqref{eqn:hatpsi}, we obtain $|\hat{\psi}(\beta,\eta)|=O_B(1)$ for all $\eta\in\Gamma_3\cap B(0,R_B)$.
On the other hand, since $\Gamma_3$ is contained in the left half-plane $\Re(\eta)<0$, we have $|\eta-\beta^2|,|\xi-\beta^2|\gg_B 1$ for all $\eta\in \Gamma_3\setminus B(0,R_B)$. Hence, by the joint continuity of $\xi(\beta,\eta)$ and the uniform boundedness of $\beta$ and $\Gamma_3$, we can extend the estimate to 
\[
    |\hat{\psi}(\beta,\eta)|=O_B(1)
    \qquad
    \forall
    \eta\in \Gamma_3.
\]
Let $a:=X_{\hat\tau}/m>0$ and write $\eta=\rho e^{i\theta}$ on
$\Gamma_3$. Since $\Gamma_3$ is a level curve of
$\Re(\Phi_\tau^{(0)}(m\eta^{1/2}))$, it follows from \eqref{eqn:Phi_tau(0)} that
\[
    0=
    \Re\left(
       \sqrt{\eta-a^2}
       \left(\frac{d\rho}{\rho}+i\,d\theta\right)
    \right)
    =
    \Re\left(\sqrt{\eta-a^2}\right)\frac{d\rho}{\rho}-\Im\left(\sqrt{\eta-a^2}\right)d\theta.
\]
Recall that $\Gamma_3$ is contained in the left half-plane $\Re(\eta)<0$ minus the real line. Since $a>0$, for all $\eta$ with $\Re(\eta)<0$ and $\eta\notin \mathbb{R}$, we have
\[
    0<\left|\Re\left(\sqrt{\eta-a^2}\right)\right|
    \leq\left|\Im\left(\sqrt{\eta-a^2}\right)\right|.
\]
It follows that $d\rho\neq0$, and thus $\rho$ is a monotone parameter on each level curve, and
\[
    \rho\left|\frac{d\theta}{d\rho}\right|\leq1,
    \qquad
    |d\eta|=|d\rho|\left|e^{i\theta}+ i\rho e^{i\theta}\frac{d\theta}{d\rho}\right|\le  \sqrt2\,|d\rho|.
\]
Hence, we obtain
\[
    \int_{\Gamma_3\cap B(0,R_B)} |\eta-\beta^2|^{-1/2}|\hat\psi(\beta,\eta)|\,|d\eta|
    \,\ll_B\,
    \int_0^{R_B}\rho^{-1/2}\,d\rho
    \,=\,
    O_B(1).\qedhere
\]
\end{proof}

\subsection{Expansion of K-Bessel functions by exponential functions}\label{subsec:K-Bessel by exponential} See \cite{Dunster_K_Bessel_by_Exp}. We estimate the term $K_{i\tau}(m\eta^{1/2})$ in Lemma~\ref{lem:Conical by KBessel}. Recall that for $z=\cosh(x+i(\pi/2-1/T))$, the corresponding $\eta$ is not a real number and lies in the left half-plane $\Re(\eta)<0$. Hence, $|\arg(\eta^{1/2})|\in (\pi/4,\pi/2)$.

We show the following lemma.

\begin{lemma}\label{KBesselExp}
Suppose $\tau>0$. For $\delta\in(0,\pi/2)$, define a set
\[
    D_\delta:=\{z\in \mathbb{C}:|\arg(z)|\in(\pi/2-\delta,\pi/2)\}.
\]
Then, for all $z\in D_\delta$, we have
\begin{equation}
    K_{i\tau}(\tau z)=\left(\frac{\pi}{2\tau}\right)^{1/2}\frac{1}{(z^2-1)^{1/4}}\exp\left(-\tau\xi-\frac{\pi \tau}{2}\right)\left(1+O_\delta(\tau^{-1})\right).
\end{equation}
\end{lemma}
\begin{proof}
The function
\[
    w(z)=z^{1/2}K_{i\tau}(\tau z)
\]
solves the equation
\[
    \frac{d^2w}{dz^2}=\left\{\tau^2\frac{z^2-1}{z^2}-\frac{1}{4z^2}\right\}w.
\]
We use the following Liouville transformation
\[
    \xi=\int_1^z\frac{(t^2-1)^{1/2}}{t}dt=(z^2-1)^{1/2}-\arcsec\,z.
\]
The branches are chosen so that $\sqrt{z^2-1},\sqrt[4]{z^2-1},$ and $\xi(z)$ are positive real for $z>1$.
By using \cite[Theorem 2.1 with $n=1$\footnote{\cite[Theorem 2.1]{Dunster_K_Bessel_by_Exp} is stated for $n\ge2$, but we can still use the same formula for $n=1$.}]{Dunster_K_Bessel_by_Exp},
we obtain
\[
    K_{i\tau}(\tau z)=\left(\frac{\pi}{2\tau}\right)^{1/2}\frac{1}{(z^2-1)^{1/4}}\exp\left(-\tau\xi-\frac{\pi \tau}{2}\right)\left(1+\eta_{1,0}(\tau,z)\right),
\]
where
\[
    |\eta_{1,0}(\tau,z)|\le \frac{1}{\tau}\Phi_{1,0}(\tau,z)\exp\left(\frac{1}{\tau}\Phi_{1,0}(\tau,z)\right),
\]
and
\[
    \Phi_{1,0}(\tau,z) = 2 \int_{\alpha_0}^{z} \left|t^{-1}(t^2-1)^{1/2}F_1(t)\,dt \right|,
\]
and 
\[
    F_1(z)=-\frac{1}{8}\frac{z^2(z^2+4)}{(z^2-1)^3}.
\]

The integral for $\Phi_{1,0}$ is taken along a progressive path. In this context, $\Re(\xi)$ is non-increasing along progressive paths. Note that
\[
    \Re(d\xi)=\Re\left(\frac{\sqrt{z^2-1}}{z}dz\right).
\]
In \cite{Dunster_K_Bessel_by_Exp}, $\alpha_0$ is taken to be $+\infty+i0$, whereas we will choose different $\alpha_0$ depending on $w\in D_\delta$.

Fix $w=re^{i\theta}\in D_\delta$. Since $K_{i\tau}(\tau \overline{z})=\overline{K_{i\tau}(\tau z)}$ for all $z\in D_\delta$, without loss of generality, we assume $\Im(w)\ge 0$.  Let $A=\max\{1,\Im(w)\}$ and $\alpha_0=+\infty+iA$. We define a path $\Gamma=\Gamma_1\cup \Gamma_2$ such that $\Gamma_1$ is a horizontal path from $\alpha_0=+\infty +iA$ to the point $w'$ with $\Im(w')=A$ and $\arg(w')=\arg(w)$. If $A=\Im(w)$, then $\Gamma_2=\emptyset$. If $A=1>\Im(w)$, then $\Gamma_2$ lies on the radial ray $\arg(z)=\arg(w)$ and connects $w'$ and $w$. The function $\Re(\xi)$ is non-increasing along $\Gamma$, and thus $\Gamma$ is a progressive path. Since $w\in D_\delta$, the path $\Gamma$ is uniformly $C(\delta)$-away from the point $1$. See Figure \ref{fig:Level_sets_KBessel_exp}.
\begin{figure}
    \centering
        \def\svgwidth{0.3\textwidth}
        %% Creator: Inkscape 1.3.2 (091e20e, 2023-11-25, custom), www.inkscape.org
%% PDF/EPS/PS + LaTeX output extension by Johan Engelen, 2010
%% Accompanies image file 'Level_sets_KBessel_exp.pdf' (pdf, eps, ps)
%%
%% To include the image in your LaTeX document, write
%%   \input{<filename>.pdf_tex}
%%  instead of
%%   \includegraphics{<filename>.pdf}
%% To scale the image, write
%%   \def\svgwidth{<desired width>}
%%   \input{<filename>.pdf_tex}
%%  instead of
%%   \includegraphics[width=<desired width>]{<filename>.pdf}
%%
%% Images with a different path to the parent latex file can
%% be accessed with the `import' package (which may need to be
%% installed) using
%%   \usepackage{import}
%% in the preamble, and then including the image with
%%   \import{<path to file>}{<filename>.pdf_tex}
%% Alternatively, one can specify
%%   \graphicspath{{<path to file>/}}
%% 
%% For more information, please see info/svg-inkscape on CTAN:
%%   http://tug.ctan.org/tex-archive/info/svg-inkscape
%%
\begingroup%
  \makeatletter%
  \providecommand\color[2][]{%
    \errmessage{(Inkscape) Color is used for the text in Inkscape, but the package 'color.sty' is not loaded}%
    \renewcommand\color[2][]{}%
  }%
  \providecommand\transparent[1]{%
    \errmessage{(Inkscape) Transparency is used (non-zero) for the text in Inkscape, but the package 'transparent.sty' is not loaded}%
    \renewcommand\transparent[1]{}%
  }%
  \providecommand\rotatebox[2]{#2}%
  \newcommand*\fsize{\dimexpr\f@size pt\relax}%
  \newcommand*\lineheight[1]{\fontsize{\fsize}{#1\fsize}\selectfont}%
  \ifx\svgwidth\undefined%
    \setlength{\unitlength}{311.57144381bp}%
    \ifx\svgscale\undefined%
      \relax%
    \else%
      \setlength{\unitlength}{\unitlength * \real{\svgscale}}%
    \fi%
  \else%
    \setlength{\unitlength}{\svgwidth}%
  \fi%
  \global\let\svgwidth\undefined%
  \global\let\svgscale\undefined%
  \makeatother%
  \begin{picture}(1,0.97685183)%
    \lineheight{1}%
    \setlength\tabcolsep{0pt}%
    \put(0,0){\includegraphics[width=\unitlength,page=1]{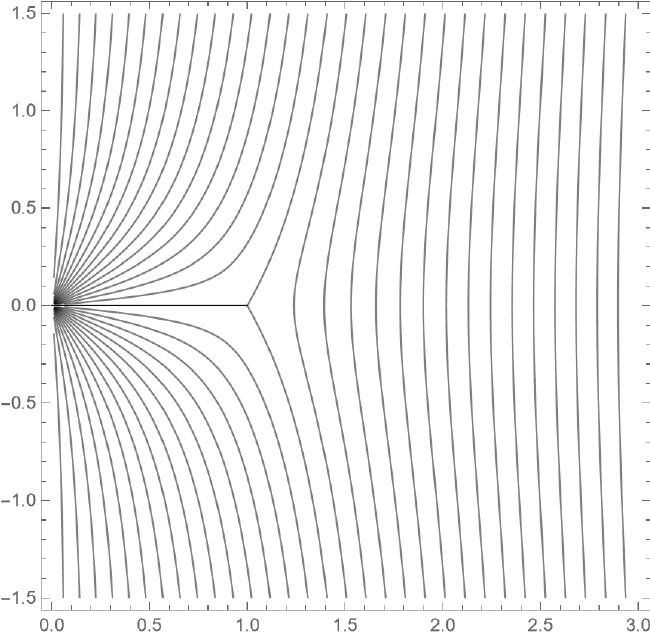}}%
    \put(0.61101982,0.49652436){\color[rgb]{0,0,0}\makebox(0,0)[lt]{\lineheight{1.25}\smash{\begin{tabular}[t]{l}$\mathscr{T}^{(0)}$\end{tabular}}}}%
    \put(0,0){\includegraphics[width=\unitlength,page=2]{Level_sets_KBessel_exp.pdf}}%
  \end{picture}%
\endgroup%

        \caption{Level sets of $\Re(\xi)$ in the right half-plane in the $z$-plane. $\Re(\xi)=0$ on the tripod level set, $\Re(\xi)>0$ on $\mathscr{T}^{(0)}$, and $\Re(\xi)<0$ elsewhere.}\label{fig:Level_sets_KBessel_exp}
\end{figure}

By estimating the integral along this progressive path, we obtain
\[
    \Phi_{1,0}(\tau,z)=O_\delta(1).
\]
Then, the lemma follows.
\end{proof}

\end{document}